\documentclass{amsart}

\usepackage{amsmath, amscd}
\usepackage{amssymb, dsfont}
\usepackage{mathrsfs}
\usepackage{mathtools}
\usepackage{graphicx}
\usepackage[english]{babel}
\usepackage{hyperref}

\newcommand{\C}[0]{\mathds{C}}
\newcommand{\R}[0]{\mathds{R}}

\newcommand{\Z}[0]{\mathds{Z}}
\newcommand{\N}[0]{\mathds{N}}
\newcommand{\T}[0]{\mathds{T}}
\newcommand{\pv}[0]{\mathrm{p.v.}}
\newcommand{\transpose}[1]{{}^{t}#1}
\DeclareMathOperator{\diag}{\mathrm{diag}}
\DeclareMathOperator{\sgn}{\mathrm{sgn}}

\DeclareMathOperator{\imag}{\mathrm{Im}}

\DeclareMathOperator{\supp}{\mathrm{supp}}
\newcommand{\Schwartz}[0]{\mathcal{S}}
\newcommand{\species}[0]{\mathsf{s}} 
\newcommand{\ph}[0]{\mathrm{ph}}
\newcommand{\eps}[0]{\varepsilon}
\newcommand{\ff}[0]{\mathsf{f}}
\newcommand{\Ef}[0]{\mathsf{E}}
\newcommand{\Bf}[0]{\mathsf{B}}

\theoremstyle{plain}
\newtheorem{theorem}{Theorem}[section]
\newtheorem{corollary}[theorem]{Corollary}
\newtheorem{lemma}[theorem]{Lemma}
\newtheorem{proposition}[theorem]{Proposition}
\theoremstyle{definition}

\newtheorem{definition}{Definition}
\theoremstyle{remark}
\newtheorem{remark}{Remark}

\begin{document}

% Front matter
% =========================================================================
%

\title[tempered-in-time response of a plasma]{
  Uniqueness of the Cauchy datum for the tempered-in-time response and
  conductivity operator of a plasma
}

\author[O. Lafitte and O. Maj]{Olivier Lafitte${}^{1,2}$ \and Omar Maj${}^{3,4}$}

\address{${}^1$ 
  Universit\'e Sorbonne Paris Nord, LAGA, CNRS, UMR 7539, F-93430, 
  Villetaneuse, France.}
\address{${}^2$
  Centre de Recherches Math\'ematiques, IRL CRM-CNRS 3457,
  Universit\'e de Montr\'eal, Pavillon Andr\'e-Aisenstadt,
  2920 Chemin de la tour, Montr\'eal (Qu\'ebec) H3T 1J4.
} 
\address{${}^3$ 
  Max Planck Institute for Plasma Physics,
  Boltzmannstr.~2, D~85748 Garching, Germany.
}
\address{${}^{4}$ Technische Universit\"at M\"unchen, Zentrum Mathematik,
  Boltzmannstra\ss{}e 3, 85747 Garching, Germany.
}

\begin{abstract}
  We study the linear Vlasov equation with a given electric field
  $E \in \Schwartz$, where $\Schwartz$ is the space of Schwartz functions.
  The associated damped partial differential equation has a unique tempered
  solution, which fixes the needed Cauchy datum. This tempered solution then
  converges to the causal solution of the linear Vlasov equation when the
  damping parameter goes to zero. This result allows us to define the plasma
  conductivity operator $\sigma$, which gives the current density
  $j = \sigma (E)$ induced by the electric field $E$. We prove that $\sigma$ is
  continuous from $\Schwartz$ to its dual $\Schwartz^\prime$. We can treat
  rigorously the case of uniform non-magnetized non-relativistic plasma
  (linear Landau damping) and the case of uniform magnetized relativistic plasma 
  (cyclotron damping). In both cases, we demonstrate that the main part of the
  conductivity operator is a pseudo-differential operator and we give its
  expression rigorously. This matches the formal results widely used in the
  theoretical physics community. 
\end{abstract}

\maketitle

\tableofcontents

% New section
% =========================================================================
%
\section{Introduction}
\label{sec:intro}

A plasma is a collection of a sufficiently large number of electrically charged
particles of various species (electrons, protons, and ions of different 
elements), subject to electromagnetic fields. In kinetic theory, the
configuration of a plasma is specified by a family of functions  
$\ff_\species : \R \times \R^3 \times \R^3 \to \R_+$, labeled by the index of
particle species $\species$ and defined so that $\ff_\species (t, x, p)$
gives the density of particles of the species $\species$ at the time $t$,
position $x$ and relativistic momentum $p$.

The equations governing the evolution of the distribution functions
$\{\ff_\species\}_\species$, together with the electric field
$\Ef : \R \times \R^3 \to \R^3$ and the magnetic field
$\Bf : \R \times \R^3 \to \R^3$,
are given by the relativistic Vlasov-Maxwell-Landau system, which writes 
\begin{equation}
  \label{eq:RVMS}
  \left\{
  \begin{aligned}
    &\partial_t \ff_\species + v_\species \cdot \nabla_x \ff_\species + q_\species 
    \big(\Ef + v_\species \times \Bf / c \big) \cdot 
    \nabla_p \ff_\species = C_\species(\{\ff_{\species'}\}_{\species'}),  \\
    &\partial_t \Ef - c\nabla \times \Bf = -4 \pi  
    \sum_\species q_\species \int_{\R^3} v_\species(p) \ff_\species (t,x,p)dp, && \\
    &\partial_t \Bf = -c \nabla \times \Ef, && \\
    &\nabla \cdot \Bf = 0, && \\
    &\nabla \cdot \Ef = 4\pi  
    \sum_\species q_\species \int_{\R^3} \ff_\species (t,x,p)dp, && 
  \end{aligned} \right.
\end{equation}
where the relativistic velocity $v_\species$ is defined by
$v_\species(p) = p / [m_\species \gamma_\species(p)]$, with
$\gamma_\species (p) = \big( 1 + p^2 / m_\species^2 c^2 \big)^{1/2}$,
$C_\species$ is the relativistic Landau collision operator
\cite{Beliaev1956, Braams1989} and depends on $\{\ff_{\species'}\}_{\species'}$,
$c$ is the speed of light, $m_\species$ is the mass of particles of the species
$\species$, and $q_\species$ is their electric charge (c.g.s. units are used
throughout the paper). 

However, a variety of reduced models are also in use for modeling plasmas and
gases in special cases. For instance, at moderate energies the non-relativistic
version is used, which follows from~(\ref{eq:RVMS}) by setting
$\gamma_\species(p) = 1$, so that $v_\species(p) = p/m_\species$, and by replacing
$C_\species$ with the non-relativistic Landau collision operator. When the
collision operator $C_\species$ can be neglected, one recovers the
Vlasov-Maxwell system both in the relativistic and non-relativistic
versions. The Vlasov-Maxwell system can be further reduced, when all effects of 
the magnetic field can be neglected; then Maxwell's equations are replaced by
the Poisson equation for the electrostatic potential $\phi$, $\Ef =-\nabla \phi$,
and the Vlasov-Maxwell system reduces to the Vlasov-Poisson system.  
For electrically neutral particles (a gas), $q_\species=0$, the electromagnetic
part of the system can be dropped and the collision operator $C_\species$ is
given by the Boltzmann operator. This gives the Boltzmann equation
\cite{Cercignani1975}. 
One can also replace the collision operator by simpler models,
such as the BGK (Bhatnagar, Gross and Krook \cite{Bhatnagar1954}) operator,
leading to the BGK kinetic model \cite{Cercignani1975}.  

The cases mentioned above are just some of the most common kinetic models for
plasmas and gases but other ``combinations'' of self-consistent forces and
collision operators are also considered. The literature on kinetic models is
vast and has many applications. We shall not attempt to give a review here. 

In this paper we are particularly interested in applications to the study of
high-frequency electromagnetic waves in high-temperature plasmas. For such problems,
relativistic effects have to be accounted for (at least for the electrons), but
collisions can be neglected, i.e. $C_\species = 0$, since the time scale of interest
is much shorter than the collision time. In addition, the wave is a small
perturbation of the electromagnetic fields of the plasma so that a formal
linearization of the Vlasov equation can be physically justified. As a result,
the linearized relativistic Vlasov-Maxwell system is the physically appropriate
model for such applications. Simpler models such as the linearized non-relativistic
Vlasov-Maxwell or the linearized Vlasov-Poisson system, even in reduced
dimension $d<3$, can be interesting as well, for fundamental plasma theory.  

It is however worth starting from an overview of the mathematical results for
the non-linear models, in order to understand the expected regularity of the
distribution functions and the electromagnetic field. The mathematical works on
such a large class of models focus in particular on the associated Cauchy problem   
\cite[and references therein]{Glassey1996, DiPerna1988, Lions1991}.

\subsection{Known results on existence and uniqueness}
For the non-relativistic Vlasov-Maxwell system, i.e. equation~(\ref{eq:RVMS}) with
$\gamma = 1$ and $C_\species = 0$, Wollman \cite{Wollman1984} obtained a local
existence and uniqueness result: for a single-species non-relativistic plasma
and with initial data in $H^s$, $s\geq 5$, and with compactly supported initial
particle distribution, there is a time $T>0$ depending on the initial conditions
such that a unique solution
$\ff \in C\big([0,T], H^s(\R^6)\big) \cap C^1\big([0,T], H^{s-1}(\R^6)\big)$
exists. Later Asano and Ukai \cite{Asano1986}, Degond \cite{Degond1986},
and Schaeffer \cite{Schaeffer1986} independently proved similar results on the
local existence and uniqueness of a solution $(\ff_\species, \Ef, \Bf)$ with
lifespan independent of the speed of light $c$. The fact that the lifespan is
independent of $c$ allowed them to take the limit $c \to \infty$ and to show
convergence to a solution of the Vlasov-Poisson system. E.g., Degond showed the
local existence and uniqueness of the solution in Sobolev spaces $H^s$ with
$s\geq 3$; specifically, if the initial data are in $H^3$, and the initial
distribution function is non-negative, $\ff_\species^0 \geq 0$, and compactly
supported in velocity, Degond has shown that there exist a time $T>0$, depending
on the initial data but not on the speed of light $c$, and a solution
$(\ff_\species, \Ef, \Bf)$ in
$L^\infty(0,T; H^3(\R^6) \times H^3(\R^3) \times H^3(\R^3) \big)$. 
Wollman \cite{Wollman1987} improved his earlier result: again for a single-species
non-relativistic plasma, he showed local existence of a $C^1$ solution with
initial data $\ff^0 \in C^1_0(\R^6)$ and  $\Ef^0, \Bf^0 \in H^3(\R^3)$; specifically
there is $T>0$ depending on the initial condition such that there exists a
solution $\ff \in C^1([0,T] \times \R^6)$ which is unique as an element
of $L^1\big(0,T; H^3(\R^6) \big) \cap C\big([0,T], H^2(\R^6) \big)$.

For the relativistic Vlasov-Maxwell system with multiple species Glassey
and Strauss \cite{Glassey1986} proved existence of a unique global solution in
$C^1(\R \times \R^3 \times \R^3)$ with initial data
$\ff_\species^{0} \in C_0^1(\R^3 \times \R^3)$, $\Ef^0, \Bf^0 \in C^2(\R^3)$
under the \emph{a priori} assumption that a solution $\ff_\species$ is
compactly supported in momenta and the radius of the support is bounded by a
continuous function of time. A simpler proof was given by Bouchut, Golse and
Pallard \cite{Bouchut2003}, and a variant of this result has been proposed by
Klainerman and Staffilani \cite{Klainerman2002}. Existence and uniqueness of a
global $C^1$ solution for $t \in [0,+\infty)$, $(x,p) \in \R^6$ has been shown
by Glassey and Schaeffer \cite{Glassey1988} with compactly supported initial
data satisfying appropriate conditions that require, in particular, the initial
distribution function and electromagnetic fields to be small in $C^1$ and $C^2$ 
norms respectively. We note also the results on global existence with small data
for the relativistic Vlasov-Maxwell system obtained by Horst \cite{Horst1990}.
The key observation is that the decay of $\|\Ef(t)\|_{L^2}^2 + \|\Bf(t)\|^2_{L^2}$
for $t \to +\infty$ completely determine the electromagnetic field, without
initial conditions. Horst makes use of a fixed-point argument to show global
existence of solutions: given the electromagnetic fields $(\Ef, \Bf)$ in a
suitable class of functions, he constructs the characteristic flow for the
kinetic equation, from which he computes the charge and electric current
densities that generate new electromagnetic fields  $(\Ef', \Bf')$.
This defines an operator $Q: (\Ef, \Bf) \mapsto (\Ef',\Bf')$ which is a
contraction if the initial distribution $\ff^0$ and its derivatives are small 
enough.  In this construction, the field $(\Ef', \Bf')$ is obtained as the
solution of Maxwell's equations with the condition 
$\|\Ef(t)\|_{L^2}^2 + \|\Bf(t)\|^2_{L^2} \to 0$ for $t \to +\infty$, 
\cite[definition 2.6 and 3.5]{Horst1990}.

Global existence of weak solution without small data assumption is due to
Di~Perna and Lions \cite{DiPerna1989}: for the non-relativistic case and with
one particle species, given initial data  
$\ff^0 \in L^1 \cap L^2(\R^3 \times \R^3)$ and $\Ef^0, \Bf^0 \in L^2(\R^3)$ with the
conditions $\ff^0 \geq 0$ and
\begin{equation*}
  \int_{\R^3 \times \R^3} |v|^2 \ff^0 dx dv <  +\infty, 
\end{equation*}
they prove existence of 
\begin{equation*}
  \ff \in L^\infty \big(0, +\infty; L^1(\R^3 \times \R^3) \big), \quad
  \Ef, \Bf \in L^\infty \big(0,+\infty; L^2(\R^3) \big),
\end{equation*}
that satisfy the non-relativistic Vlasov-Maxwell system in the sense of
distributions. The conditions on the data are the natural ones since 
$\ff(t, \cdot, \cdot)$ is a phase-space density of particles (and thus must be
non-negative and in $L^1$) and the quantity
\begin{equation*}
  \mathcal{E} = \frac{1}{2} m \int_{\R^3 \times \R^3} |v|^2 \ff dxdv
  + \frac{1}{8\pi} \big(
  \|\Ef\|_{L^2}^2 + \|\Bf\|_{L^2}^2 \big),
\end{equation*}
is the total energy of the system. Di~Perna and Lions have shown that for
such global weak solutions, one has $\mathcal{E}(t) \leq \mathcal{E}(0)$, that
is, they are finite-energy solutions. The key idea of this proof is the use of
renormalized solutions \cite{DiPerna1988, Lions1991}, and this idea has been
applied to the Boltzmann equation \cite{DiPerna1989a} as well. More recently, the
relativistic version has been addressed by Rein \cite{Rein2004}, while
time-periodic weak solutions in bounded (in space) domains have been considered
by Bostan \cite{Bostan2006}. 

As for the Vlasov-Poisson system the study of the Cauchy problem developed along
the same lines, moving from local-in-time classical solutions up to global weak
solutions \cite[and references therein]{Glassey1988, DiPerna1988}. Particularly,
the works by Asano and Ukai, Degond, and Schaeffer cited above give results on
the existence and uniqueness of local solutions the Vlasov-Poisson system. There
are however earlier results on local solutions \cite{Horst1981, Horst1982} and
on global weak solutions \cite{Arsenev1975, Horst1984}.
Existence of a global $C^1$ solutions with small initial data has been
established by Bardos and Degond \cite{Bardos1985}. The first results on global
classical solutions is due to Pfaffelmoser \cite{Pfaffelmoser1992}, where
``classical'' here means that the characteristics system for the Vlasov equation
has a unique classical solution and $\ff$ is constant along the characteristics.

All these results are for the fully nonlinear problem. In this work, however, we
address the linearized problem and we focus specifically on the associated
linear kinetic equation. Wollman \cite[section 3]{Wollman1987} reports the
classical results on the existence of $C^1$ solutions for such linear problems,
the proof of which is based on the standard method of characteristics.
Specifically, if the electric and magnetic fields $\Ef, \Bf$ are in
$C\big([0,T]; H^3(\R^3) \big)$ with $H^3$-norm bounded uniformly in time, the
linear non-relativistic equation with initial datum
$\ff^0 \in C^1_0(\R^3 \times \R^3)$ has a unique classical solution
$\ff \in C^1([0,T] \times \R^3 \times \R^3)$. Here $T$ is the life-span of the
fields and does not depend on the initial distribution.

\subsection{Framework of this paper}
In this paper, we consider a given stationary configuration 
$\{F_{\species,0}(x, p)\}_{\species}$ of the particle distribution functions,
with zero electric field $E_0=0$, and a constant magnetic field $B_0$. 
Then we address the linearized system around the stationary solution
$(\{F_{\species,0}\}_\species, E_0, B_0)$. The associated unknowns are the
linear perturbations $(\{f_{\species}\}_\species, E, B)$, where $f_\species$  is
the perturbed distribution function for the particle species $\species$, while
$E$ and $B$ are the perturbations of electric and magnetic field, respectively. 

Since we consider a linearized problem, the solution for $f_\species$
is not necessarily non-negative, but $F_{\species,0} + f_\species \geq 0$
(essentially, if $F_{\species,0} + f_\species$ fails to be non-negative, we no
longer are in the linear regime). We need to keep the assumption on the
existence of the velocity moments of $f_\species$ and the $L^1$-in-$x$ behavior
(or $L^1_{\mathrm{loc}}$ in the idealized case of a plasma with an infinite
number of particles, e.g., a uniform plasma over the whole space).

We could expect to need and to be able to consider given arbitrary initial data. 
Our aim however is characterizing and determining what is referred to as the 
\emph{dielectric response of the plasma} in the physics literature 
\cite{Brambilla1998, Stix1992, Larsson1982, Bornatici2003a}. 
The underlying physical idea is that the application of a small-amplitude
electromagnetic perturbation determines a small change in the distribution
functions $\ff_\species$, which represents the response of the plasma to the
imposed electromagnetic disturbance. It is similar to the construction of the
operator $Q$ in Horst's fixed-point argument mentioned above. In addition
however, physical reasoning suggests that the response of the plasma should be
uniquely determined by and depend continuously on this imposed perturbation,
hence there should be no need to prescribe a Cauchy datum.

The evolution of a small perturbation $f_\species$ induced by an externally
imposed small electric field disturbance $E$ is governed by the linear
relativistic Vlasov equation
\cite{Bornatici1983, Brambilla1998, Kennel1966, Stix1992},  
\begin{equation}
  \label{eq:linearized-vlasov}  
  \partial_t f_\species + v_\species \cdot \nabla_x f_\species + q_\species 
  \big(v_\species \times B_0/c \big) \cdot \nabla_p f_\species = 
  - q_\species \big(E + v_\species \times B/c \big) \cdot
  \nabla_p F_{\species,0},
\end{equation}
where the electric field $E$ of the disturbance is given, e.g.,  
$E \in [\Schwartz(\R^4)]^3$, and the magnetic field depends linearly on
the electric field via the Faraday law, 
\begin{equation}
  \label{eq:Faraday}
  \partial_t B + c \nabla \times E = 0.
\end{equation}
Since $E$ is given, this is a system of partial differential equations for
$(f_\species, B)$ which reads
\begin{equation*}
  \left\{
  \begin{aligned}
    & \partial_t f_\species + v_\species \cdot \nabla_x f_\species + q_\species 
    (v_\species \times B_0 / c) \cdot \nabla_p f_\species = 
    - q_\species \big(E + v_\species \times B / c \big) \cdot
    \nabla_p F_{\species,0}, \\
    & \partial_t B =- c \nabla \times E.
  \end{aligned}
  \right.
\end{equation*}
From this system, by introducing the new unknown $g_\species = \partial_t f_\species$,
one deduces the decoupled system
\begin{equation}
  \label{eq:gsB-system}
  \left\{
  \begin{aligned}
    & \partial_t g_\species + v_\species \cdot \nabla_x g_\species + q_\species 
    (v_\species \times B_0 / c) \cdot \nabla_p g_\species = 
    - q_\species \big(\partial_t E - v_\species \times \nabla \times E \big) \cdot
    \nabla_p F_{\species,0}, \\
    & \partial_t B =- c \nabla \times E.
  \end{aligned}
  \right.
\end{equation}
The associated homogeneous equation for $g_\species$ is
\begin{equation*}
  \mathcal{H}(g_\species) \coloneqq
  \partial_t g_\species + v_\species \cdot \nabla_x g_\species + q_\species 
  \big( v_\species \times B_0 / c \big) \cdot \nabla_p g_\species = 0.
\end{equation*}
We shall see that if we impose a control of the growth at $t\to\pm\infty$,
the modified equation
$\mathcal{H}(g_{\species,\nu}) + \nu g_{\species,\nu} = 0$ for any $\nu > 0$ has
the unique solution $g_{\species,\nu} = 0$. This is similar to scattering theory
\cite{Lax1967}, in which the scattered field is determined by conditions at
infinity. The introduction of the damping term $\nu g_{\species,\nu}$ is
analogous to the limiting absorption principle
\cite{Lax1967, Eidus1986, Sanchez-Palencia1980} for elliptic equations of the
form $Au=(\lambda + i \eps) u + f$ for $\eps \to 0^+$ (resp. $\eps \to 0^-$), 
where one constructs two resolvents for the elliptic operator $A$. A
classical application of the limiting absorption principle consists in selecting
the unique outgoing-wave solution of the Helmholtz equation. Other examples
include the solution of elliptic equations, as well as the extension of the
resolvent of the operator $-\Delta+V$
\cite[and references therein]{Eidus1986, Nakamura2006, Cacciafesta2018}.
Recently the idea of using scattering theory for the linearized Vlasov-Maxwell
system has been developed by Despr\'es in order to prove linear Landau damping
for inhomogeneous equilibrium distributions \cite{Despres2019}. In this paper,
however, we shall not take advantage of scattering theory, but rather focus on
the solution selected by the growth conditions at infinity in time and the
limiting absorption principle.

The main idea here is to apply the limiting absorption principle to either the
inhomogeneous problem~(\ref{eq:linearized-vlasov}) or (\ref{eq:gsB-system}). We
shall see that for $\nu > 0$, there is a unique tempered solution, which has a
limit for $\nu \to 0^+$, and the limit itself is a tempered solution of
either~(\ref{eq:linearized-vlasov}) or~(\ref{eq:gsB-system}) without damping. We
will find that the limit amounts  exactly to the solution which is referred to
as the \emph{causal solution} in the physics literature and which describes the
response of the plasma to the imposed $E$.

Finding the response of a system is classical. For example, in the modeling of
electric circuits with capacitance $C$ inductance $L$ and resistance $R$, the
electric charge $q: t \mapsto  q(t)$ satisfies the ordinary differential equation 
\begin{equation*}
  LC \ddot{q} + RC \dot{q} + q = CU_0 \cos(\omega t),
\end{equation*}
which leads to the response
\begin{equation*}
  q(t) = \mathrm{Re} \Big[\frac{CU_0 e^{i\omega t}}{1 - LC \omega^2 + i \omega RC}
    \Big],
\end{equation*}
even though one has infinitely many solutions (depending on Cauchy data).

Returning to the linearized Vlasov equation, one can also envisage the use of
other dissipation mechanisms such as the Fokker-Planck operator
$\nu \nabla_v \cdot(\nabla_v f - vf)$ or the collision operators mentioned
above, but this is not addressed here. 

Having identified the unique solution of the linearized Vlasov equation which
describes the response of the plasma to the imposed electric field disturbance,
the corresponding unique perturbation of the electric current density is defined by 
\begin{equation}
  \label{eq:current-density}
  j(t,x) \coloneqq \sum_{\species} q_\species \int_{\R^3} v_\species(p)
  f_\species (t, x, p) dp,
\end{equation}
which requires that the solution $f_\species$ is at least in $L^1$ (the
relativistic velocity is bounded by the speed of light, $|v_\species (p)| < c$).
We shall generalize such integrals to the case of distributional solutions and
see that, since $f_\species$ depends linearly on $E$, the induced current
density $j$ is given by the action of a linear operator on the electric field
$E$, namely,   
\begin{equation}
  \label{eq:response}
  j = \sigma (E).
\end{equation}
This is referred to as the linear constitutive relation of the plasma
and the operator $\sigma$ is the conductivity operator.
Equation~(\ref{eq:response}) is also referred to as (generalized) Ohm's law.

%% However, in order to define the operator $\sigma$ in~(\ref{eq:response}), for a
%% given electric-field disturbance $E$ there must be a unique distribution
%% function $f_\species$, which describes the effect of $E$ on the $\species$-th
%% particle species. But both equation~(\ref{eq:linearized-vlasov}) and
%% (\ref{eq:Faraday}) require initial data that are not known a priori. We know
%% that the physically acceptable solution must satisfy the causality conditions 
%% $f_\species (t,x,p) \to 0$ and $B(t,x) \to 0$ for $t \to -\infty$, hence the
%% name \emph{causal solution}. We demonstrate here that the limiting absorption
%% principle yields the causal solution and thus the response of the plasma.

%% One should note however, that equation~(\ref{eq:linearized-vlasov}) has a
%% classical non-trivial null space and the magnetic field is not uniquely
%% determined by the Faraday's law unless a Cauchy datum is provided. The same
%% remark holds for the equation for $g_\species$ in (\ref{eq:gsB-system}). The
%% limiting absorption principle selects one specific Cauchy datum which is the
%% only one for which both $f_\species$ and $B$ have a zero limit for
%% $t \to -\infty$ (null Cauchy conditions at $t_0 = -\infty$), hence the name
%% \emph{causal solution}. 

A precise mathematical analysis of the response of a plasma is important because
it is the basis for the construction of constitutive relations for linear plasma
waves, the simplest of which being the Ohm's law. Together with Maxwell's
equations, it determines the linear wave equation describing plasma waves
\cite{Bornatici2003a}. The same problem has been considered by Omnes for a
bounded plasma \cite{Omnes2003}. More recently, Cheverry and Fontaine
\cite{Cheverry2017, Cheverry2017a} have addressed the characteristic variety (or
dispersion relation) for the linearized Maxwell-Vlasov system using asymptotic
methods, but here we focus on the properties of the plasma constitutive relation
as an operator.

We carry out this ideas for the one-dimensional non-relativistic linearized
Vlasov equation without background magnetic field (non-magnetized) and for the
three-dimensional relativistic linearized Vlasov equation with uniform
background and with constant magnetic field.

\subsection{Main results 1: the non-relativistic, one-dimensional case}  
\label{sec:main-1}

We consider first the case of a non-magnetized, non-relativistic plasma in one
dimension in space and velocity. We also restrict ourselves to the case of a
single particle species (and thus drop the index $\species$). This last
simplification does not imply any further loss of generality as the current
density is the sum of the currents carried by the individual species.
With background distribution $F_0 \in \Schwartz(\R)$ depending on $v \in \R$
only, and $E \in \Schwartz(\R^2)$, we consider the linear kinetic equation
\begin{equation}
  \label{eq:Landau-1d}
  \mathcal{L} f \coloneqq 
  \partial_t f + v \partial_x f = - \frac{q}{m} E F'_{\species,0},
\end{equation}
which is a reduced version of equation~(\ref{eq:linearized-vlasov}).
The damped form is 
\begin{equation}
  \label{eq:Landau-damped}
  \mathcal{L}_\nu f_\nu \coloneqq 
  \partial_t f_\nu + \nu f_\nu + v \partial_x f_\nu = - \frac{q}{m} E F'_0.
\end{equation}
If $f$ is a generic distribution (not necessarily a solution of
(\ref{eq:Landau-1d})) with finite first velocity moment, i.e., with
$v \mapsto vf(t,x,v)$ in $L^1$ for all $(t,x)$, we define the operator
\begin{equation}
  \label{eq:current-1d}
  \mathbb{J}(f)(t,x) \coloneqq q \int_\R v f (t,x,v) dv,
\end{equation}
which gives the associated current density, cf. equation~(\ref{eq:current-density}),
in the one-dimensional, non-relativistic case.

\begin{theorem}
  \label{th:main-1-1}
  Let $F_0 \in \Schwartz(\R)$ be given.
  \begin{itemize}
  \item[(i)] If $\nu > 0$, for any $E \in \Schwartz(\R^2)$
    equation~(\ref{eq:Landau-damped}) has a solution $f_\nu \in \Schwartz(\R^3)$
    which is unique as an element of $\Schwartz^\prime(\R^3)$,
    and $j_\nu = \mathbb{J}(f_\nu) \in \Schwartz(\R^2)$. 
  \item[(ii)] For $\nu \to 0^+$, $f_\nu$ and $j_\nu$ have pointwise limits
    $f \in C^\infty_b(\R^3)$ and $j \in C^\infty_b(\R^2)$, respectively;
    in addition, $f_\nu \to f$, and $j_\nu \to j$ also in the topology of
    $\Schwartz^\prime$. 
  \item[(iii)] The limit $f$ is a solution of equation~(\ref{eq:Landau-1d}) and
    $j = \mathbb{J}(f)$.  
  \end{itemize}
\end{theorem}

By the limiting absorption principle of theorem~\ref{th:main-1-1}, to each
$E\in\Schwartz(\R^2)$ we thus can associate a unique $f$, and thus a unique
current density $j$. The conductivity operator is then defined as
the map
\begin{equation}
  \label{eq:sigma-1-2}
  \sigma : \Schwartz(\R^2) \ni E \mapsto j \in C_b^\infty(\R^2) \subset
  \Schwartz^\prime(\R^2).
\end{equation}
For any $\chi \in C_0^\infty(\R)$ with $\chi = 1$ in a neighborhood of zero,
we also introduce the following Fourier multiplier
(cf. appendix~\ref{sec:basic-def} for definitions and notations) 
\begin{equation}
  \label{eq:sigma_1-chi}
  \mathcal{F} \big( \sigma_{1-\chi} (E) \big) (\omega, k) = 
  \big(1 - \chi(k)\big) \hat{\sigma}_{\mathrm{ph}} (\omega,k) 
  \hat{E}(\omega,k), 
\end{equation}
where $\mathcal{F}$ denotes the Fourier transform and $\hat{\sigma}_{\ph}$ 
is the physical conductivity tensor (explicit formula given in
equation~(\ref{eq:sigma-ph})). 

\begin{theorem}
  \label{th:main-1-2}
  The map $\sigma$ defined by equation~ (\ref{eq:sigma-1-2}) is linear and
  continuous from $\Schwartz(\R^2) \to \Schwartz^\prime(\R^2)$ and for every 
  $\chi\in C_0^\infty(\R)$ with $\chi=1$ near zero, $\sigma(E) = \sigma_{1-\chi}(E)$
  for all $E$ satisfying $\hat{E}(\omega,k) = 0$ if $k \in \supp\chi$.
\end{theorem}
  
\begin{remark}
  \label{rem:mein-1-1}
  \leavevmode
  \begin{enumerate}
  \item The limit $f$ established in theorem~\ref{th:main-1-1} coincides with
  the causal solution of~(\ref{eq:Landau-1d}), which is reviewed in
  appendix~\ref{sec:causal-solutions}.
  \item Expressions for the solutions $f_\nu$, $f$, their Fourier transforms, the
    associated currents, and the operators $\sigma_\nu$
    and $\sigma$ are given in section~\ref{sec:Landau}.
  \item  The operator $\sigma_\chi \coloneqq \sigma-\sigma_{1-\chi}$ is well defined
    and $\sigma_\chi(E) = 0$ when $\hat{E}(\omega,k) = 0$ for $k$ small. 
    Expressions of both $\sigma$ and $\sigma_\chi$ are available
    (see proposition~\ref{th:limit-current-2}).
  \item Theorem~\ref{th:main-1-1} can be straightforwardly generalized to the
    case of a non-magnetized non-relativistic plasma with a spatially
    non-homogeneous equilibrium distribution of the form  
    $F_0 (x,v) = n_0(x) \tilde{F}_0(v)$, for which
    the velocity distribution is the same at any point in space. For such
    equilibria, $\sigma_{1-\chi}$ is a pseudo-differential operator. The
    expression of the symbol is obtained in
    section~\ref{sec:limiting-absorption-Landau}, remark~\ref{rem:PsiDO}. 
  \end{enumerate}
\end{remark}

\subsection{Main results 2: relativistic, three-dimensional case} 
\label{sec:main-2}

The second case under consideration is the relativistic Vlasov equation
with a uniform background magnetic field, that is, $B_0$ in
equation~(\ref{eq:linearized-vlasov}) is taken constant and non-zero.
We choose $B_0 = |B_0| e_\parallel$ directed along the third
axis of a Cartesian frame $\{e_1, e_2, e_3 = e_\parallel\}$. It is natural to
normalize the relativistic momentum $p$ to $m_\species c$, and thus to introduce
normalized momentum variables 
\begin{equation}
  \label{eq:us}
  u \coloneqq p /(m_\species c)
  \qquad
  u_\perp \coloneqq (u_1^2 + u_2^2)^{1/2},
  \qquad
  u_\parallel \coloneqq u_3.
\end{equation}
We define the relativistic cyclotron frequency for the considered species, 
\begin{equation}
  \label{eq:omega-c}
  \Omega_\species (u) \coloneqq \frac{1}{\gamma(u)} \frac{q_\species |B_0|}{m_\species c }
  = \sgn(q_\species) \frac{\omega_{c,\species}}{\gamma(u)} > 0,
\end{equation}
which has the sign of the charge $q_\species$ and depends on $u^2$ through
$\gamma(u)=(1+u^2)^{1/2}$, whereas the classical cyclotron frequency
$\omega_{c,\species} \coloneqq |q_\species B_0|/(m_\species c) > 0$ is a positive constant.

The background distribution functions are taken uniform and \emph{gyrotropic},
i.e., $F_{\species,0}$ is constant in time $t$ and space $x$ and depends only on
$u_\parallel$, $u_\perp$, namely,
\begin{equation}
  \label{eq:gyrotropic}
  F_{\species,0} (t,x,p) = \frac{n_{\species,0}}{(m_\species c)^3}
  G_\species (u_\parallel, u_\perp),
\end{equation}
where $n_{\species,0} > 0$ is the constant background particle density and
$G_\species$ is such that 
$u \mapsto G_{\species}(u_3, (u_1^2+u_2^2)^{1/2}\big)$ belongs to $\Schwartz(\R^3)$.
Such property is usually satisfied by the background distribution functions of
practical interest. The momentum distribution $G_\species$ has unit norm in
$L^1(\R \times \R_+, 2\pi u_\perp du_\parallel du_\perp)$. 

Instead of addressing equation~(\ref{eq:linearized-vlasov}) directly, we
consider the kinetic equation in~(\ref{eq:gsB-system})
for $g_{\species} = \partial_t f_{\species}$, which amounts to
\begin{equation}
  \label{eq:magnetized-uniform-g}
  \mathcal{V}_{\species} g_{\species}
  = - q_\species \big(\partial_t E - v_\species \times \nabla \times E \big)
  \cdot \nabla_p F_{\species,0},
\end{equation}
where
\begin{equation}
  \label{eq:vlasov-magnetized-uniform}
  \mathcal{V}_{\species} = \partial_t + v_\species \cdot\nabla_x
  +\sgn(q_\species) \frac{\omega_{c,\species}}{\gamma}
  (u \times e_\parallel) \cdot \nabla_u,
\end{equation}
with $v_\species = p /(m_\species \gamma_\species) = c u/\gamma(u)$ being the
relativistic velocity as a function of the normalized momentum.

We add to equation~(\ref{eq:magnetized-uniform-g}) a
damping term, with the idea of applying the limiting absorption principle.
Because of the $u$-dependent relativistic factor $\omega_{c,\species}/\gamma$
the damping coefficient will be multiplied by $\gamma$ after Fourier transform
in time and space of equation~(\ref{eq:magnetized-uniform-g}). It is therefore
convenient to allow the damping coefficient $\nu_\species$ to depend on the
species $\species$ and on momentum $u$ from the beginning, subject to the conditions 
\begin{equation}
  \label{eq:nu-conditions}
  \left\{
  \begin{aligned}
    & \nu_\species \in C^\infty(\R^3), \\
    & \text{there exists $\nu_0 >0$ : } \gamma(u)\nu_\species(u) \geq \nu_0, \\
    &  (u_1 \partial_{u_2} - u_2 \partial_{u_1})\nu_\species(u) = 0, \\
    & \text{and there exists $m \in \R$ : }
    |\partial^\alpha_u \nu_\species(u)| \leq C_\alpha (1+u^2)^m, \; \forall u \in
    \R^3, \forall \alpha \in \N_0^3,
  \end{aligned}
  \right.
\end{equation}
where $C_\alpha \in \R$ are constants depending only on the order $\alpha$ of
the derivatives. For example, if $\nu_\species \in C^\infty_b(\R^3)$ is a function of
$u_1^2 + u_2^2$ and $u_3$ only, then conditions~(\ref{eq:nu-conditions}) are
fulfilled with $m=0$. We shall see that the dissipation-less limit is
independent of the choice of this damping function. 

For $\eps > 0$ and for any function $\nu_\species$, satisfying conditions
(\ref{eq:nu-conditions}), we consider the regularized equation
\begin{equation}
  \label{eq:magnetized-uniform-g-damped}
  \mathcal{V}_{\species,\eps} g_{\species,\eps}
  = - q_\species \big(\partial_t E - v \times \nabla \times E \big)
  \cdot \nabla_p F_{\species,0},
\end{equation}
where $\mathcal{V}_{\species,\eps} = \mathcal{V}_\species  + \eps \nu_\species$.

If $g_\species$ are generic distributions with finite first velocity moment,
i.e., $g_\species(t,x,\cdot) \in L^1 (\R^3)$ for every $(t,x)$, we define 
\begin{equation}
  \label{eq:djnu_dt}
  \mathbb{K}(\{g_\species\})(t,x) \coloneqq \sum_{\species} q_\species (m_\species c)^3
  \int_{\R^3} v_\species(u) g_\species(t,x,u) du,
\end{equation}
which gives the time-derivative of the current~(\ref{eq:current-density}) when
$g_\species = \partial_t f_\species$. 
For this model, the limiting absorption principle parallels theorem~\ref{th:main-1-2}.

\begin{theorem}
  \label{th:main-2-1}
  Let $F_{\species,0} \in \Schwartz(\R^3)$ be given uniform gyrotropic
  distribution functions and let $\nu_\species$ be any function satisfying
  conditions~(\ref{eq:nu-conditions}).  
  \begin{itemize}
  \item[(i)] If $\eps > 0$, for any $E \in [\Schwartz(\R^4)]^3$
    equation~(\ref{eq:magnetized-uniform-g-damped}) has a solution
    $g_{\species,\eps} \in \Schwartz(\R^7)$ which is
    unique as an element of $\Schwartz^\prime(\R^7)$,
    $\partial_t j_\eps = \mathbb{K}(\{g_{\species,\eps}\}) \in [\Schwartz(\R^4)]^3$.
  \item[(ii)] For $\eps \to 0^+$, $g_{\species,\eps}$ and
    $\partial_t j_\eps$ have limits $g_\species$ and $\partial_t j$ in
    $\Schwartz^\prime$, independent of $\nu_\species$. 
  \item[(iii)] The limit $g_\species$ belongs to $C^\infty_b(\R^7)$ and is a
    classical solution of equation~(\ref{eq:magnetized-uniform-g}); in addition
    $g_\species$ belongs to the domain of $\mathbb{K}$ and
    $\partial_t j = \mathbb{K}(\{g_\species\}) \in C^\infty_b(\R^4)$.
  \end{itemize}
\end{theorem}

Therefore we can define the map
\begin{equation}
  \label{eq:sigma-2-2}
  \varsigma : [\Schwartz(\R^4)]^3 \ni E
  \mapsto \partial_t j \in [\Schwartz^\prime(\R^4)]^3.
\end{equation}
We shall see that it can be represented by a Fourier multiplier if we exclude
the hyperplane $\omega = 0$ in Fourier space.
For any cut-off function $\chi \in C_0^\infty(\R)$ with $\chi=1$
in a neighborhood of zero, we define the Fourier multiplier $\varsigma_{1-\chi}$ by
\begin{equation}
  \label{eq:sigma_ph-uniform-magnetized}
  \mathcal{F}(\varsigma_{1-\chi}(E))(\omega,k) = \big(1-\chi(\omega)\big)
  \hat{\varsigma}_0 (\omega,k)\hat{E}(\omega,k),
\end{equation}
where $\hat{\varsigma}_0(\omega,k)$ is the limit established
in proposition~\ref{th:main-4-2}. An expression for $\hat{\varsigma}_0$ is given
in equation~(\ref{eq:varsigma-ph}), and proposition~\ref{th:main-4-2}
establishes that $\hat{\varsigma}_0$ is continuous  for $\omega\not=0$ and
$C^\infty$ where $\omega^2 \not= c^2k^2 + n^2 \omega_{c,\species}^2$ for all
$n \in \Z$ and all species $\species$.  

\begin{theorem}
  \label{th:main-2-2}
  The map $\varsigma$ defined in (\ref{eq:sigma-2-2}) is continuous
  and for any $\chi\in C_0^\infty(\R)$ with $\chi = 1$ in a neighborhood of zero,
  $\varsigma(E) = \varsigma_{1-\chi}(E)$
  if $\hat{E}(\omega,k) = 0$ for $\omega \in \supp\chi$.
\end{theorem}

\begin{remark}
  \label{rem:mein-1-2}
  \leavevmode
  \begin{enumerate}
  \item The hypothesis that $\nu_\species$ is in the kernel of the operator
    $u_1 \partial_{u_2} - u_2 \partial_{u_1}$ expresses the fact that
    $\nu_\species$ must  be gyrotropic.
  \item The solution $g_\species$ is the unique causal solution
    of (\ref{eq:magnetized-uniform-g}) (defined in
    appendix~\ref{sec:causal-solutions}). 
  \item We also have pointwise convergence of $g_{\species,\eps} \to g_\species$.
  \item Existence and uniqueness of the solution $g_{\species,\eps}$ is
    established via Fourier transform, while the causal solution is obtained by
    integration along the characteristics, cf. appendix~\ref{sec:causal-solutions}.
    Hence the proof of theorem~\ref{th:main-2-1} (iii) establishes a link
    between the formulations in Fourier and physical variables.
  \item An explicit expression of the linear
    operator $\varsigma$ valid without restriction of the support of $E$ is
    given below in proposition~\ref{th:jnu-convergence}. 
  \end{enumerate}
\end{remark}

\subsection{Concluding remarks and structure of the paper}
Theorems~\ref{th:main-1-2} and~\ref{th:main-2-2} in particular show that the
response of a uniform plasma to oscillatory electromagnetic disturbances can be
expressed by a Fourier multiplier. Although limited to a simple plasma
equilibrium, these results support the physics theories that rely on the
pseudo-differential form of the conductivity operator \cite{Kravtsov1969a,
  Bernstein1975, McDonald1985, McDonald1988, McDonald1991, Bornatici2003a}. More
precisely, even though the response of a plasma is rigorously not a
pseudo-differential operator, it can be written as the sum of a
pseudo-differential operator plus a remainder which vanishes if the spectrum of
the electric-field disturbance is supported away from $\omega=0$
(or $k=0$ in the simpler case of theorem~\ref{th:main-1-1}); this is typically
the case in the envisaged applications, since the frequency of the perturbation
is set by an external source and it is tuned to resonate with the cyclotron
motion of a particle species. 

In order to illustrate these specific applications at least qualitatively, we
return to the linearized Vlasov-Maxwell system, that is,
\begin{equation}
  \label{eq:lin-VM-system}
  \left\{
  \begin{aligned}
    \partial_t f_\species +v_\species(p)\cdot\nabla_x f_\species +
    q_\species (v_\species \times B_0/c) \cdot \nabla_p f_\species
    &= -q_\species (E + v_\species \times B/c) \cdot \nabla_pF_{\species,0}, \\
    \partial_t E - c \nabla \times B &= -4\pi \sum_\species q_\species
    \int_{\R^3} v_\species f_\species dp, \\
    \partial_t B + c \nabla \times E &= 0,
  \end{aligned}
  \right.
\end{equation}
for $f_\species$, $E$, and $B$. The two equations for the electromagnetic field
$(E,B)$ imply (formally at least, by taking the time-derivative of the
Amp\`ere-Maxwell law)
\begin{equation*}
  \partial_t^2 E + c^2 \nabla \times \big(\nabla \times E\big) + 4\pi
  \partial_t j = 0,
\end{equation*}
where $\partial_t j = \mathbb{K}(\{\partial_t f_\species\})$. This equation
depends on the time-derivative of the induced current $\partial_t j$ rather than
on $j$ alone, and the map~(\ref{eq:sigma-2-2}) give $\partial_tj = \varsigma(E)$
when $E \in \Schwartz(\R^4)$. However, if the solution is highly oscillatory
(high-frequency waves), $\hat{E}(\omega,k) = 0$ near $\omega=0$ and we can
replace $\varsigma$ by the Fourier
multiplier~(\ref{eq:sigma_ph-uniform-magnetized}) with the low-frequency
cut-off, obtaining   
\begin{equation}
  \label{eq:wave-equation}
  D(i\partial_t, -i\partial_x)E \coloneqq
  \partial_t^2 E + c^2 \nabla \times \big(\nabla \times E\big)
  + 4\pi \varsigma_{1-\chi}(E) = 0,
\end{equation}
which is a constant-coefficients pseudo-differential equation for the electric
field only. Theorem~\ref{th:main-2-2} implies that the operator
$D(i\partial_t,-i\partial_x): [\Schwartz(\R^4)]^3 \to [\Schwartz^\prime(\R^4)]^3$
is continuous and we have established regularity results for its symbol in
proposition~\ref{th:main-4-2}. The symbol, in particular, is polynomially bounded
and thus the operator extends to  
\begin{equation*}
  D(i\partial_t,-i\partial_x) : [H_{(s)}(\R^4)]^3  \to [H_{(s-m)}(\R^4)]^3,
 \qquad s \in \R,
\end{equation*}
where $H_{(s)}(\R^4)$ is the space of $w \in \Schwartz^\prime(\R^4)$ such
that $(1 + (\tfrac{\omega}{\hat{\omega}})^2 +
(\tfrac{ck}{\hat{\omega}})^2)^{\frac{s}{2}} \hat{w} \in L^2(\R^4)$, with the
normalization frequency chosen by $\hat{\omega} \coloneqq \max_\species |\omega_{c,s}|$
and with $m = \max\{2M, 2\}$ where the integer $M$ being the degree of the
polynomial bound for $\hat{\varsigma}_0$ established in
proposition~\ref{th:main-4-2} in section~\ref{sec:Limiting-absorption-magnetized}.

The semi-classical methods commonly used to find approximate solutions of the
wave equation~(\ref{eq:wave-equation}) (cf. the work of Prater et
al. \cite{Prater2008} for an overview of computational tools) are valid under
strong assumptions on the symbol of the operator $D$. These assumptions are
smoothness of the symbol and the fact that its anti-Hermitian part is small in a
certain sense (weakly non-Hermitian operators) \cite{Kravtsov1969a,
  Bernstein1975, McDonald1985, McDonald1988, McDonald1991, Bornatici2003a}. 
In this paper we analyze the construction of the operator in detail.
Theorem \ref{th:main-2-2} shows that the full operator $\varsigma$ defined
in~(\ref{eq:sigma_ph-uniform-magnetized}) has an additional contribution that
accounts for the low-frequency response of the plasma. For the high-frequency
part (operator~(\ref{eq:sigma_ph-uniform-magnetized})), smoothness of the symbol
is established almost everywhere in Fourier space, cf. proposition~\ref{th:main-4-2}.
As for the assumption of weak anti-Hermitian part, this is always violated  
near cyclotron resonances and the application of standard computational methods
is justified by heuristic arguments only. Propagation near a resonance has been
addressed in the physics literature \cite{Westerhof1997, Balakin2007} but a
satisfactory theory is not available. Then the only rigorous approach to the
problem would be the direct numerical computation of the solution of the
linearized Vlasov-Maxwell system, which is computationally too expensive in
realistic cases. The precise characterization of response operators may help to
improve the available methods toward including resonances. 

In section~\ref{sec:simple-case}, a simple case study is presented in order to
illustrate the ideas. The rest of the paper is dedicated to the proofs.
In section~\ref{sec:Landau} the case of a non-relativistic isotropic plasma in
one dimension is addressed, while section~\ref{sec:uniform} is dedicated to
relativistic uniform magnetized plasmas. An overview of notations and standard
definitions together with technical results can be found in the appendices. In
appendix \ref{sec:causal-solutions} in particular a precise definition of causal
solution is given for advection equations associated to global-in-time flows.

% New section
% =========================================================================
%
\section{Characterization of the response operator: a simple case study}
\label{sec:simple-case}

In this section, we study a simple model which however contains the essential
elements of the full problem. The aim of these simple considerations is showing
how the limiting absorption principle determines the causal solution of a
hyperbolic equation. All proofs are straightforward and reported in
appendix~\ref{sec:proofs-toy-model} for the reader's convenience.   

Given $v \in \Schwartz(\R^{1+d})$, we consider the equation
\begin{equation*}
  \partial_t u(t,x) = v(t,x), \quad u(0,\cdot) = u_0 \in C^\infty_b(\R^d),
\end{equation*}
where $C_b^\infty$ is the space of smooth bounded functions with bounded
derivatives, cf. appendix~\ref{sec:basic-def} for the precise definition.

\begin{proposition}
  \label{th:simple-model-causal-solution}
  For $v \in \Schwartz(\R^{1+d})$, there exists a unique solution $u$ in 
  $C^\infty_b(\R^{1+d})$ of the equation $\partial_t u = v$ such that
  $\lim_{t\to-\infty} u(t,x) = 0$ pointwise in $x$, and  
  \begin{equation*}
    u(t,x) = \int_{-\infty}^t v(s,x)ds.
  \end{equation*}
  The map $v \mapsto u$ is a continuous linear operator both from
  $\Schwartz(\R^{1+d}) \to \Schwartz^\prime(\R^{1+d})$ and from
  $\Schwartz(\R^{1+d}) \to L^\infty(\R^{1+d})$.
\end{proposition}

If we think of $u$ as the response to a localized perturbation $v$, causality
requires that $u \to 0$ for $t\to - \infty$ since the perturbation decreases
faster than any polynomial when $t \to -\infty$. Hence, the solution given in
proposition \ref{th:simple-model-causal-solution} is referred to as the causal
solution. Since $v \in \mathcal{S}(\R^{1+d})$, 
\begin{equation*}
  \int_{-\infty}^0 v(s,\cdot)ds 
\end{equation*}
is finite. The causality principle selects a unique initial condition $u_0$ given by
\begin{equation*}
  u_0 (x) = \int_{-\infty}^0 v(s,x)ds,
\end{equation*}
thereby allowing us to define the linear continuous operator $v \mapsto u$, which
we view as the response of the operator $\partial_t$ for a perturbation $v$.
We also note that the limit for $t \to +\infty$ of the solution gives the time
integral of the perturbation $\int_\R v(s,x)ds$. 

The following simple result provides us with a characterization of the causal
solution, that is used in this paper for more general problems.

\begin{proposition}
  \label{th:characterization}
  The damped problem $\partial_t u^\nu + \nu u^\nu = v$ in 
  $\Schwartz^\prime(\R^{1+d})$ with $\nu >0$ and $v \in \Schwartz(\R^{1+d})$ has
  a unique solution $u^\nu \in \Schwartz^\prime(\R^{1+d})$. It belongs to
  $\Schwartz(\R^{1+d})$ and it is given by 
  \begin{equation*}
    u^\nu(t,x) = \int_{-\infty}^t e^{-\nu \cdot (t-t')} v(t', x)dt'.
  \end{equation*}
  Furthermore, $u^\nu \to u$ in $\Schwartz^\prime(\R^{1+d})$ as $\nu \to 0^+$,
  where $u$ is the causal solution obtained in
  proposition~\ref{th:simple-model-causal-solution}. 
\end{proposition}
 
\begin{remark}
  The integral in the definition of $u^\nu$ is absolutely convergent as 
  $t - t' \geq 0$ on the domain of integration and $v(\cdot, x) \in L^1(\R)$.
  The fact that $u^\nu \in \Schwartz(\R^{1+d})$ is proven by showing that the
  Fourier transform $\hat{u}^\nu$ is in $\mathcal{S}(\R^{1+d})$.
\end{remark}

Proposition~\ref{th:characterization} establishes that the admissible solution
$u$ of the model without dissipation
(in the sense that $u \to 0$ when $t \to -\infty$) is the limit of the unique
solution $u^\nu$ in $\Schwartz^\prime$ of the model with dissipation.

% New section
% =========================================================================
%
\section[Uniform isotropic plasmas in one spatial dimension]{%
  \for{toc}{Uniform isotropic plasmas in one spatial dimension}%
  \except{toc}{%
    Uniform isotropic plasmas in one spatial dimension:
    \texorpdfstring{\\}{} the standard linear Landau damping}%
}
\label{sec:Landau}

Here we consider in detail the case of a non-magnetized non-relativistic plasma
in one spatial dimension and for a single particle species. Coupled to the Poisson
equation this is the textbook example for linear Landau damping. Equation
(\ref{eq:linearized-vlasov}) reduces to 
\begin{equation}
  \label{eq:linearized-Vlasov-1d}
  \partial_t f(t,x,v) + v \partial_x f(t,x,v) = - (q/m) E(t,x) F'_0(v),
\end{equation}
where $F_0 \in \Schwartz(\R)$ is the equilibrium
distribution function, and $E \in \Schwartz(\R^2)$ is the electric field
perturbation. We drop the species index $\species$ since we consider one species
only. In this case the linearized Vlasov operator is the free-transport operator
\begin{equation*}
  \mathcal{L} = \partial_t  + v \partial_x.
\end{equation*}
Viewed as an operator from $\Schwartz^\prime(\R^3)$ into itself, $\mathcal{L}$
has a non-trivial null space given by all tempered distributions with partial
Fourier transform in $(t,x)$ of the form
$\hat{f} = 2\pi \kappa^* (\delta \otimes h)$ for $h \in \Schwartz^\prime(\R^2)$
and with $\kappa : (\omega, k, v) \mapsto (\omega - kv, k, v)$ being a
volume-preserving diffeomorphism of $\R^3$, and $\kappa^*$ denotes the pull-back
of distributions. Usually, this is formally written as
$\hat{f}(\omega, k, v) = 2\pi \delta(\omega - kv) \hat{h}(k,v)$, in the physics
literature. In particular, (\ref{eq:linearized-Vlasov-1d}) has infinitely many
solutions in $\Schwartz^\prime$.

We consider a damped version of the advection operator and then pass to the
limit to recover a solution of the original problem (limiting absorption
principle). More precisely, for a given $F_0\in \Schwartz(\R)$, we first prove
the existence and uniqueness of the solution $f_{\nu}$ in
$\Schwartz^\prime(\R^3)$ of~(\ref{eq:Landau-damped}). We find that the unique
solution $f_\nu \in \Schwartz^\prime(\R^3)$ is an element of $\Schwartz(\R^3)$
and we thus define a map from $\Schwartz(\R^2)$ to $\Schwartz(\R^3)$. We first
calculate the Fourier transform of $f_\nu$ in $(t,x)$
(proposition \ref{th:reg-problem}), from which we deduce the function itself. We
obtain from this unique solution the damped current 
$j_{\nu}(t,x)\coloneqq\int_{\R}vf_{\nu}(t,x,v)dv$ 
through its Fourier transform, and the limits when $\nu \to 0^+$, respectively 
in $\Schwartz^\prime(\R^3)$ and in $\Schwartz^\prime(\R^2)$ of $f_{\nu}$ 
(proposition \ref{th:limit-real-space-Landau}) and of $j_{\nu}$ 
(proposition \ref{th:Landau-limit-jnu}). 
This shows that $\lim_{\nu \to 0^+} j_{\nu} =  \sigma(E)$ where $\sigma$ 
is an operator (called the conductivity operator), for which we give expressions.  
Indeed, equality (\ref{th:limit-current-1}) below gives its pointwise limit as a
Fourier multiplier and its global expression is described in
proposition~\ref{th:limit-current-2}, which rewrites using the plasma physics 
language as proposition  \ref{th:physics-conductivity} on the object called
conductivity.

\subsection{Solution of the linearized Vlasov equation}
\label{sec:reg-problem}

In this section we prove existence and uniqueness of the solution of the
damped problem~(\ref{eq:Landau-damped}), which reads 
\begin{equation*}
  \mathcal{L}_\nu  f_\nu  = - (q/m)E F'_0, 
\end{equation*}
where $\mathcal{L}_\nu = \mathcal{L} + \nu = \partial_t  + v \partial_x + \nu$.
Let us also define the function
\begin{equation*}
  f_{\nu,*} (x,v) \coloneqq - (q/m) F'_0(v) \int_{-\infty}^0 e^{\nu s} 
  E\big(s,x + vs\big)ds,
\end{equation*}
the integral being absolutely convergent.

\begin{proposition}
  \label{th:reg-problem}
  For any $\nu > 0$, $E \in \Schwartz(\R^2)$ and $F_0\in \Schwartz(\R)$,
  equation~(\ref{eq:Landau-damped}) has a solution
  $f_\nu \in \Schwartz(\R^3)$, given by its Fourier transform
  \begin{equation}
    \label{eq:prototype1}
    \hat{f}_\nu = -i(q/m) \frac{\hat{E} F_0'}
        {\omega - kv +i \nu}.
  \end{equation}
  This is the unique solution in $\Schwartz^\prime(\R^3)$ and the unique
  (classical) solution of the Cauchy problem 
  \begin{equation*}
    \mathcal{L}_\nu  f_\nu(t,x,v)  = - (q/m)E(t,x) F'_0(v), \qquad
    f_\nu(0,x,v) = f_{\nu,*}(x,v),
  \end{equation*}
  with initial condition given at $t=0$.
\end{proposition}

\begin{remark}
  In proposition~\ref{th:reg-problem}, we distinguish between the equation in
  $\Schwartz^\prime$ and the equation stated for $C^1(\R^2)$ functions. 
\end{remark}

\begin{proof}
  For $F_0 \in \Schwartz(\R)$, $E \in \Schwartz(\R^2)$ and $\nu > 0$, after
  partial Fourier transform in $(t,x)$ we observe that necessarily a solution in
  $\Schwartz^\prime$ of the considered equation is given by~(\ref{eq:prototype1}).
  Since the function $(\omega,k,v) \mapsto (\omega -kv +i\nu)^{-1}$ belongs to
  $C^\infty$ and has polynomially bounded derivatives, from
  equation~(\ref{eq:prototype1}) we check that
  $\hat{f}_\nu \in \Schwartz (\R^3)$, and thus $f_\nu \in \Schwartz(\R^3)$. 
  However the sequence $\{f_\nu\}_\nu$ is not uniformly bounded in $\Schwartz$. 
  
  The inverse partial Fourier transform gives 
  \begin{equation}
    \label{eq:Landau-fnu-again}
    f_\nu (t,x,v) = - (q/m) F_0'(v) \int_{-\infty}^t e^{-\nu \cdot (t-s)} 
    E \big(s, x - v \cdot (t-s) \big) ds,
  \end{equation}
  and we note that
  \begin{equation*}
    X(s,t,x,v) = x - v \cdot (t - s), \quad
    V(s,t,x,v) = v,
  \end{equation*}
  is the solution of the equations for the characteristic curves of
  $\mathcal{L}_\nu$ integrated backward in time from $(t,x,v)$. Therefore
  (\ref{eq:Landau-fnu-again}) is the classical solution as claimed.
\end{proof}

\begin{remark}
  This result shows that the requirement $f_\nu \in \Schwartz^\prime (\R^3)$
  leads to the selection of a specific initial condition $f_{\nu,*}$ for the
  Cauchy problem, thus uniquely determining the solution. Conversely, the
  solution of the Cauchy problem with initial condition $f_{\nu,*}$ is an
  element of $\Schwartz$ and thus of $\Schwartz^\prime$.

  As an alternative way to illustrate how the condition
  $f_\nu \in \Schwartz^\prime (\R^3)$ leads to the selection of a specific
  Cauchy datum, one can consider for $f_{\nu,0} \in L^2(\R^2)$  the initial value
  problem in $C^1 \big(\R, L^2(\R^2)\big)$,
  \begin{equation*}
    \left\{
    \begin{aligned}
      & \partial_t f_\nu (t,x,v) + v \partial_x f_\nu (t,x,v)
      + \nu f_\nu (t,x,v) = - (q/m) E(t,x)  F_0'(v), \\
      & f_\nu (0,x,v) = f_{\nu,0} (x,v).
    \end{aligned}
    \right.
  \end{equation*}
  Performing the Fourier transform in space, we obtain an ordinary
  differential equation almost everywhere in the $(k,v)$-space,
  \begin{equation*}
    \left\{
    \begin{aligned}
      & \partial_t \tilde{f}_\nu (t,k,v) + ikv \tilde{f}_\nu (t,k,v) 
      + \nu \tilde{f}_\nu (t,k,v) = - (q/m) \tilde{E}(t,k) F_0'(v), \\
      & \tilde{f}_\nu (0,x,v) = \tilde{f}_{\nu,0} (k,v),
    \end{aligned}
    \right.
  \end{equation*}
  the solution of which is
  \begin{equation*}
    \tilde{f}_\nu (t,k,v) = e^{-(\nu + ikv)t} 
    \Big[ \tilde{f}_{\nu,0} (k,v) - (q/m) F_0'(v) \int_0^t e^{(\nu + ikv)s}
      \tilde{E}(s,k)ds \Big].
  \end{equation*}
  We see that $f_\nu, \tilde{f}_\nu \in C^\infty \big(\R, L^2(\R^2)\big)$. 
  The integral factor has a finite limit for $t \to -\infty$,
  \begin{align*}
    \tilde{f}_{\nu,*}(k,v) &= (q/m) F_0'(v) \lim_{t\to -\infty} \int_0^t e^{(\nu + ikv)s}
    \tilde{E}(s,k)ds \\
    &= -(q/m) F_0'(v) \int_{-\infty}^0 e^{(\nu + ikv)s} \tilde{E}(s,k)ds,
  \end{align*}
  and $\tilde{f}_{\nu,*} \in L^2(\R^2)$. We show that if 
  $\tilde{f}_{\nu,0} \not = \tilde{f}_{\nu,*}$, then the solution $\tilde{f}_\nu$,
  or equivalently $f_\nu$, is not tempered in time for any $k,v$ fixed. With
  this aim we write  
  \begin{equation*}
    \tilde{f}_\nu(t,k,v) = e^{-(\nu + ikv)t} \big[\tilde{f}^{(0)}_\nu (k,v)
      - \tilde{f}_{\nu,*}(k,v)\big] - \frac{qF'_0(v)}{m} \int_{-\infty}^t
    e^{-(\nu + ikv)(t-s)} \tilde{E}(s,k)ds,
  \end{equation*}
  and one observes that, for every $k,v$, the second term on the right-hand side
  belongs to $C_b^\infty(\R)$ and thus to $\Schwartz^\prime(\R)$. As for the
  the first term on the right-hand side, for any $k$, $v$, and $\nu>0$
  the function $t \mapsto e^{-\nu t +ikvt}$ is not tempered, since there exists a
  test function $\varphi(t) = e^{-\nu \sqrt{1+t^2} -ikvt}$ in $\Schwartz(\R)$
  such that $\int e^{-\nu t +ikvt} \varphi(t) dt = +\infty$. Hence the
  first term cannot be the partial Fourier transform of a distribution, which
  was the hypotheses that allowed us to write the equation in Fourier space. 
  
  It follows that we have $\tilde{f}_\nu \in \Schwartz^\prime$ if and only if
  the initial condition satisfies
  $\tilde{f}^{(0)}_\nu (k,v) = \tilde{f}_{\nu,*}(k,v)$ almost everywhere, that is,
  \begin{equation*}
    \tilde{f}_{\nu,0} (k,v) = \tilde{f}_{\nu,*}(k,v) 
    = - (q/m) F_0'(v) \int_{-\infty}^0 e^{(\nu + ikv)s} \tilde{E}(s,k)ds.
  \end{equation*}
  The corresponding solution amounts to
  \begin{equation*}
    \tilde{f}_\nu (t,k,v) = - (q/m) F_0'(v) \int_{-\infty}^t e^{-(\nu +ikv)(t-s)}
    \tilde{E}(s,k)ds,
  \end{equation*}
  and this is the unique solution in $\Schwartz^\prime (\R^3)$. Upon inserting the
  full Fourier transform of $E(t,x)$, one can check that this
  gives~(\ref{eq:Landau-fnu-again}).  
\end{remark}

We apply now the limiting absorption principle, that is we consider the limit of
the distribution function $f_\nu$ for $\nu \to 0^+$.

\begin{proposition}
  \label{th:limit-real-space-Landau}
  For any $E \in \Schwartz(\R^2)$ and $F_0 \in \Schwartz(\R)$, the solution
  $f_\nu$ defined in~(\ref{eq:Landau-fnu-again}) has a pointwise limit for
  $\nu \to 0^+$ given by 
  \begin{equation*}
    f (t,x,v) = - (q/m) F'_0(v) \int_{-\infty}^t 
    E \big(s, x - v \cdot (t-s) \big) ds,
  \end{equation*}
  which is in $C^\infty_b(\R^3)$, with $f(t,x,\cdot ) \in \Schwartz(\R)$,
  and solves $\mathcal{L}_0 f = -(q/m) E F_0'$.
\end{proposition}

\begin{proof}
  We observe that for $\nu > 0$ and $s < t$ the function
  \begin{equation*}
    s \mapsto  e^{-\nu \cdot (t-s)} E \big(s, x - v \cdot (t-s) \big),
  \end{equation*}
  is bounded by 
  \begin{equation*}
    \big|E \big(s, x - v \cdot (t-s) \big)\big| \leq \frac{1}{(1+s^2)^m} 
    \sup_{t,x} \big|(1+t^2 + x^2)^m E(t,x)\big|,
  \end{equation*}
  for all $m\geq 0$. If we choose $m > 1/2$, then $ 1/(1+s^2)^m$ is integrable
  and by the dominated convergence theorem, for any $(t,x,v) \in \R^3$,
  \begin{equation*}
    f_\nu(t,x,v) \xrightarrow{\nu\to 0^+} f(t,x,v) \coloneqq -(q/m) F'_0(v) 
    \int_{-\infty}^t E \big(s, x - v \cdot (t-s) \big) ds.
  \end{equation*}
  We observe that the pointwise limit is the causal solution of linear advection
  equation $\mathcal{L}_0 f = -(q/m) F'_0 E$ in the sense of
  appendix~\ref{sec:causal-solutions} and the characteristic flow satisfies the
  hypothesis of proposition~\ref{th:example-flow}. Hence proposition~\ref{th:causal}
  gives $f \in C_b^\infty(\R^3)$. Since $f(t,x,v)$ is proportional to $F'_0(v)$,
  we have $f(t,x,\cdot) \in \Schwartz(\R)$. 
\end{proof}

\begin{remark} 
  We can deduce other properties of the solution $f$. 
  Since $f(t,x,\cdot)$ is rapidly decreasing, we also have 
  $f(t,x, \cdot) \in L^1(\R)$ as it should be (in view of its meaning as particle
  density). Continuity implies that $f(t, \cdot, \cdot)$ is in $L^1(K \times \R)$ 
  for every compact $K \subset \R$, and this is physically appropriate for such
  an idealized model, which, being spatially uniform, has an infinite number of
  particles: only the number of particles $\|f(t,\cdot,\cdot)\|_{L^1(K\times\R)}$ in a
  compact spatial domain $K$ has to be finite. 
\end{remark}

The pointwise limit obtained in proposition~\ref{th:limit-real-space-Landau} is
referred to as the response of the plasma to the perturbation $E$.

\subsection{Current density and conductivity operator} 
\label{sec:limiting-absorption-Landau}

We can now compute the electric current density via
equation~(\ref{eq:current-density}), namely, 
\begin{equation}
  \label{eq:Landau-jnu}
  j_\nu(t,x) = q \int_\R v f_\nu (t,x,v) dv,
\end{equation}
and $j_\nu \in \Schwartz(\R^2)$. The map $E \mapsto j_\nu = \sigma_\nu(E)$ 
defines a linear continuous operator 
$\sigma_\nu : \Schwartz(\R^2) \to \Schwartz(\R^2)$ which is given by the Fourier
multiplier 
\begin{equation}
  \label{sigma-nu}
  \hat{\jmath}_\nu(\omega,k) = \hat{\sigma}_\nu(\omega,k) \hat{E}(\omega, k), \quad
  \hat{\sigma}_\nu (\omega, k) = -i\frac{q^2}{m}\int_\R 
  \frac{v F'_0(v)}{\omega - k v + i\nu} dv.
\end{equation}
The continuity of $\sigma_\nu$ in particular follows from the estimate
\begin{equation*}
  \big| \partial_\omega^\alpha \partial_k^\beta \hat{\sigma}_\nu(\omega,k) \big| \leq
  \frac{C}{\nu^{\alpha+\beta+1}} \int_\R \big| v^{\beta+1} F'_0(v) \big| dv,
\end{equation*}
for any non-negative integers $\alpha,\beta$, where the constant $C$ depends
only on $q^2/m$, $\alpha$, and $\beta$. We observe that this estimate is not
uniform in $\nu$ as expected, since the sequence $f_\nu$ is not uniformly
bounded in $\Schwartz$. 

For the slightly more general case of non-homogeneous equilibria of the form 
$F_0(x,v) = n_0(x) \tilde{F}_0(v)$ with $n_0 \in C^\infty_b$, one can define
$\tilde{\sigma}_\nu :\Schwartz(\R^2) \to \Schwartz(\R^2)$ as the Fourier multiplier
with symbol 
\begin{equation}
  \label{eq:s-nu}
  \widehat{\tilde{\sigma}}_\nu(\omega,k) =
  -i\frac{q^2}{m}\int_\R 
  \frac{v \tilde{F}'_0(v)}{\omega - k v + i\nu} dv,
\end{equation}
and obtain the induced current
\begin{equation}
  \label{eq:generalized-jnu}
  \hat{\jmath}_\nu(\omega,k) = \widehat{\tilde{\sigma}}_\nu(\omega,k)
  \widehat{n_0 E}(\omega, k).
\end{equation}
In this case the conductivity operator is
\begin{equation*}
  j_\nu = \sigma_\nu(E) \coloneqq \tilde{\sigma}_\nu(n_0 E),
\end{equation*}
and it amounts to the pseudo-differential operator
\begin{equation}
  \label{eq:jnu_psido}
  j_\nu(t,x) =
  \frac{1}{(2\pi)^2} \int e^{-i\omega (t-t') + ik (x-x')}
  n_0(x') \widehat{\tilde{\sigma}}_\nu(\omega,k) E(t',x') dt'dx' d\omega dk,
\end{equation}
where the integral is in the sense of oscillatory integrals and the symbol of
the operator is
$\hat{\sigma}_\nu(x',\omega,k) = n_0(x') \widehat{\tilde{\sigma}}_\nu(\omega,k)$. 

By using the dominated convergence theorem we have that the limit of the
current density $j_\nu$ is equal to the current carried by the limit
distribution function $f$. Specifically we have

\begin{proposition}
  \label{th:Landau-limit-jnu}
  With $E \in \Schwartz(\R^2)$ and $F_0 \in \Schwartz(\R)$, the function
  defined by  
  \begin{equation*}
    j(t,x) = -\frac{q^2}{m} \int_{D_t} v F'_0(v) 
    E \big(s, x - v \cdot (t-s) \big) ds dv, 
  \end{equation*}
  with $D_t = (-\infty,t] \times \R$, belongs to $C_b^{\infty}(\R^2)$, hence to
  $\Schwartz^\prime$. The map $\sigma : E \mapsto j = \sigma(E)$ is a linear
  continuous operator from $\Schwartz(\R^2) \to \Schwartz^\prime(\R^2)$.
\end{proposition}

\begin{proof}
  In the domain $D_t = (-\infty,t] \times \R$, we change variables to
  $s' = t-s$, $v' = v$ for $(s,v) \in D_t$, and thus
  $(s',v') \in [0,+\infty) \times \R$. Then
  \begin{equation*}
    j(t,x) = -\frac{q^2}{m} \int_0^{+\infty} \int_\R v F'_0(v) E(t-s, x-vs) dvds.
  \end{equation*}
  The integrand $\mathcal{E}(t,x,s,v) = v F'_0(v) E(t-s, x-vs)$ is such that
  \begin{equation*}
    \partial_t^\alpha \partial_x^\beta \mathcal{E}(t,x,s,v) = vF'_0(v)
    \partial_t^\alpha \partial_x^\beta E(t-s, x-vs) 
    \leq \frac{|vF'_0(v)|}{\big(1+(t-s)^2\big)^m} \|E\|_{\alpha + \beta + 2m}.
  \end{equation*}
  For any $t \in \R$, this upper bound belongs to $L^1(\R^2)$. By the dominated
  convergence theorem, we deduce that $j \in C_b^\infty(\R^2)$ and
  $|\partial_t^\alpha \partial_x^\beta j(t, x)| \leq C_m \|E\|_{\alpha + \beta + 2m}$.

  As $C_b^\infty(\R^2) \subset \Schwartz^\prime(\R^2)$ and
  $\|j\|_{L^\infty(\R^2)} \leq C \|E\|_{2m}$, one has that the map 
  $E \mapsto j$ from $\Schwartz \to \Schwartz^\prime$ is continuous.
  %%
  %% Let us first show that $j \in C_b^\infty(\R^2)$. With this aim we consider the
  %% function 
  %% \begin{equation*}
  %%   \mathcal{E}(t,x,s,v) = v F'_0(v) E \big(s, x - v \cdot (t-s) \big),
  %% \end{equation*}
  %% for $(t,x) \in \R^2$ and $(s,v) \in D_t$. Then $\mathcal{E}$ is of class
  %% $C^\infty$ and for all $\alpha, \beta,m \in \N_0$ 
  %% we have
  %% \begin{align*}
  %%   |\partial_t^\alpha \partial_x^\beta \mathcal{E}(t,x,s,v) | &\leq
  %%   |v^{\alpha+1} F_0'(v)| \cdot |\partial^{\alpha+\beta}_x E(s, x - v(t-s))| \\
  %%   &\leq \frac{|v^{\alpha+1} F_0'(v)|}{(1+s^2)^m}
  %%   \sup_{t,x} \big|(1+t^2 + x^2)^m \partial_x^{\alpha+\beta} E(t, x)\big| \\
  %%   &\leq \frac{|v^{\alpha+1} F_0'(v)|}{(1+s^2)^m} \|E\|_{\alpha+\beta+2m}.
  %% \end{align*}
  %% where $\|E\|_k$ denotes the standard norms in $\Schwartz$, cf. section
  %% \ref{sec:basic-def}. For $m > 1/2$ this upper bound is integrable on $\R^2$
  %% and thus we can differentiate in the integral and obtain
  %% \begin{equation*}
  %%   \big|\partial_t^\alpha \partial_x^\beta j(t,x)\big| \leq
  %%   \Big(\int_{\R^2} \frac{|v^{\alpha+1} F_0'(v)|}{(1+s^2)^m} ds dv \Big)
  %%   \|E\|_{\alpha+\beta+2m}, \quad \forall (t,x) \in \R^2,
  %% \end{equation*}
  %% uniformly in $\R^2$. In the upper bound we have extended the integration
  %% domain from $D_t$ to the whole space $\R^2$. 
  %% This proves the claim $j \in C_b^\infty(\R^2)$. 
  %%
  %% As $C_b^\infty(\R^2) \subset \Schwartz^\prime(\R^2)$ and one has
  %% $\|j\|_{L^\infty(\R^2)} \leq C \|E\|_{2m}$, that the map 
  %% $E \mapsto j$ from $\Schwartz \to \Schwartz^\prime$ is continuous.
\end{proof}

\begin{proposition}
  \label{th:Landau-limit-jnu-bis}
  When $\nu \to 0^+$, the sequence $j_\nu$ defined in equation~(\ref{sigma-nu})
  converges to $j$ both pointwise in $\R^2$ and in $\Schwartz^\prime(\R^2)$.
\end{proposition}

\begin{proof}
  With $\mathcal{E}(t,x,s,v) = v F'_0(v) E \big(s, x - v \cdot (t-s) \big)$
  as in the proof of proposition~\ref{th:Landau-limit-jnu}, we have
  \begin{equation*}
    j_\nu(t,x) - j(t,x) = \frac{q^2}{m} \int_{D_t} \big(1 - e^{-\nu(t-s)}\big)
    \mathcal{E}(t,x,s,v) dsdv.
  \end{equation*}
  Since $e^{-\nu (t-s)} \leq 1$ for $(s,v) \in D_t$, we have 
  \begin{equation*}
    \big(1 - e^{-\nu (t-s)}\big) \big|\mathcal{E}(t,x,s,v) \big| \leq 
    \frac{\|F_0\|_{2m_2+2} \|E\|_{2m_1}}{(1+s^2)^{m_1} (1+v^2)^{m_2}} 
  \end{equation*}
  and for $m_1, m_2 >1/2$ the bound is in $L^1$. The dominated convergence
  theorem then yields pointwise convergence:
  $\lim \big(j_\nu(t,x) - j(t,x)\big) = 0$ for all $(t,x) \in \R^2$. The same
  estimate also gives convergence in $\Schwartz^\prime(\R^2)$: for any test
  function $\phi \in \Schwartz(\R^2)$,
  \begin{equation*}
    \langle j_\nu - j, \phi \rangle =
    \frac{q^2}{m} \int_{\R^2} \int_{D_t} \big(1 - e^{-\nu(t-s)}\big)
    \mathcal{E}(t,x,s,v) dsdv\, \phi(t,x)\, dtdx,
  \end{equation*}
  and the integrand is bounded by
  $\|F_0\|_{2m_2+2} \|E\|_{2m_1} (1+s^2)^{-m_1} (1+v^2)^{-m_2} |\phi(t,x)|$
  which is integrable. Again the dominated convergence theorem allows us to pass
  to the limit in the integral and obtain $\langle j_\nu - j, \phi \rangle \to 0$.
\end{proof}

For an explicit calculation of the conductivity operator we consider the limit
$\nu \to 0^+$ in Fourier space. As a tempered distribution, $\hat{\jmath}_\nu$
acts on $\psi \in \Schwartz(\R^2)$ by
\begin{equation}
  \label{eq:Landau-jnu-hat}
  \langle \hat{\jmath}_\nu, \psi \rangle = 
  - i (q^2 / m) \int_{\R^2 \times \R} vF_0'(v)
  \frac{ \hat{E}(\omega, k) \psi(\omega,k)}{\omega - kv + i \nu}
  d\omega dk dv.
\end{equation}

We now want to pass to the limit for $\nu \to 0^+$. We use the Hilbert transform 
(appendix \ref{sec:calc-limits-proof}).

Let $G(v) = v F_0'(v)/n_0$ where $n_0>0$ is the uniform background plasma
density and let $\omega_p = \sqrt{4\pi q^2 n_0 /m}$ be the plasma frequency of the
considered species. For $(\omega,k)\in\R^2$, $k \not= 0$, let
\begin{equation}
  \label{eq:sigma-ph}
  \hat{\sigma}_{\ph} (\omega, k) \coloneqq - i \frac{\omega_p^2}{4\pi} 
  \frac{1}{k} \Big[ \pi\mathcal{H}(G)(\omega / k ) -i\pi G(\omega/k) \Big],
\end{equation}
where $\mathcal{H}(G)$ is the Hilbert transform of $G$.
The tensor $\hat{\sigma}_{\ph}$ is the same as the one obtained formally in the
physics literature. We have 
$4\pi i \omega \hat{\sigma}_{\ph}(\omega,k) = \omega_p^2 H(\omega/k)$ for $k \not=0$,
where the function $H \in C^{\infty}(\R)$ is given by
$H(z) = z[\pi\mathcal{H}(G)(z) - i\pi G(z)]$. Then we have
\begin{equation}
  \label{th:limit-current-1}
  \lim_{\nu \to 0^+} \hat{\sigma}_\nu = \hat{\sigma}_\ph
  \text{ pointwise in } (\omega, k) \in \R^2, \; k \not=0,
\end{equation}
since $\hat{\sigma}_\nu(\omega,k) = -i(\omega_p^2/4\pi) A_\nu(\omega,k)$,
where $A_\nu$ is defined in appendix~\ref{sec:calc-limits-proof}, and
equation~(\ref{th:limit-current-1}) follows from lemma~\ref{th:Anu-Bnu}. 

Let us introduce $\lambda >0$, a cut-off function $\chi \in C_0^\infty (\R)$, 
$\chi (z) = 1$  for $|z| \leq 1/2$ and $\supp\chi \subset (-1,1)$, and let
$\chi_\lambda (k) = \chi(\lambda k)$. The real number $\lambda$ can be
interpreted as a scale-length in the frequency domain. Then we define two
operators $\sigma_{\lambda,1-\chi}, \sigma_{\lambda,\chi} : \Schwartz(\R^2) \to
\Schwartz^\prime(\R^2)$ given by
\begin{equation}
  \label{eq:sigma-lambda-chi}
  \sigma_{\lambda,1-\chi}(E) \coloneqq \mathcal{F}^{-1} \big((1-\chi_\lambda)
  \hat{\sigma}_\ph \hat{E} \big), \qquad
  \sigma_{\lambda,\chi}(E) \coloneqq \sigma - \sigma_{\lambda,1-\chi}.
\end{equation}
We note that $\sigma_{\lambda,1-\chi}$ is continuous since it is the
composition of continuous operations, and we have shown in
proposition~\ref{th:Landau-limit-jnu} that $\sigma$ is continuous,
therefore $\sigma_{\lambda,\chi}$ is continuous.

\begin{proposition}
  \label{th:limit-current-2}
  The operators $\sigma$ and $\sigma_{\lambda,\chi}$ defined in proposition
  \ref{th:Landau-limit-jnu} and equation~(\ref{eq:sigma-lambda-chi}),
  respectively, are such that
  \begin{align*}
    \langle \mathcal{F}\big(\sigma_{\lambda,\chi} (E)\big), \hat{\psi} \rangle
    &= i \frac{\omega_p^2}{4\pi} \int_{\R^2} \chi_\lambda(k) G(v)
    \Big[\pi\mathcal{H}\big(\hat{E}\hat{\psi}(\cdot, k)\big)(kv) 
      + i \pi \hat{E}\hat{\psi} (kv,k) \Big] dk dv, \\
    \langle \mathcal{F}\big(\sigma (E)\big), \hat{\psi} \rangle
    &= i \frac{\omega_p^2}{4\pi} \int_{\R^2} G(v)
    \Big[\pi\mathcal{H}\big(\hat{E}\hat{\psi}(\cdot, k)\big)(kv) 
      + i \pi \hat{E}\hat{\psi} (kv,k) \Big] dk dv,    
  \end{align*}
  for all $E, \psi \in \Schwartz(\R^2)$.
\end{proposition}

\begin{proof}
  Let $G(v) = v F_0'(v) / n_0$, $\omega_p^2 = 4\pi q^2 n_0/m$, and for any
  $\phi \in \Schwartz(\R^2)$, let $A_\nu$ and $B_\nu$ be the integrals defined
  in appendix~\ref{sec:calc-limits-proof}. Let
  \begin{align*}
    I^\nu_{\lambda,1-\chi}(\phi) &\coloneqq - i \frac{\omega_p^2}{4\pi} \int_{\R^3}
    \big(1-\chi_\lambda (k)\big)  G(v) 
    \frac{ \phi(\omega, k) }{\omega - kv + i \nu}
    d\omega dk dv \\
    &= - i \frac{\omega_p^2}{4\pi} \int_{\R^2}
    \big(1-\chi_\lambda (k)\big) \phi(\omega,k) A_\nu(\omega,k) d\omega dk, \\
    I^\nu_{\lambda,\chi}(\phi) &\coloneqq - i \frac{\omega_p^2}{4\pi} \int_{\R^3}
    \chi_\lambda (k)  G(v) \frac{ \phi(\omega, k) }{\omega - kv + i \nu}
    d\omega dk dv \\
    &= - i \frac{\omega_p^2}{4\pi} \int_{\R^2} \chi_\lambda (k)  G(v) B_\nu(v,k) dkdv.
  \end{align*}
  Since $k \not= 0$ in the support of $1-\chi_\lambda$, the integrands have a
  pointwise limit as $\nu \to 0^+$ computed in lemma~\ref{th:Anu-Bnu}. In
  addition, lemma~\ref{th:Anu-Bnu} shows that the integrands are bounded by a
  function in $L^1$ uniformly in $\nu$. The dominated convergence theorem
  applies and we can pass to the limit $\nu \to 0^+$ in the integrals, obtaining
  \begin{align*}
    I^\nu_{\lambda,1-\chi}(\phi) &\to
    - i \frac{\omega_p^2}{4\pi} \int_{\R^2}
    \big(1-\chi_\lambda (k)\big) \phi(\omega,k)
    \frac{1}{k}\Big[\pi \mathcal{H}(G)(\omega/k)
      -i\pi G(\omega/k)\Big]d\omega dk, \\
    I^\nu_{\lambda,\chi} (\phi) &\to
    - i \frac{\omega_p^2}{4\pi} \int_{\R^2} 
    \chi_\lambda (k)  G(v) \Big[
      -\pi \mathcal{H}\big(\phi(\cdot,k) \big)(kv)
      - i\pi \phi(kv,k) \Big] dk dv.
  \end{align*}
  Particularly, $I^\nu_{\lambda,1-\chi}(\phi) \to \int \big(1-\chi_\lambda (k)\big)
  \hat{\sigma}_\ph(\omega,k) \phi(\omega,k) d\omega dk$.
  Then we have
  \begin{align*}
    \langle \hat{\jmath}_\nu, \hat{\psi} \rangle &=
    I^\nu_{\lambda,1-\chi}(\hat{E} \hat{\psi}) +
    I^\nu_{\lambda,\chi}(\hat{E} \hat{\psi}) 
    \xrightarrow{\nu \to 0^+}
    \big\langle \mathcal{F}\big(\sigma_{\lambda,1-\chi}(E)\big), \hat{\psi} \big\rangle\\
    &\qquad + i \frac{\omega_p^2}{4\pi} \int_{\R^2} \chi_\lambda(k) G(v)
    \Big[\pi\mathcal{H}\big(\hat{E}\hat{\psi}(\cdot, k)\big)(kv) 
      + i \pi \hat{E}\hat{\psi} (kv,k) \Big] dk dv.
  \end{align*}
  On the other hand we know from proposition~\ref{th:Landau-limit-jnu-bis} that
  as $\nu \to 0^+$, $j_\nu \to j = \sigma(E)$ in $\Schwartz^\prime$, hence
  \begin{equation*}
    \langle \hat{\jmath}_\nu, \hat{\psi} \rangle \xrightarrow{\nu \to 0^+}
    \big\langle \mathcal{F}\big(\sigma(E)\big), \hat{\psi} \big\rangle,
  \end{equation*}
  and by definition $\sigma(E) = \sigma_{\lambda,1-\chi}(E) + \sigma_{\lambda,\chi}(E)$.
  Uniqueness of the limit gives the claimed expression for $\sigma_{\lambda,\chi}$.

  The claimed expression for $\sigma(E)$ follows analogously on noting that
  \begin{equation*}
    \langle \hat{\jmath}_\nu, \hat{\psi} \rangle =
    I^\nu(\hat{E} \hat{\psi}) \coloneqq 
    - i \frac{\omega_p^2}{4\pi} \int_{\R^2} G(v) B_\nu(v,k) dkdv,
  \end{equation*}
  where the function $B_\nu$ is now computed with $\phi = \hat{E} \hat{\psi}$,
  i.e.,
  \begin{equation*}
    B_\nu(\omega,k) = \int_\R \frac{\hat{E}(\omega,k) \hat{\psi}(\omega,k)
    }{\omega - kv +i\nu} d\omega.
  \end{equation*}
  As $\nu \to 0^+$ the right-hand side converges to
  $\langle \mathcal{F}\big(\sigma (E)\big), \hat{\psi} \rangle$, while the limit
  of the left-hand side is dealt with as in the case of $I^\nu_{\lambda,\chi}$.
\end{proof}

The operator $\sigma_{\lambda, \chi}$ does not play any role when the electric
field perturbation is supported away from $k=0$, i.e., for non-constant
fields. More precisely we have the following result, which expresses the usual
Ohm's law for a uniform plasma. 

\begin{corollary}
  \label{th:physics-conductivity}
  If $E \in \Schwartz(\R^2)$ is such that $\hat{E} (\omega,k) = 0$ for
  $|k| \leq 1/\lambda$ , then $\hat{\jmath} \in C^\infty (\R^2)$ and
  $\hat{\jmath}(\omega,k) = \hat{\sigma}_{\mathrm{ph}} (\omega,k) \hat{E}(\omega,k)$.
\end{corollary}

\begin{proof}
  By hypothesis $\chi_\lambda(k) \hat{E}(\omega, k) = 0$ for all 
  $(\omega,k) \in \R^2$, hence, $\sigma_{\lambda, \chi} (E) =0$; this follows
  directly from the expression given in proposition~\ref{th:limit-current-2}
  since, in particular, $\chi_\lambda(k)
  \mathcal{H}\big(\hat{E}\hat{\psi}(\cdot, k)\big)(kv) =
  \mathcal{H}\big(\chi_\lambda \hat{E}\hat{\psi}(\cdot, k)\big)(kv)$. Then
  \begin{equation*}
    \mathcal{F}\big(\sigma(E)\big) = \mathcal{F}\big(\sigma_{\lambda, 1-\chi}(E)\big)
    = (1 - \chi_\lambda) \hat{\sigma}_{\mathrm{ph}} \hat{E}
    = \hat{\sigma}_{\mathrm{ph}} \hat{E},
  \end{equation*}
  since $(1 - \chi_\lambda) \hat{E} = \hat{E}$.
  The fact that $\hat{\jmath}$ is in $C^\infty$ follows from the properties of the
  Hilbert transform summarized in proposition~\ref{th:Hilbert-transform}, that
  imply in particular, $\mathcal{H}(G) \in H^\infty (\R)$.
\end{proof}

\begin{remark}
  \label{rem:PsiDO}
  In the case of non-homogeneous equilibria of the form 
  $F_0(x,v) = n_0(x) \tilde{F}_0(v)$, the statement of
  corollary~\ref{th:physics-conductivity} remains true with $G$ and $\hat{E}$
  replaced by $v \tilde{F}_0(v)$ and $\widehat{n_0 E}$,
  respectively. Particularly, one has
  \begin{equation*}
    \mathcal{F}\big(\sigma_{\lambda, 1-\chi} (E) \big) (\omega,k) = 
    \big(1 - \chi_\lambda (k)\big) \widehat{\tilde{\sigma}}_{\mathrm{ph}} (\omega,k) 
    \widehat{n_0E}(\omega,k),
  \end{equation*}
  where $\widehat{\tilde{\sigma}}_{\mathrm{ph}}$ is obtained from~(\ref{eq:s-nu}) in
  analogy with $\hat{\sigma}_{\mathrm{ph}}$. Then, the operator
  \begin{multline*}
    \sigma_{\lambda, 1-\chi} (E)(t,x) 
    = \frac{1}{(2\pi)^2} \int e^{-i \omega (t-t') + i k (x-x')}\\
    \times [n_0(x') \big(1 -\chi_\lambda(k) \big)
      \widehat{\tilde{\sigma}}_{\mathrm{ph}}(\omega, k)]  
    E(t',x')dt'dx' d\omega dk
  \end{multline*}
  is pseudo-differential with symbol
  $n_0(x') \big(1 -\chi_\lambda(k) \big)
  \widehat{\tilde{\sigma}}_{\mathrm{ph}}(\omega, k)$, 
  recovering an expression similar to~(\ref{eq:jnu_psido}).   
\end{remark}

\subsection[Proofs for the results of section~\ref{sec:main-1}]{%
  \for{toc}{Proofs for the results of section~\ref{sec:main-1}}%
  \except{toc}{%
  Proof of the main results for the non-magnetized non-relativistic 
  one-dimensional case (section~\ref{sec:main-1})}%
}

We collect at last the partial results of this section and give the proofs of
the two theorems stated in section~\ref{sec:main-1}.

\begin{proof}[Proof of theorem~\ref{th:main-1-1}]
  (i) The fact that $f_\nu$ belongs to $\Schwartz(\R^3)$ and is the unique 
  solution of equation~(\ref{eq:Landau-damped}) in $\Schwartz^\prime(\R^3)$ is
  proven in proposition~\ref{th:reg-problem}. The current density $j_\nu$ and the
  operator $\sigma_\nu$ are given in equation~(\ref{eq:Landau-jnu})
  and~(\ref{sigma-nu}) and related comments.
  
  (ii) and (iii) Pointwise convergence $f_\nu \to f$ is established in
  proposition~\ref{th:limit-real-space-Landau} and an expression for the solution
  $f$ is given there. In proposition~\ref{th:limit-real-space-Landau}, it is
  also proven that $f(t,x, \cdot) \in \Schwartz(\R)$ for every $(t,x) \in\R^2$.
  As for the convergence of $f_\nu \to f$ in the topology of $\Schwartz^\prime$,
  proposition~\ref{th:limit-real-space-Landau} establishes pointwise (but not
  uniform) convergence $f_\nu \to f$ with limit $f \in C^\infty_b$; in addition,
  for any integer $m \geq 1/2$, we have
  \begin{equation*}
    \big|f_\nu(t,x,v) \big| \leq \| qF_0'/m\|_0 \cdot \|E\|_{2m}
    \int_{-\infty}^{+\infty} \frac{ds}{(1+s^2)^m},
  \end{equation*}
  uniformly in $\nu \in [0, +\infty)$. Therefore for every
  $\varphi \in \Schwartz(\R^3)$, the function $(f_\nu - f)\varphi$ satisfies
  \begin{equation*}
    \big|f_\nu(t,x,v) - f(t,x,v)\big| \big|\varphi(t,x,v)\big| \leq
    C \big|\varphi(t,x,v)\big|.
  \end{equation*}
  As $|\varphi| \in L^1$, $(f_\nu - f)\varphi$ satisfies the hypothesis of the
  dominated convergence theorem and 
  \begin{equation*}
    \langle f_\nu - f, \phi \rangle
    = \int_{\R^3} (f_\nu - f) \phi dtdxdv \to 0, \quad
    \text{for all $\phi \in \Schwartz(\R^3)$.}
  \end{equation*}
  In proposition~\ref{th:Landau-limit-jnu}, it is shown that
  $j = \mathbb{J}(f) \in C^\infty_b(\R^2)$,
  and proposition~\ref{th:Landau-limit-jnu-bis} establishes the limit
  $j_\nu \to j$ both pointwise and in $\Schwartz^\prime(\R^2)$.
\end{proof}

\begin{proof}[Proof of theorem~\ref{th:main-1-2}]
  Proposition \ref{th:Landau-limit-jnu} also establishes the continuity of
  the operator $\sigma: E \mapsto j$. The relation to the physical conductivity
  operator is proven in corollary~\ref{th:physics-conductivity}.
\end{proof}

% New section
% =========================================================================
%

\section[Study of an auxiliary problem]{\for{toc}{PDE of the form 
    \texorpdfstring{$-(u_1 \partial_{u_2} - u_2 \partial_{u_1})\varphi(\theta,u)
      - ia(\theta,u) \varphi(\theta,u) = \psi(\theta,u)$}{}  
  }%
  \except{toc}{Study of a PDE with parameters of the form 
    \texorpdfstring{$-(u_1 \partial_{u_2} - u_2 \partial_{u_1})\varphi(\theta,u)
      - ia(\theta,u) \varphi(\theta,u) = \psi(\theta,u)$}{}
  }%
}
\label{sec:auxiliary-pde}

In this section we establish existence and uniqueness results for a partial
differential equation with parameters that arises in the study of the
relativistic, linear Vlasov equation with uniform magnetic field $B_0$, 
addressed below in section~\ref{sec:uniform}. Specifically the equation is 
\begin{equation}
  \label{eq:psi-psi-eq}
  -(u_1 \partial_{u_2} - u_2 \partial_{u_1})\varphi (\theta,u)
  - ia(\theta,u) \varphi (\theta,u) = \psi(\theta,u),
\end{equation}
where $a, \psi \in C^\infty(\R^l \times \R^3)$ are given complex-valued
functions of $u\in\R^3$ and depend on parameters $\theta \in \R^l$.
The operator $-(u_1 \partial_{u_2} - u_2 \partial_{u_1})$
originates from the Lorentz force term $q(v \times B_0) \cdot \nabla_p$ with
$q>0$ and $B_0$ constant and directed along the third axis.
Eventually, the parameters $\theta$ will be related to the Fourier variables
$(\tau,\xi)$, and $a$ to $a_\eps$ defined in equation~(\ref{eq:a-and-bs}) of
section~\ref{sec:uniform} below. Therefore, we assume that $a$ satisfies a
condition similar to (\ref{eq:nu-conditions}), that is,
\begin{subequations}
  \label{eq:general-a-condition}
  \begin{equation}
    \label{eq:minimal-a-condition}
    a \in C^\infty(\R^l\times \R^3,\C), \quad
    (u_1 \partial_{u_2} - u_2 \partial_{u_1}) a(\theta,u)=0,
    \quad
    \big| \imag a(\theta,u) \big| \geq \eta >0,
  \end{equation}
  for a given constant $\eta > 0$. As we need to control the growth of
  derivatives of the solution at infinity, we shall also assume that
  \begin{equation}
    \label{eq:growth-a-condition}
    |\partial^\alpha_\theta \partial_u^\beta a(\theta,u))| \leq
    C_{\alpha,\beta} (1+\theta^2+u^2)^m,
    \quad \forall \alpha \in \N_0^l,
    \quad \forall \beta \in \N_0^3,
  \end{equation}
  uniformly in $(\theta,u) \in \R^l \times \R^3$ for a given $m \in \R$ and with
  constants $C_{\alpha,\beta} >0$ depending on the multi-indices. 
\end{subequations}

First we establish the uniqueness of the solution under rather general
conditions. 

\begin{lemma}
  \label{th:phi-uniqueness}
  Let $a \in C^\infty(\R^l \times \R^3)$ satisfy
  condition~(\ref{eq:minimal-a-condition}), $\Theta \subseteq \R^l$ be an open
  set, and let $\varphi \in L^2_{\mathrm{loc}}(\Theta \times \R^3)$ be a function with 
  weak derivatives $\partial_{u_1} \varphi, \partial_{u_2} \varphi \in
  L^2_{\mathrm{loc}}(\Theta \times \R^3)$ and such that
  \begin{equation*}
    -(u_1 \partial_{u_2} - u_2 \partial_{u_1})\varphi -i a \varphi = 0,
    \quad \text{a.e. in } \Theta \times \R^3.
  \end{equation*}
  Then, $\varphi = 0$ a.e. in $\Theta \times \R^3$.
\end{lemma}

\begin{proof}
  For almost all $(\theta,u_3) \in \Theta \times \R$, the function
  $\tilde{\varphi}(u_1,u_2) \coloneqq \varphi(\theta,u_1,u_2,u_3)$ belongs to
  $H^1\big(B_{r}(0)\big)$ for every $r > 0$, where $B_r(0) \subset \R^2$ is the
  open ball of radius $r$ and centered in zero in $\R^2$. From the equation we deduce
  \begin{equation*}
    -(u_1 \partial_{u_2} - u_2 \partial_{u_1}) |\tilde{\varphi}|^2 +
    2 \imag(a) |\tilde{\varphi}|^2 = 0.
  \end{equation*}
  The first term amounts to the divergence of the vector field
  $(-u_2,u_1) |\tilde{\varphi}|^2$ which is tangent to $\partial B_r(0)$, hence
  Gauss theorem for the divergence, which holds for $H^1$ functions, gives
  \begin{equation*}
    0 = \int_{B_r(0)} (u_1 \partial_{u_2} - u_2 \partial_{u_1})
    |\tilde{\varphi}|^2 du_1 du_2 = 2 \int_{B_r(0)} \imag{a}\;
    |\tilde{\varphi}|^2 du_1 du_2,
  \end{equation*}
  for every radius $r > 0$. We can now conclude upon accounting for
  hypotheses~(\ref{eq:minimal-a-condition}). If $\imag a \geq \eta >0$, we have
  \begin{equation*}
    0 \leq \eta \int_{B_r(0)} |\tilde{\varphi}|^2 du_1 du_2 \leq
    \int_{B_r(0)} \imag a\; |\tilde{\varphi}|^2 du_1 du_2 = 0.
  \end{equation*}
  If instead $-\imag a \geq \eta > 0$,
  \begin{equation*}
    0 \leq \eta \int_{B_r(0)} |\tilde{\varphi}|^2 du_1 du_2 \leq
    -\int_{B_r(0)} \imag a\; |\tilde{\varphi}|^2 du_1 du_2 = 0.
  \end{equation*}
  In both cases we deduce
  \begin{equation*}
    \int_{B_r(0)} |\tilde{\varphi}|^2 du_1 du_2 = 0,
  \end{equation*}
  and thus $\tilde{\varphi} = 0$ a.e. in $B_r(0)$ for all $r$ and for
  almost all $(\theta,u_3) \in \Theta \times \R$.
  It follows that $\varphi=0$ a.e. in $\Theta \times \R^3$.
\end{proof}

In the remaining part of this section, we first give an existence result for the
case in which the source term $\psi$ is a polynomial in $(u_1,u_2)$; this is
based on an algebraic argument. Then, we prove the existence of a smooth solution
$\varphi \in C^\infty$ when $\psi \in C^\infty$ and of a solution
$\varphi \in \Schwartz$ when $\psi \in \Schwartz$. The latter implies
uniqueness of the solution in $\Schwartz^\prime$.

\subsection{Equation with a polynomial source term}
\label{sec:equat-with-polynomial-rhs}

Let the source term in equation~(\ref{eq:psi-psi-eq}) be a polynomial of the form
\begin{equation}
  \label{eq:psi-polynomial}
  \psi(\theta, u) = \sum_{0 \leq m+n \leq L} \mathcal{Y}_{m,n}(\theta, u_3)  u_1^m u_2^n,
  \quad \mathcal{Y}_{m,n} \in C^\infty,
\end{equation}
and let us consider for $z \in \C\setminus \Z$ the equation
\begin{equation}
  \label{eq:psi-psi-z}
  -(u_1 \partial_{u_2} - u_2 \partial_{u_1})\tilde{\varphi}(z;\theta,u)
  - iz \tilde{\varphi}(z;\theta,u) = \psi(\theta,u).
\end{equation}
We can search for solutions of the form
\begin{equation}
  \label{eq:phi-polynomial}
  \tilde{\varphi}(z; \theta,u)
  = \sum_{0 \leq m+n \leq L} \mathcal{X}_{m,n} (z;\theta, u_3) u_1^m u_2^n,
\end{equation}
that is, a polynomial with at most the same degree as the source
term. Substitution into (\ref{eq:psi-psi-z}) yields
\begin{multline*}
  -\sum_{m=1} \sum_{n=0} (n+1) \mathcal{X}_{m-1,n+1} u_1^m  u_2^n
  +\sum_{m=0} \sum_{n=1} (m+1) \mathcal{X}_{m+1,n-1} u_1^m u_2^n \\
  - iz \sum_{m=0} \sum_{n=0} \mathcal{X}_{m,n} u_1^m u_2^n
  = \sum_{m=0} \sum_{n=0} \mathcal{Y}_{m,n} u_1^m u_2^n,
\end{multline*}
where the sums are all finite since $m+n\leq L$. We observe that only the
coefficients $\mathcal{X}_{m,n}$ with $m + n = \ell$ for
$\ell = 0,1,2,\ldots$ are coupled. For every integer $0 \leq \ell \leq L$,
we define
\begin{equation*}
  x_\ell = (\mathcal{X}_{\ell-j,j})_{j=0}^\ell, \qquad
  y_\ell = (-i\mathcal{Y}_{\ell-j,j})_{j=0}^\ell, 
\end{equation*}
then the linear equation for the coefficients splits into $(\ell+1)$-dimensional
blocks of the form
\begin{equation}
  \label{eq:linear-algebraic-eq}
  (A_\ell - z) x_\ell = y_\ell, \qquad 0 \leq \ell \leq L,
\end{equation}
where the matrix $A_\ell \in \C^{(\ell+1)\times(\ell+1)}$ is defined and given in
appendix~\ref{sec:matrices}. For each $\ell$, equation~(\ref{eq:linear-algebraic-eq})
has a unique solution when $z$ is not an eigenvalue of the matrix $A_\ell$.
Lemma~\ref{th:matrix-Ak} shows that the spectrum of $A_\ell$ is given by
$\{2s - \ell :\; s = 0,1,\ldots,\ell\}$ and it is  contained in the set of
relative integers $\Z$ for any $\ell$. Hence, if $z \in \C \setminus \Z$
equation~(\ref{eq:linear-algebraic-eq}) has a unique solution for all $\ell$. 

\begin{lemma}
  \label{th:psi-psi-lemma}
  If $z \in \C\setminus \Z$ and $\psi$ is given by (\ref{eq:psi-polynomial}),
  equation~(\ref{eq:psi-psi-z}) has a solution which is of the
  form~(\ref{eq:phi-polynomial}) with  
  $\mathcal{X}_{m,n} (\cdot;\theta, u_3)$ analytic in $\C \setminus \Z$, 
  and $\mathcal{X}_{m,n} (z;\cdot) \in C^\infty(\R^l \times \R)$.
\end{lemma}

\begin{proof}
  Lemma \ref{th:matrix-Ak} establishes that the matrix $A_\ell$ is
  diagonalizable with eigenvalues $2s - \ell$, $s=0,\ldots, \ell$. We denote by
  $S$ and $T = S^{-1}$ the matrices (explicitly given in the proof of
  lemma~\ref{th:matrix-Ak}) such that $T A_\ell S$ is diagonal.
  Then, for $z \in \C \setminus \Z$, $A_\ell - z$ is invertible and
  equation~(\ref{eq:linear-algebraic-eq}) has a unique solution
  $x_\ell = (\mathcal{X}_{m,n})_{m+n=\ell}$ given by
  \begin{equation*}
    \mathcal{X}_{\ell-j,j}(z;\theta,u_3) =
    \sum_{r,s=0}^\ell \frac{S_{jr} T_{rs}}{2r-\ell -z}
    \mathcal{Y}_{\ell-s,s}(\theta,u_3),
  \end{equation*}
  which is analytic in $z \in \C \setminus \Z$, and $C^\infty$ in
  $(\theta,u_3)$. 
\end{proof}

We can now use $\tilde{\varphi}$ to construct the unique solution
of~(\ref{eq:psi-psi-eq}). 

\begin{proposition}
  \label{th:psi-psi-polynomial}
  Let $a \in C^\infty(\R^l \times \R^3,\C)$ satisfy the
  condition~(\ref{eq:minimal-a-condition}), $\psi$ be given in the form
  (\ref{eq:psi-polynomial}), $\tilde{\varphi}$ be the solution
  established in lemma~\ref{th:psi-psi-lemma}, and let
  \begin{equation*}
    \varphi (\theta,u) \coloneqq \tilde{\varphi}\big( a(\theta,u); \theta,u).
  \end{equation*}
  Then $\varphi \in C^\infty(\R^l \times \R^3)$ is the unique solution
  of (\ref{eq:psi-psi-eq}). 
\end{proposition}
  
\begin{proof}
  The fact that $\varphi \in C^\infty$ is a solution follows by
  equation~(\ref{eq:psi-psi-z}) and assumption~(\ref{eq:minimal-a-condition})
  which in particular implies
  \begin{equation*}
    -(u_1 \partial_{u_2} - u_2 \partial_{u_1}) \varphi(\theta,u) = 
    -(u_1 \partial_{u_2} - u_2 \partial_{u_1}) 
    \tilde{\varphi}(z;\theta,u)|_{z=a(\theta,u)}.
  \end{equation*}
  Uniqueness has been proven in lemma~\ref{th:phi-uniqueness}. 
\end{proof}

\subsection{Equation with source term in \texorpdfstring{$\Schwartz$}{} and
  uniqueness in \texorpdfstring{$\Schwartz^\prime$}{}.} 
\label{sec:equation-with-source-in-S}

If $\varphi \in C^1$ is a solution of~(\ref{eq:psi-psi-eq}) and
$(u_\perp,\phi) \in \R_+ \times [0,2\pi]$ are polar coordinates defined by 
$u_1 = u_\perp \cos\phi$, $u_2 = - u_\perp \sin\phi$, then the functions 
\begin{align*}
  U(r,\phi) &= \varphi(\theta, u_\perp \cos\phi, - u_\perp \sin\phi,u_3),
  \\
  V(r,\phi) &= \psi(\theta, u_\perp \cos\phi, - u_\perp \sin\phi, u_3),
\end{align*}
with parameter $r = (\theta, u_\perp, u_3)$ satisfy
\begin{equation}
  \label{eq:ode}
  \partial_\phi U(r,\phi) - i\tilde{a}(r) U(r,\phi) = V(r,\phi),
  \quad U(r,0) = U(r,2\pi),
\end{equation}
with
\begin{equation*}
  \tilde{a}(r) = a(\theta, u_\perp \cos\phi, - u_\perp \sin\phi, u_3),
\end{equation*}
which is independent of $\phi$ because of condition~(\ref{eq:minimal-a-condition}).

\begin{remark}
  \label{rem:orientation-phi}
  The choice of the angle $\phi$, in the clockwise direction, is unusual for
  polar coordinates. This is motivated by the fact that, with this definition,
  $\phi$ increases in the direction of gyration of a positively charged particle
  under the Lorentz force.  
\end{remark}

For smooth solutions $\varphi \in C^\infty$ of (\ref{eq:psi-psi-eq}),
we find that the derivatives
$\partial^\alpha_\theta \partial_u^\beta \varphi$ with the same order of
differentiation in $(u_1,u_2)$ are related to the solution of an ordinary
differential equation analogous to~(\ref{eq:ode}). In fact, differentiating
equation~(\ref{eq:psi-psi-eq}) yields
\begin{multline*}
  \partial_\theta^\alpha \partial_u^\beta [
    -(u_1 \partial_{u_2} - u_2 \partial_{u_1}) \varphi] =
  -(u_1 \partial_{u_2} - u_2 \partial_{u_1})
  (\partial_\theta^\alpha \partial_u^\beta \varphi) \\
  -\beta_1 \partial_\vartheta^\alpha \partial_{u_1}^{\beta_1-1}
  \partial_{u_2}^{\beta_2+1} \partial_{u_3}^{\beta_3} \varphi
  + \beta_2 \partial_\vartheta^\alpha \partial_{u_1}^{\beta_1+1}
  \partial_{u_2}^{\beta_2-1} \partial_{u_3}^{\beta_3} \varphi,
\end{multline*}
for any multi-index $\alpha \in \N_0^l$ and
$\beta = (\beta_1,\beta_2,\beta_3) \in \N_0^3$. This can be shown either
directly using the identities
\begin{align*}
  \partial_{u_1}^{\beta_1} \partial_{u_2}^{\beta_2} (u_1 \partial_{u_2} \varphi)
  &= u_1 \partial_{u_1}^{\beta_1} \partial_{u_2}^{\beta_2+1} \varphi + \beta_1
  \partial_{u_1}^{\beta_1-1} \partial_{u_2}^{\beta_2+1} \varphi, \\
  \partial_{u_1}^{\beta_1} \partial_{u_2}^{\beta_2} (u_2 \partial_{u_1} \varphi)
  &= u_2 \partial_{u_1}^{\beta_1+1} \partial_{u_2}^{\beta_2} \varphi + \beta_2
  \partial_{u_1}^{\beta_1+1} \partial_{u_2}^{\beta_2-1} \varphi,
\end{align*}
or by induction over $\beta_1$ and $\beta_2$.
Therefore, the $(\ell+1)$-dimensional complex-vector-valued function defined by
\begin{equation}
  \label{eq:pde-derivatives-var}
  \varphi_{\alpha,\beta,\ell}(\theta,u) = \big(
  \partial_\theta^\alpha \partial_u^\beta \varphi(\theta,u)
  \big)_{\beta_1+\beta_2=\ell}
  = \big(
  \partial_\theta^\alpha \partial_{u_1}^{\ell-j} \partial_{u_2}^j \partial_{u_3}^{\beta_3}
  \varphi(\theta,u) \big)_{j=0}^\ell,
\end{equation}
satisfies the system of partial differential equations
\begin{equation}
  \label{eq:pde-derivatives}
  - (u_1 \partial_{u_2} - u_2 \partial_{u_1}) \varphi_{\alpha,\beta,\ell}
  -i(a + \transpose{A}_\ell) \varphi_{\alpha,\beta,\ell} = \psi_{\alpha,\beta,\ell},
\end{equation}
where $A_\ell$ are the same matrices introduced in
equation~(\ref{eq:linear-algebraic-eq}) and studied in 
appendix~\ref{sec:matrices}, and the right-hand side is the
$(\ell+1)$-dimensional-vector-valued function
\begin{equation}
  \label{eq:pde-derivatives-rhs}
  \psi_{\alpha,\beta,\ell} = \Big(\partial_\theta^\alpha\partial_u^\beta \psi + i
  \sum_{\alpha' < \alpha} \sum_{\beta' < \beta} \binom{\alpha}{\alpha'}
  \binom{\beta}{\beta'} (\partial_\theta^{\alpha-\alpha'}
  \partial_u^{\beta-\beta'} a) (\partial_\theta^{\alpha'} \partial_u^{\beta'}
  \varphi)\Big)_{j=0}^{\ell},
\end{equation}
with $\beta= (\beta_1,\beta_2,\beta_3)$, 
$\beta_1 = \ell-j$ and $\beta_2 = j$, $j=0,\ldots,\ell$.
For $(u_1,u_2) \not= (0,0)$, let
\begin{align*}
  U_{\alpha,\beta,\ell}(r,\phi) &=
  \varphi_{\alpha,\beta,\ell}(\theta,u_\perp \cos\phi, - u_\perp \sin\phi,u_3), \\
  V_{\alpha,\beta,\ell}(r,\phi) &=
  \psi_{\alpha\beta,\ell}(\theta,u_\perp \cos\phi, - u_\perp \sin\phi,u_3).
\end{align*}
Then, if $\varphi \in C^\infty$ is a solution of (\ref{eq:psi-psi-eq}),
necessarily it must hold that
\begin{equation}
  \label{eq:ode-derivatives}
  \left\{
  \begin{aligned}
    \partial_\phi U_{\alpha,\beta,\ell}(r,\phi) - i\big(
    \tilde{a}(r) + \transpose{A}_\ell \big)
    U_{\alpha,\beta,\ell}(r,\phi) &= V_{\alpha,\beta,\ell}(r,\phi), \\
    U_{\alpha,\beta,\ell}(r,0) &= U_{\alpha,\beta,\ell}(r,2\pi).
  \end{aligned}
  \right.
\end{equation}
For $\ell=0$ the matrix $A_\ell$ reduces to $A_0=0$, hence
equation~(\ref{eq:ode}) is a special case of (\ref{eq:ode-derivatives}) obtained
for $\alpha=\beta=0$ and $\ell=0$. In general, 
equation~(\ref{eq:ode-derivatives}) is a system of $\ell+1$ first-order
ordinary differential equations on $[0,2\pi]$ with periodic boundary conditions,
for which we have the following result.

\begin{proposition}
  \label{th:ode-existence-main}
  For $\ell \in \N_0$, let $M_\ell$ be a $(\ell+1) \times (\ell +1)$
  diagonalizable, complex matrix with eigenvalues
  $\lambda_{\ell,j} \in \C \setminus \Z$, $j=0,1,\ldots,\ell$.
  Then, for any function
  $V_\ell \in C^\infty([0,2\pi],\C^{\ell+1})$ satisfying $V_\ell(0)=V_\ell(2\pi)$,
  there exists a unique solution $U_\ell \in C^\infty([0,2\pi],\C^{\ell+1})$ of
  \begin{equation*}
    U_\ell'(\phi) - iM_\ell U_\ell(\phi) = V_\ell(\phi), \qquad
    U_\ell(0) = U_\ell(2\pi),
  \end{equation*}
  given in Fourier series by
  \begin{equation*}
    U_\ell(\phi) = \sum_{n \in \Z} [i(M_\ell - n)^{-1} \hat{V}_{\ell,n}] e^{+in\phi},
    \qquad
    \hat{V}_{\ell,n} = \frac{1}{2\pi}\int_0^{2\pi} V_\ell(\phi) e^{-in\phi} d\phi,
  \end{equation*}
  and if $\imag \lambda_{\ell,j} \not = 0$, the solution satisfies, for all
  $\phi \in [0,2\pi]$, 
  \begin{equation*}
    |U_\ell(\phi)|_\infty \leq \frac{\kappa_\ell}{\lambda_{\ell,\mathrm{m}}}
    \max_{\phi' \in [0,2\pi]} |V_\ell(\phi')|_\infty,
  \end{equation*}
  where, for $z = (z_0,z_1,\ldots,z_\ell)\in\C^{\ell+1}$,
  $|z|_\infty \coloneqq \max_j|z_j|$ is the $L^\infty$ norm in $\C^{\ell+1}$,
  $\kappa_\ell$ is a constant depending only on $M_\ell$, and
  $\lambda_{\ell,\mathrm{m}} = \min_j |\imag \lambda_{\ell,j}|$.
\end{proposition}

\begin{proof}
  Since $V_\ell \in C^\infty([0,2\pi],\C^{\ell+1})$,
  for any $\mu \in \N_0$, the Fourier coefficients satisfy
  $|n|^\mu |\hat{V}_{\ell,n}|_\infty \leq \max_{\phi'} |\partial_\phi^\mu V_\ell
  (\phi')|_\infty$, hence the corresponding Fourier series converges in
  $C^k([0,2\pi],\C^{\ell+1})$ for every $k \in \N_0$. Dini's test implies that
  the sum of the Fourier series is equal to $V_\ell$, that is, $V_\ell$ can be
  represented by a Fourier series.
  
  Analogously a function $U_\ell \in C^1([0,2\pi],\C^{\ell+1})$ can be represented
  by a Fourier series, with convergence in $C^1([0,2\pi],\C^{\ell+1})$ and it is
  a solution if and only if the Fourier coefficients $\hat{U}_{\ell,n}$ satisfy 
  \begin{equation*}
    (M_\ell - n) \hat{U}_{\ell,n} = i \hat{V}_{\ell,n}.
  \end{equation*}
  As it was assumed that $M_\ell$ is diagonalizable, that is, there exists a
  non-singular complex matrix $S_\ell$ such that
  $S_\ell^{-1} M_\ell S_\ell = \Lambda_\ell$ where
  $\Lambda_\ell = \diag(\lambda_{\ell,0}, \lambda_{\ell,1}, \ldots, \lambda_{\ell,\ell})$
  is the diagonal matrix of eigenvalues.
  Then $M_\ell - n = S_\ell (\Lambda_\ell - n) S_{\ell}^{-1}$ is
  non-singular for all $n \in \Z$, if and only if $\lambda_{\ell,j} \not\in \Z$.
  Since this is the case, the Fourier coefficients of the solution are uniquely
  determined and given by $\hat{U}_{\ell,n} = i(M_\ell - n)^{-1} \hat{V}_{\ell,n}$.
  The norm of the Fourier coefficients can be readily estimated by
  \begin{equation*}
    |\hat{U}_{\ell,n}|_\infty
    \leq  |S_\ell|_\infty \cdot |(\Lambda_\ell - n)^{-1}|_\infty \cdot
    |S_\ell^{-1}|_\infty  \cdot | \hat{V}_{\ell,n}|_\infty
    \leq \frac{\kappa_\ell}{\delta_{\ell}} |\hat{V}_{\ell,n}|_\infty,
  \end{equation*}
  where $\kappa_\ell = |S_\ell|_\infty |S_\ell^{-1}|_\infty$ is the
  condition number of the matrix $S_\ell$ and thus depends only on $M_\ell$, while
  $\delta_\ell = \min_{j,n} |\lambda_{\ell,j} - n| > 0$ measures the distance of
  the eigenvalues from $\Z$.  
  Since for $n\not=0$, $|\hat{V}_{\ell,n}|_\infty = O(|n|^{-\mu})$ for all
  $\mu \in \N_0$, the Fourier series of $U_\ell$ converges in
  $C^k([0,2\pi],\C^{\ell+1})$ for every $k \in \N_0$, hence the sum
  $U_\ell$ belongs to $C^\infty([0,2\pi],\C^{\ell+1})$ and it is the unique
  classical solution of the problem. 
  
  We obtain an equivalent representation of the classical solution. In fact
  $U_\ell$ must necessarily satisfy  
  \begin{equation*}
    \big( e^{-iM_\ell \phi} U_\ell(\phi) \big)' = e^{-iM_\ell \phi}V_\ell(\phi),
  \end{equation*}
  and the general solution of this equation, with arbitrary initial condition
  $U_\ell(0)$, is 
  \begin{equation*}
    e^{-iM_\ell \phi} U_\ell(\phi) = U_\ell(0) +
    \int_0^\phi e^{-iM_\ell \phi'} V_\ell(\phi')d\phi'.
  \end{equation*}
  Then the periodic boundary condition $U_\ell(0) = U_\ell(2\pi)$ amounts to
  \begin{equation*}
    \big(e^{-2\pi iM_\ell} - 1\big) U_\ell(0) = \int_0^{2\pi} e^{-iM_\ell \phi'}
    V_\ell (\phi')d\phi'.
  \end{equation*}
  The matrix on the left-hand side is diagonalizable with eigenvalues
  $e^{-2\pi i\lambda_{\ell,j}} - 1$; for $\lambda_{\ell,j} \in \C \setminus \Z$
  the eigenvalues are all non-zero, the matrix is invertible, and the integration 
  constant $U_\ell(0)$ is uniquely determined. At last one finds that there is a
  unique periodic solution given by
  \begin{equation}
    \label{eq:U-solution}
    U_\ell (\phi) = [e^{-2\pi iM_\ell} -1]^{-1}
    \int_0^{2\pi} e^{+iM_\ell (\phi-\phi')} V_\ell(\phi')d\phi'
    + \int_0^{\phi} e^{+iM_\ell (\phi-\phi')} V_\ell(\phi')d\phi'.
  \end{equation}
  Equation (\ref{eq:U-solution}) shows that $U_\ell \in C^\infty([0,2\pi],\C^{\ell+1})$.
  The components of the vector $S_\ell^{-1} U_\ell$ are given by
  \begin{multline}
    \label{eq:U-solution-mode}
    \big(S_\ell^{-1} U_\ell (\phi)\big)_j =
    \frac{1}{e^{-2\pi i\lambda_{\ell,j}} -1}
    \int_0^{2\pi} e^{+i\lambda_{\ell,j} (\phi-\phi')}
    \big(S_\ell^{-1} V_\ell(\phi')\big)_j d\phi' \\
    + \int_0^{\phi} e^{+i\lambda_{\ell,j} (\phi-\phi')}
    \big(S_\ell^{-1} V_\ell(\phi')\big)_j d\phi'.
  \end{multline}
  Therefore,
  \begin{multline*}
    \big| \big(S_\ell^{-1} U_\ell (\phi)\big)_j \big|
    \leq \max_{\phi'} \big|\big(S_\ell^{-1} V_\ell (\phi')\big)_j\big|
    \Bigg[\Big|\frac{1}{e^{-2\pi i\lambda_{\ell,j}}-1}\Big|
    \int_0^{2\pi} e^{-\imag\lambda_{\ell,j} (\phi-\phi')} d\phi' \\
    + \int_0^{\phi} e^{-\imag\lambda_{\ell,j} (\phi-\phi')} d\phi' \Bigg].
  \end{multline*}
  For the factor in square brackets, we use
  \begin{equation*}
    \big|1 - e^{-2\pi i\lambda_{\ell,j}} \big| \geq
    \big| 1 - |e^{-2\pi i\lambda_{\ell,j}}|\big| =
    \big|1 - e^{2\pi \imag\lambda_{\ell,j}}\big|,
  \end{equation*}
  so that
  \begin{equation*}
    \Big|\frac{1}{e^{-2\pi i\lambda_{\ell,j}}-1}\Big| \leq
    \frac{1}{|e^{2\pi \imag\lambda_{\ell,j}} - 1|}.
  \end{equation*}
  For any $\phi_1>0$ and $y \not= 0$ we have
  \begin{equation*}
    \int_0^{\phi_1} e^{-y(\phi-\phi')} d\phi' = \frac{e^{-y \phi}}{y}
    \big(e^{y \phi_1} - 1 \big) = \frac{e^{-y \phi}}{|y|}
    \big|e^{y \phi_1} - 1 \big|,
  \end{equation*}
  and the two needed integrals are obtained for $\phi_1=2\pi$ and $\phi_1=\phi$.
  Hence,
  \begin{equation*}
    \big| \big(S_\ell^{-1} U_\ell (\phi)\big)_j \big| \leq 
    \frac{1}{|\imag\lambda_{\ell,j}|}
    \Big[ e^{-\imag\lambda_{\ell,j} \phi} +
      \big|1 - e^{-\imag\lambda_{\ell,j} \phi}\big| \Big]
    \max_{\phi'} \big|\big(S_\ell^{-1} V_\ell (\phi')\big)_j\big|.
  \end{equation*}
  If $\imag\lambda_{\ell,j} > 0$, the term in square brackets is
  equal to one and we obtain
  \begin{equation}
    \label{eq:U-estimate-mode}
    \big| \big(S_\ell^{-1} U_\ell (\phi)\big)_j \big| \leq 
    \frac{1}{|\imag\lambda_{\ell,j}|}
    \max_{\phi'} \big|\big(S_\ell^{-1} V_\ell (\phi')\big)_j\big|.
  \end{equation}
  If $\imag\lambda_{\ell,j} < 0$,
  we consider $\big(S_\ell^{-1} U_\ell(2\pi-\phi)\big)_j$ 
  and upon changing integration variable in (\ref{eq:U-solution-mode}) we obtain
  \begin{multline*}
    \big(S_\ell^{-1} U_\ell (2\pi-\phi)\big)_j =-\Bigg[
      \frac{1}{e^{2\pi i\lambda_{\ell,j}} -1}
      \int_0^{2\pi} e^{-i\lambda_{\ell,j} (\phi-\phi')}
      \big(S_\ell^{-1} V_\ell(2\pi-\phi')\big)_j d\phi' \\
      + \int_0^{\phi} e^{-i\lambda_{\ell,j} (\phi-\phi')}
      \big(S_\ell^{-1} V_\ell(2\pi-\phi')\big)_j d\phi'\Bigg].
  \end{multline*}
  The factor in square brackets has the same form as the right-hand side of
  (\ref{eq:U-solution-mode}) with $\lambda_{\ell,j}$ replaced by
  $-\lambda_{\ell,j}$ and now $\imag(-\lambda_{\ell,j}) > 0$. Hence we obtain
  \begin{equation*}
    \big| \big(S_\ell^{-1} U_\ell (2\pi-\phi)\big)_j \big| \leq 
    \frac{1}{|\imag\lambda_{\ell,j}|}
    \max_{\phi'} \big|\big(S_\ell^{-1} V_\ell (\phi')\big)_j\big|.
  \end{equation*}
  Since $\phi$ is arbitrary, this is equivalent to~(\ref{eq:U-estimate-mode})
  for $\imag\lambda_{\ell,j}<0$.
  Since $|\imag\lambda_{\ell,j}|\geq \lambda_{\ell,\mathrm{m}} >0$,
  taking the maximum over $j$ in (\ref{eq:U-estimate-mode}) yields
  \begin{equation*}
    \big|S_\ell^{-1} U_\ell(\phi)\big|_\infty \leq
    \frac{1}{\lambda_{\ell,\mathrm{m}}} \max_j \max_{\phi'}
    \big| \big(S_\ell^{-1} V_\ell(\phi')\big)_j\big| =
    \frac{1}{\lambda_{\ell,\mathrm{m}}}
    \max_{\phi'} \big|S_\ell^{-1} V_\ell(\phi')\big|_\infty.
  \end{equation*}
  Then
  \begin{equation*}
    \big|U_\ell(\phi)\big|_\infty \leq
    \big|S_\ell\big|_\infty \big|S_\ell^{-1} U_\ell(\phi)\big|_\infty \leq
    \frac{\kappa_l}{\lambda_{\ell,\mathrm{m}}} \max_{\phi'} \big|V_\ell(\phi')\big|_\infty,
  \end{equation*}
  which is the claimed estimate.
\end{proof}

As a corollary we obtain the solution of the problem stated in
proposition~\ref{th:ode-existence-main} with generic parameters
$r \in \mathcal{O} \subseteq \R^m$, where $\mathcal{O}$ is an open set.

\begin{corollary}
  \label{th:ode-with-parameters-1}
  For $m \in \N$, $\mathcal{O} \subseteq \R^m$, and $\ell \in \N_0$,
  let $M_\ell \in C^\infty(\mathcal{O}, \R^{(\ell+1) \times (\ell+1)})$ and
  $V_\ell \in C^\infty(\mathcal{O} \times [0,2\pi], \C^{\ell+1})$ be such that
  \begin{itemize}
  \item[1)] for any $r \in \mathcal{O}$, there is a non-singular matrix
    $S_\ell(r)$ for which $\Lambda_\ell (r) = S_\ell(r)^{-1} M_\ell(r) S_\ell(r)$
    is diagonal with eigenvalues $\lambda_{\ell,j}(r) \in \C \setminus \R$ satisfying
    $|\imag \lambda_{\ell,j} (r) | \geq \eta > 0$ for $j=0,\ldots \ell$,
    $r \in \mathcal{O}$ and for a given $\eta>0$, 
  \item[2)] $V(r,0) = V(r,2\pi)$ for $r \in \mathcal{O}$, and
  \item[3)] $S_{\ell}$, $S_\ell^{-1}$, and $\lambda_{\ell,j}$ are of class
    $C^\infty(\mathcal{O})$.
  \end{itemize}
  Then, the problem
  \begin{equation*}
    \partial_\phi U_\ell(r,\phi) - iM_\ell(r) U_\ell(r,\phi) = V_\ell(r,\phi),
    \quad U_\ell(r,0) = U_\ell(r,2\pi),
  \end{equation*}
  has a unique solution $U_\ell \in C^\infty(\mathcal{O} \times [0,2\pi], \C^{\ell+1})$,
  and, for all $\phi \in [0,2\pi]$,
  \begin{equation*}
    |U_\ell(r,\phi)| \leq \frac{\kappa_\ell}{\eta} \max_{\phi' \in [0,2\pi]}
    |V_\ell(r,\phi')|.
  \end{equation*}
\end{corollary}

\begin{proof}
  A function $U_\ell$ is a solution if and only if, for any $r \in \mathcal{O}$,
  $U_\ell(r,\cdot)$ solves the ordinary differential equation in
  proposition~\ref{th:ode-existence-main}. Assumptions 1) and 2) imply
  that all the hypotheses of proposition~\ref{th:ode-existence-main} are
  verified and we note that $\lambda_{\ell,\mathrm{m}}(r) \geq \eta$ uniformly for
  $r\in \mathcal{O}$. Therefore, for any $r \in \mathcal{O}$, there is a unique
  solution $U_\ell(r,\cdot)$ to the problem of
  corollary~\ref{th:ode-with-parameters-1}. The fact that $U_\ell$ is in
  $C^\infty(\mathcal{O} \times [0,2\pi])$ follows from the explicit
  formula~(\ref{eq:U-solution}) and assumption 3) by using the classical results of
  differentiation in the integral. The estimates follow from the ones proven in 
  proposition~\ref{th:ode-existence-main}.
\end{proof}

We can now give the general result for equation~(\ref{eq:psi-psi-eq}) with
right-hand side in $C^\infty$ and then in the Schwartz space, which will
immediately imply uniqueness in $\Schwartz^\prime$. Uniqueness of the $C^\infty$
solution, in particular, is a special case of lemma~\ref{th:phi-uniqueness}, but
here we give a different more explicit argument. 

\begin{proposition}
  \label{th:psi-psi-Schwartz}
  Let $a \in C^\infty(\R^l \times \R^3)$ satisfy
  condition~(\ref{eq:minimal-a-condition}). Then for any
  $\psi \in C^\infty(\R^l \times \R^3)$, equation~(\ref{eq:psi-psi-eq}) has a 
  unique solution $\varphi \in C^\infty(\R^l \times \R^3)$.
\end{proposition}

\begin{proof}
  Uniqueness of a $C^1$ solution. If $\varphi \in C^1(\R^l \times \R^3)$ is a
  solution of~(\ref{eq:psi-psi-eq}), evaluating the equations at $(u_1,u_2)=(0,0)$
  yields 
  \begin{equation*}
    \varphi(\theta,0,0,u_3) = i\psi(\theta,0,0,u_3) / a(\theta,0,0,u_3),
  \end{equation*}
  while for $(u_1,u_2) \not = (0,0)$, 
  \begin{equation*}
    \varphi(\theta,u) = U(r,\phi),
  \end{equation*}
  where $U$ is the unique solution of~(\ref{eq:ode}) constructed in
  using corollary~\ref{th:ode-with-parameters-1}.
  These conditions completely define the value of a $C^1$ solution everywhere in
  $\R^l \times \R^3$.  

  Existence of a $C^\infty$ solution. First we address the special case
  \begin{equation}
    \label{eq:special-psi}
    \psi(\theta,u) = \sum_{m+n=k} u_1^m u_2^n \tilde{\psi}_{m,n}(\theta, u),
  \end{equation}
  where $\tilde{\psi}_{m,n} \in C^\infty(\R^l \times \R^3)$ and $k \geq 2$ is a
  given integer. Let us consider the ordinary differential
  equation~(\ref{eq:ode}) with source term $V$ determined by the $\psi$ 
  given in~(\ref{eq:special-psi}). This equation is a special cases of the problem 
  addressed in corollary~\ref{th:ode-with-parameters-1} with $\ell = 0$;
  particularly, because of assumption (\ref{eq:minimal-a-condition}),
  $M_0(r) = \tilde{a}(r)$ satisfies the hypotheses of the corollary. Therefore,
  equation~(\ref{eq:ode}) has a unique solution
  $U \in C^\infty (\mathcal{O} \times [0,2\pi])$. 
  Due to the special choice of $\psi$ and the estimate in
  corollary~\ref{th:ode-with-parameters-1}, one deduces that
  for any given point $(\theta,u_3)$ and $\delta > 0$, there are constants
  $c_{U,\theta,u_3,\delta}$ and $c_{V,\theta,u_3,\delta} > 0$ for which 
  \begin{equation*}
    \big|U(r,\phi)\big| \leq u_\perp^k c_{U,\theta,u_3,\delta}, \quad
    \big|V(r,\phi)\big| \leq u_\perp^k c_{V,\theta,u_3,\delta},
  \end{equation*}
  uniformly in $u_\perp \in (0, \delta]$ and $\phi \in [0,2\pi]$.
  Let us construct the function
  \begin{equation*}
    \varphi(\theta,u) \coloneqq
    \begin{cases}
      U(r,\phi), & \text{ for } (u_1,u_2) \not=(0,0), \\
      0, \qquad  & \text{ for } (u_1,u_2) =(0,0).
    \end{cases}
  \end{equation*}
  Since polar coordinates in the region $(u_1,u_2) \not= (0,0)$ define a
  diffeomorphism which maps the partial differential equation~(\ref{eq:psi-psi-eq})
  into the ordinary differential equation~(\ref{eq:ode}), the function 
  $\varphi$ is of class $C^\infty$ and solves equation~(\ref{eq:psi-psi-eq}) 
  in the open set $\{(\theta,u) \in \R^l \times \R^3: (u_1,u_2) \not=(0,0)\}$.
  We also have
  $|\varphi(\theta,u)| \leq c_{U, \theta,u_3,\delta} |(u_1,u_2)|^k$,
  and thus $\varphi$ is continuous on whole domain $\R^l \times \R^3$.

  We now show that $\varphi \in C^{k-1}(\R^l \times \R^3)$. We need to check
  the existence of derivatives at $(u_1,u_2) = (0,0)$ and their continuity.
  With this aim we collect the derivatives of $\varphi$ with the same order
  $\ell$ of differentiation in $(u_1,u_2)$ into the vector-valued functions 
  $\varphi_{\alpha,\beta,\ell}$, as defined in~(\ref{eq:pde-derivatives-var});
  analogously let $\psi_{\alpha,\beta,\ell}$ be given by
  (\ref{eq:pde-derivatives-rhs}). For $(u_1,u_2) \not= (0,0)$, $\varphi$ is
  of class $C^\infty$ and solves equation~(\ref{eq:psi-psi-eq}), so that
  $\varphi_{\alpha,\beta,\ell}$ satisfies equation~(\ref{eq:pde-derivatives})
  with source $\psi_{\alpha,\beta,\ell}$. In polar coordinates those
  equations amount to ordinary differential equations~(\ref{eq:ode-derivatives})
  for the vector-valued functions given by 
  $U_{\alpha,\beta,\ell}(r,\phi) = \varphi_{\alpha,\beta,\ell}(\theta,
  u_\perp \cos\phi, - u_\perp\sin\phi,u_3)$ and with source
  $V_{\alpha,\beta,\ell}(r,\phi) = \psi_{\alpha,\beta,\ell}(\theta,
  u_\perp \cos\phi, - u_\perp\sin\phi,u_3)$.
  From lemma~\ref{th:matrix-Ak} we know that the matrices $A_\ell$, and thus 
  $\transpose{A}_\ell$, are diagonalizable with integer eigenvalues. It follows
  that the matrices $M_\ell(r) = \tilde{a}(r) + \transpose{A}_\ell$ in
  equation~(\ref{eq:ode-derivatives}) 
  are diagonalizable and the imaginary part of the eigenvalues coincides
  with $\imag \tilde{a}$. In view of assumption~(\ref{eq:minimal-a-condition}),
  we have $|\imag \tilde{a}|\geq \eta > 0$, and the hypotheses of
  corollary~\ref{th:ode-with-parameters-1} are therefore satisfied. We can
  conclude that
  \begin{equation}
    \label{eq:estimate-Upm}
    \big| U_{\alpha,\beta,\ell}(r,\phi) \big|_\infty \leq \frac{\kappa_\ell}{\eta}
    \max_{\phi'} \big| V_{\alpha,\beta,\ell}(r,\phi') \big|_\infty.
  \end{equation}
  We can use this estimate to show that, for $(u_1,u_2) \not= (0,0)$ we have
  \begin{equation}
    \label{eq:claim-phi-pm-derivatives}
    |\partial_\theta^\alpha \partial_u^\beta \varphi (\theta,u) | \leq 
    K_{\alpha,\beta} (\theta,u_3) u_\perp^{k-\ell} \text{ with }
    \ell = \beta_1 + \beta_2, \text{ and }
    0 < |(u_1,u_2)| \leq \delta.
  \end{equation}
  We prove this by induction over $\alpha$, $\beta_3$ and $\ell = \beta_1+\beta_2$.
  For $\alpha = 0$, $\beta_3=0$, and $\ell = 0$ the claim follows directly from
  the estimate in corollary~\ref{th:ode-with-parameters-1} and
  $V = u_\perp^k V_R$. For the induction step, let us assume that the claim
  holds for all $\alpha' < \alpha$, $\beta_3' < \beta_3$, and
  $\ell' = \beta_1' + \beta_2' < \ell = \beta_1+\beta_2$.
  Then from~(\ref{eq:pde-derivatives-rhs}) we deduce
  \begin{equation*}
    \big|V_{\alpha,\beta,\ell}(r,\phi)\big|_\infty =
    \big|\psi_{\alpha,\beta\ell}(\theta,u)\big|_\infty \leq
    c_{\alpha,\beta}^\psi(\theta,u_3) u_\perp^{k-\ell},
  \end{equation*}
  and from estimate (\ref{eq:estimate-Upm}) we deduce that, if
  $\beta_1 + \beta_2 = \ell$,
  \begin{equation*}
    \big|\partial_\theta^\alpha \partial^\beta_u \varphi(\theta,u)\big| \leq
    \big|\varphi_{\alpha,\beta,\ell}(\theta,u)\big|_\infty =
    \big|U_{\alpha,\beta,\ell}(r,\phi)\big|_\infty \leq
    \frac{\kappa_\ell}{\eta} c_{\alpha,\beta}^\psi(\theta,u_3) u_\perp^{k-\ell},
  \end{equation*}
  which is (\ref{eq:claim-phi-pm-derivatives}) as claimed. Therefore,
  \begin{equation*}
    \frac{|\partial_\theta^\alpha \partial_u^\beta \varphi (\theta,u)|}{
      |(u_1,u_2)|} \leq K_{\alpha,\beta}(\theta,u_3) |(u_1,u_2)|^{k-\ell-1},
    \text{ for } \beta_1 + \beta_2 = \ell \leq k-2,
  \end{equation*}
  which implies that the derivatives
  $\partial_\theta^\alpha \partial_u^\beta \varphi(\theta,0,0,u_3)$
  exist and are zero for $\beta_1+\beta_2\leq k-1$.
  Continuity of $\partial_\theta^\alpha \partial_u^\beta \varphi$ follows from
  inequality (\ref{eq:claim-phi-pm-derivatives}). Hence
  $\varphi \in C^{k-1}(\R^l \times \R^3)$ as claimed. Since $k \geq 2$,
  $\varphi \in C^1(\R^l \times \R^3)$ and equation~(\ref{eq:psi-psi-eq}) is
  satisfied also at $(u_1,u_2) = (0,0)$ since all terms vanish if $u_1 = u_2 = 0$.

  For the general case $\psi \in C^\infty(\R^l \times \R^3)$, for any integer
  $k \geq 2$ we write
  \begin{equation*}
    \psi = \psi_{k-1} + \psi_{\mathrm{r},k},
  \end{equation*}
  where $\psi_{k-1}$ is the Taylor polynomial of degree $k-1$ in $(u_1,u_2)$
  centered at $(u_1,u_2) = (0,0)$ and $\psi_{\mathrm{r},k}$ is the remainder,
  which is of the form~(\ref{eq:special-psi}). Hence the above argument applies
  to $\psi_{\mathrm{r},k}$. Let $\varphi_{\mathrm{r},k} \in C^{k-1}(\R^l \times \R^3)$
  be the unique solution obtained with $\psi_{\mathrm{r},k}$ as a source
  term. On the other hand proposition~\ref{th:psi-psi-polynomial} established
  the existence of a unique solution $\varphi_{k-1} \in C^\infty(\R^l \times \R^3)$
  for the case with source term $\psi_{k-1}$.
  The sum $\varphi = \varphi_{k-1} + \varphi_{\mathrm{r},k}$ is of class
  $C^{k-1}$ and it is the unique solution of~(\ref{eq:psi-psi-eq}). 
  Since $k$ is arbitrary we conclude that $\varphi \in C^{\infty}(\R^l \times \R^3)$.
\end{proof}

\begin{corollary}
  \label{th:psi-psi-Schwartz-estimates}
  Let $a \in C^\infty(\R^l \times \R^3)$ satisfy both conditions
  (\ref{eq:general-a-condition}) and let $m \in \R$ be the constant in
  (\ref{eq:growth-a-condition}).  
  Then the unique solution $\varphi \in C^\infty(\R^l \times \R^3)$ of
  equation~(\ref{eq:psi-psi-eq}) obtained in
  proposition~\ref{th:psi-psi-Schwartz} is such that: 
  \begin{itemize}
  \item[(i)] If there are $m_0 \in \R$ and $n_0 \in \N_0$ such that
    \begin{equation*}
      \big|\partial^\alpha_\theta \partial^\beta_u \psi(\theta,u) \big| \leq
      C_{\alpha,\beta}^\psi (1+\theta^2+u^2)^{m_0}, \quad
      |\alpha| + |\beta| \leq n_0,
    \end{equation*}
    then 
    \begin{equation*}
      \big|\partial_\theta^\alpha \partial_u^\beta \varphi(\theta,u) \big| \leq
      C_{\eta, \alpha, \beta}^\varphi (1+\theta^2+u^2)^{m_{\alpha,\beta}}, \quad
      |\alpha|+|\beta| \leq n_0,
    \end{equation*}
    where $m_{\alpha,\beta} = m_0$ if $m \leq 0$, and
    $m_{\alpha,\beta} = m_0 + (|\alpha|+|\beta|) m$ if $m >0$.
  \item[(ii)] If $\psi \in \Schwartz(\R^l \times \R^3)$,
    then $\varphi \in \Schwartz(\R^l \times \R^3)$.
  \end{itemize}
\end{corollary}

\begin{proof}
  (i) We proceed by induction as in the proof of~(\ref{eq:claim-phi-pm-derivatives}).
  For $\alpha=0$ and $\beta=0$, the claim follows directly from the estimate in
  corollary~\ref{th:ode-with-parameters-1} since for $u_1^2 + u_2^2 \geq \rho^2 > 0$,
  \begin{equation*}
    \big|\varphi(\theta,u) \big| = \big|U(r,\phi)\big|
    \leq \frac{1}{\eta} \max_\phi \big|V(r,\phi)\big| \leq
    \frac{1}{\eta} C^\psi_{0,0} (1+\theta^2+u^2)^{m_0}.
  \end{equation*}
  The constant is independent of the radius $\rho$, hence the claim. For the
  induction step, if the claim is true for all multi-indices 
  $\alpha' < \alpha$ and $\beta' < \beta$, where $\alpha$ and $\beta$ are
  any multi-indices satisfying $|\alpha|+|\beta| \leq n_0$,
  then, again for $u_1^2 + u_2^2 \geq \rho^2 > 0$, we estimate
  $\big|V_{\alpha,\beta,\ell}(r,\phi)\big|_\infty =
  \big|\psi_{\alpha,\beta,\ell}(\theta,u)\big|_\infty$ by  
  \begin{align*}
    \big|V_{\alpha,\beta,\ell}\big|_\infty &\leq \max \Big\{
    \big| \partial_\theta^\alpha \partial_u^\beta \psi \big| 
    + \sum_{\alpha' < \alpha} \sum_{\beta' < \beta} \binom{\alpha}{\alpha'}
    \binom{\beta}{\beta'} \big|\partial_\theta^{\alpha-\alpha'}
    \partial_u^{\beta-\beta'} a\big| \, |\partial_\theta^{\alpha'}
    \partial_u^{\beta'} \varphi\big| \Big\} \\
    &\leq \max \Big\{ C^\psi_{\alpha,\beta} (1+\theta^2+u^2)^{m_0}
    + \sum_{\alpha' < \alpha} \sum_{\beta' < \beta}
    \mathcal{C}^{\alpha,\beta}_{\eta,\alpha',\beta'}
    (1+\theta^2+u^2)^{m+m_{\alpha',\beta'}} \Big\},
  \end{align*}
  where the maximum is computed over all $(\beta_1, \beta_2)$ such that
  $\beta_1+\beta_2=\ell$, holding $\alpha$ and $\beta_3$ fixed. We observe that,
  if $m \leq 0$, then $m+m_{\alpha',\beta'} = m + m_0 \leq m_0 = m_{\alpha,\beta}$,
  while if $m>0$,
  $m+m_{\alpha',\beta'} = m_0 + (|\alpha'|+|\beta'|+1)m \leq m_0 +
  (|\alpha|+|\beta|)m = m_{\alpha,\beta}$. In both cases we have
  \begin{equation*}
    \max_\phi \big|V_{\alpha,\beta,\ell}(r,\phi)\big|_\infty \leq
    \tilde{C}^\psi_{\eta,\alpha,\beta_3,\ell}
    (1+\theta^2+u^2)^{m_{\alpha,\beta}}.
  \end{equation*}
  We can now apply inequality~(\ref{eq:estimate-Upm}) and deduce
  \begin{equation*}
    \big|\partial_\theta^\alpha \partial_u^\beta \varphi(\theta,u)\big|
    \leq \frac{\kappa_{\ell}}{\eta} \tilde{C}^\psi_{\eta,\alpha,\beta_3,\ell}
    (1+\theta^2+u^2)^{m_{\alpha,\beta}},
  \end{equation*}
  where $\ell = \beta_1+\beta_2$; this proves the claim for $\alpha$ and $\beta$
  satisfying $|\alpha|+|\beta| \leq n_0$.

  (ii) If $\psi \in \Schwartz(\R^l \times \R^3)$, then it satisfies the
  assumption of item (i) for all $m_0 \in \R$ and for all $\alpha,\beta$. For any
  $\mu \in \R$, $\alpha \in \N_0^l$, and $\beta \in \N_0^3$, let us apply the
  estimate proven in item (i) with $m_0 = \mu - (|\alpha|+|\beta|) m$; we obtain
  $\big|\partial_\theta^\alpha \partial_u^\beta \varphi (\theta,u) \big| \leq
  C_{\eta, \alpha, \beta}^\varphi (1+\theta^2+u^2)^\mu$.
  Hence $\varphi \in \Schwartz(\R^l \times \R^3)$ as claimed. 
\end{proof}

The existence of a solution $\varphi \in \Schwartz(\R^l \times \R^3)$ of
equation~(\ref{eq:psi-psi-eq}) implies (by duality)
uniqueness in $\Schwartz^\prime(\R^l \times \R^3)$ for the linear equation 
\begin{equation}
  \label{eq:generic-magnetized-kinetic}
  -(u_1 \partial_{u_2} - u_2 \partial_{u_1}) h - i a h = s,
\end{equation}
for $s \in \Schwartz^\prime(\R^l \times \R^3)$ and $a$ satisfying conditions
(\ref{eq:general-a-condition}).  
 
\begin{proposition}
  \label{th:uniqueness-magnetized}
  If $a \in C^\infty(\R^l \times \R^3)$ satisfies all conditions
  (\ref{eq:general-a-condition}), 
  equation~(\ref{eq:generic-magnetized-kinetic}) has at most one solution in
  $\Schwartz^\prime$. 
\end{proposition}

\begin{proof}
  We show that the only solution of the associate homogeneous equation is the
  trivial solution $h = 0$. Explicitly, this means that if
  \begin{equation*}
    \big\langle h,
    +(u_1 \partial_{u_2} - u_2 \partial_{u_1})\chi - ia \chi \big\rangle = 0,
  \end{equation*}
  for all $\chi \in \Schwartz(\R^l \times \R^3)$, then $h = 0$.

  Given an arbitrary test function $\psi \in \Schwartz(\R^l \times \R^3)$
  let us consider the equation 
  \begin{equation*}
    (u_1 \partial_{u_2} - u_2 \partial_{u_1}) \varphi - ia \varphi = \psi.
  \end{equation*}
  Since $-a$ satisfies conditions~(\ref{eq:general-a-condition}), we have
  established in corollary~\ref{th:psi-psi-Schwartz-estimates} that this
  equation has a unique solution $\varphi \in \Schwartz$. Then for any
  $\psi \in \Schwartz$,
  \begin{equation*}
    \langle h, \psi \rangle =
    \big\langle h,
    +(u_1 \partial_{u_2} - u_2 \partial_{u_1})\varphi -
    i a \varphi \big\rangle = 0,
  \end{equation*}
  and thus $h = 0$ as a tempered distribution.
\end{proof}

% New section
% =========================================================================
%

\section{Response of a uniform magnetized plasma} 
\label{sec:uniform}

In this section we address the case of a uniform magnetized plasma and prove
the results stated in section~\ref{sec:main-2}. With this aim we shall rely
heavily on the preparatory results of section~\ref{sec:auxiliary-pde} and on a
stationary-phase argument postponed to section~\ref{sec:stationary-phase}.

\subsection{Notation} 
\label{sec:notation-magnetized}

We shall make use of normalized momentum (\ref{eq:us}) and for
$(u_1,u_2) \not = (0,0)$ we define the two additional systems of cylindrical
coordinates  
\begin{equation}
  \label{eq:cylindrical}
  u_1 = u_\perp \cos \phi, \qquad u_2 = \mp u_\perp \sin\phi,
  \qquad u_3 = u_\parallel,
\end{equation}
with $u_\perp \in \R_+$ and $\phi \in [0,2\pi]$. We re-write the linearized
Vlasov equation in one of these two cylindrical coordinate systems depending on
the electric charge of the considered particle species: we choose the sign $-$
(resp., $+$) for a positively (resp., negatively) charged particle species.

With normalized Fourier variables
\begin{equation}
  \label{eq:norm-omega-k}
  \tau \coloneqq \omega/\omega_{c,\species}, \quad
  \xi  \coloneqq c k / \omega_{c,\species}, \quad \text{and with} \quad
  \kappa_\species (u) \coloneqq \gamma(u) \nu_\species(u) /\omega_{c,\species},
\end{equation}
we define the quantities
\begin{equation}
  \label{eq:a-and-bs}
  \begin{aligned}
    a_\eps (\tau,\xi,u) &\coloneqq
    \frac{\omega + i \eps\nu - k_3 v_3}{\omega_{c,\species}/\gamma} =
    \gamma (u) \tau - \xi_3 u_3 + i \eps \kappa_\species (u), \\
    b_i(\xi,u) &\coloneqq
    \frac{k_i v_\perp}{\omega_{c,\species}/\gamma} = u_\perp \xi_i, \quad
    i = 1,2.
  \end{aligned}
\end{equation}
Written in terms of normalized variables, the functions
$a_0 \coloneqq a_\eps|_{\eps=0}$, $b_i$, and $\gamma$ are independent of the
particle species.  

With $G_\species$ defined in (\ref{eq:gyrotropic}), the functions of $u \in \R^3$
defined by
\begin{equation}
  \label{eq:test-functions}
  \mathscr{F}_\species(u) \coloneqq \frac{1}{u_\perp}
  \frac{\partial G_\species}{\partial u_\perp}(u_\parallel,u_\perp), \qquad
  \mathscr{G}_\species(u) \coloneqq \frac{\partial G_\species}{\partial u_\parallel}
  (u_\parallel,u_\perp)
  - \frac{u_\parallel}{u_\perp} \frac{\partial G_\species}{\partial u_\perp}
  (u_\parallel,u_\perp),
\end{equation}
belong $\Schwartz(\R^3)$ because of the assumptions on $G_\species$.

Next we define the first-order partial differential operators
\begin{equation}
  \label{eq:Qj-operator}
  Q_{\species,j}(\tau,\xi,u,\partial_\xi) \coloneqq
  \mathscr{F}_\species(u) \Phi_j(\tau,\xi,u,\partial_\xi) +
  \mathscr{G}_\species(u) \Psi_j(\tau,\xi,u,\partial_\xi),
\end{equation}
where $\Phi_j, \Psi_j$, for $j=1,2,3$, are given in terms of the coefficients
\begingroup
\allowdisplaybreaks
\begin{subequations}
  \label{eq:psi-coeff}
  \begin{align}
    \Gamma_0(\tau,\xi,u) &\coloneqq \tau \gamma(u) e^{\pm i(\xi_1 u_2 - \xi_2u_1)}, \\
    \Gamma_j(\tau,\xi,u) &\coloneqq \xi_j e^{\pm i(\xi_1 u_2 - \xi_2u_1)},
    \qquad\quad j=1,2,3,
  \end{align}
\end{subequations}
by
\begin{equation*}
  \begin{aligned}
    \Phi_1 &\coloneqq \pm \Gamma_0 \partial_{\xi_2}, \\
    \Psi_1 &\coloneqq \pm \Gamma_3 \partial_{\xi_2},
  \end{aligned}
  \qquad
  \begin{aligned}
    \Phi_2 &\coloneqq \mp \Gamma_0 \partial_{\xi_1}, \\
    \Psi_2 &\coloneqq \mp \Gamma_3 \partial_{\xi_1},
  \end{aligned}
  \qquad
  \begin{aligned}  
    \Phi_3 &\coloneqq   i u_\parallel \Gamma_0, \\
    \Psi_3 &\coloneqq   i \Gamma_0 \mp
    \big(\Gamma_1 \partial_{\xi_2} - \Gamma_2
    \partial_{\xi_1} \big).
  \end{aligned}
\end{equation*}
\endgroup
The coefficients $\Gamma_j$ all satisfy
\begin{equation*}
  |\partial^\alpha_{\tau,\xi,u} \Gamma_j(\tau,\xi,u)| \leq C_\alpha
  (1+\tau^2+\xi^2+u^2)^{1+|\alpha|},
\end{equation*}
for all $\alpha \in \N_0^7$, and thus multiplication by $\Gamma_j$ is
closed in $\Schwartz$.  

We also introduce the functions of $b=(b_1,b_1)$ given by
\begin{equation}
  \label{eq:Anpm}
  A_n^\pm (b) \coloneqq
  \sum_{k,\ell\in\Z:\; k \pm \ell = n} (\pm i)^{\ell} J_{k}(b_1) J_{\ell}(b_2), 
\end{equation}
where $J_k$ are the Bessel's functions of the first kind.
If $b = (b_1,b_2) = (0,0)$, the only non-zero term in the series is the one for
$k=0$ and $\ell=0$, hence $A_n^\pm(0) = 1$ for $n=0$, and $A_n^\pm(0) = 0$ for
$n \not = 0$; if, on the other hand, $b = (b_1, b_2) \not= (0,0)$ we can write
$b_1 = |b| \cos \theta$, $b_2 = \mp |b| \sin \theta$ and
\begin{equation}
  \label{eq:Anpm-explicit}
  A^\pm_n(b) = J_n(|b|) e^{- i n \theta}, \qquad b \not= 0,
\end{equation}
This identity follows from Jacobi-Anger expansions \cite[p.361]{Abramowitz1964}
\begin{equation*}
  e^{iz\cos\phi} = \sum_{n \in \Z} i^n J_n(z) e^{in\phi},
  \quad
  e^{iz\sin\phi} = \sum_{n \in \Z} J_n(z) e^{in\phi},
\end{equation*}
for $z \in \C$ and $\phi \in \R$, that imply
\begin{equation}
  \label{eq:Anpm-der}
  e^{\pm i (b_2 \cos\phi \pm b_1 \sin\phi)} = \sum_{n\in\Z} A^\pm_n(b)
  e^{in \phi}.
\end{equation}
With $b_1 = |b| \cos \theta$ and $b_2 = \mp |b| \sin \theta$, one computes
\begin{equation*}
  \sum_{n\in\Z} A^\pm_n (b) e^{in \phi} = e^{ i |b| \sin(\phi - \theta)}
  = \sum_{n \in \Z} J_n(|b|) e^{in(\phi-\theta)},
\end{equation*}
which yields identity~(\ref{eq:Anpm-explicit}). At last let
\begin{equation}
  \label{eq:r-pm-eps}
  \begin{aligned}
    r^\pm_\eps (\tau, \xi, u) &\coloneqq
    \int_0^{+\infty} e^{ia_\eps(\tau,\xi,u) \lambda
      - i(\xi_1 u_1 + \xi_2 u_2) \sin\lambda 
      \pm i (\xi_2 u_1 - \xi_1 u_2) \cos\lambda} d\lambda \\
    &= \frac{i}{2} \frac{e^{-i\pi a_\eps(\tau,\xi,u)} }{\sin\big(\pi
      a_\eps(\tau,\xi,u)\big)} 
    P^\pm_\eps (\tau,\xi,u),
  \end{aligned}
\end{equation}
with
\begin{equation}
  \label{eq:Ppm}
  P^\pm_\eps(\tau,\xi,u) \coloneqq \int_0^{2\pi} e^{ia_\eps (\tau,\xi,u) \lambda
    - i(\xi_1 u_1 + \xi_2 u_2) \sin\lambda 
    \pm i (\xi_2 u_1 - \xi_1 u_2) \cos\lambda} d\lambda.
\end{equation}
The quantities defined in equations~(\ref{eq:psi-coeff})-(\ref{eq:Ppm}) are
independent on the particle species.

\begin{lemma}
  \label{th:lemma-4.0}
  For $\eps > 0$ and for each choice of the sign, the equation 
  \begin{equation*}
    \mp (u_1 \partial_{u_2} - u_2 \partial_{u_1}) r^\pm_\eps -
    ia_\eps r^\pm_\eps = e^{\pm i(\xi_2 u_1 - \xi_1 u_2)},
  \end{equation*}
  has a unique solution in $C^\infty(\R^7)$ given by (\ref{eq:r-pm-eps}) and
  for every $\alpha \in \N_0^4$, $\beta \in \N_0^3$ there are constants
  $C_{\alpha,\beta,\eps,\nu} > 0$ and $m_{\alpha,\beta} \in\R$ such that 
  \begin{equation*}
    |\partial_{\tau,\xi}^\alpha \partial_u^\beta r^\pm_\eps (\tau,\xi,u)| \leq
    C_{\alpha,\beta,\eps,\nu} (1+\tau^2+\xi^2+u^2)^{m_{\alpha,\beta}},
  \end{equation*}
  uniformly in $(\tau,\xi,u)$.
\end{lemma}

\begin{proof}
  For both choices of the sign, the equation for $r^\pm_\eps$ is of the
  form~(\ref{eq:psi-psi-eq}) with right-hand side in $C^\infty$ and with
  $\theta = (\tau,\xi)$. (For $r^-_\eps$ in particular, one can multiply the
  equation by $-1$ and set $a = -a_\eps$ and $\psi = - e^{\pm i(\xi_2 u_1 - \xi_1 u_2)}$.)
  For $\eps > 0$, the function $a_\eps$ defined in (\ref{eq:a-and-bs}) is such that
  all conditions (\ref{eq:general-a-condition}) are true. Therefore 
  proposition~\ref{th:psi-psi-Schwartz} ensures the existence a unique
  solution $r^\pm_\eps \in C^\infty(\R^7)$ for each choice of the sign.
  For the claimed estimates it is enough to show that for any $n_0 \in \N_0$ the
  right-hand side of the equation satisfies the hypothesis of
  corollary~\ref{th:psi-psi-Schwartz-estimates} (i), and this is straightforward.

  The integral expressions~(\ref{eq:r-pm-eps}) can be checked
  by direct substitution into the equation. In fact if
  $I^\pm_\eps(\tau,\xi,u; \lambda)$ denotes the integrand
  in~(\ref{eq:r-pm-eps}), we have the identity 
  \begin{equation*}
    \big[\mp(u_1\partial_{u_2} - u_2 \partial_{u_1}) -ia_\eps(\tau,\xi,u) \big]
    I^\pm_\eps(\tau, \xi,u; \lambda) 
    = -\frac{\partial}{\partial \lambda} I^\pm_\eps(\tau, \xi, u; \lambda),
  \end{equation*}
  and $I^\pm_\eps(\tau, \xi,u;0) = e^{\pm i(\xi_2 u_1 - \xi_1 u_2)}$. For the
  second form of $r_\eps^\pm$, we notice that, following Qin~et~al.~\cite{Qin2007},
  \begin{equation*}
    r^\pm_\eps(\tau,\xi,u) = \int_0^{+\infty} e^{ia_\eps (\tau,\xi,u) \lambda -
      i(\xi_1u_1+\xi_2u_2) \sin(\lambda + 2\pi)
      \pm i (\xi_2 u_1 - \xi_1u_2) \cos(\lambda + 2\pi)}d\lambda,
  \end{equation*}
  and changing integration variable we have
  \begin{align*}
    r^\pm_\eps(\tau,\xi,u) &= e^{-2\pi i a_\eps(\tau,\xi,u)}
    \int_{2\pi}^{+\infty} e^{ia_\eps(\tau,\xi,u) \lambda - i(\xi_1u_1+\xi_2u_2)
      \sin\lambda \pm i (\xi_2 u_1 - \xi_1u_2) \cos\lambda} d\lambda \\
    &= e^{-2\pi i a_\eps(\tau,\xi,u)} \bigg[r_\eps^\pm(\tau,\xi,u)  \\
      & \qquad\qquad-
      \int_{0}^{2\pi} e^{ia_\eps(\tau,\xi,u) \lambda - i(\xi_1u_1+\xi_2u_2) \sin\lambda
        \pm i (\xi_2 u_1 - \xi_1u_2) \cos\lambda}d\lambda \bigg],
    \intertext{hence,}
    r^\pm_\eps(\tau,\xi,u) &= \frac{-e^{-2\pi ia_\eps(\tau,\xi,u)}}
    {1-e^{-2\pi ia_\eps(\tau,\xi,u)}}
    \int_{0}^{2\pi} e^{ia_\eps(\tau,\xi,u) \lambda - i(\xi_1u_1+\xi_2u_2) \sin\lambda
      \pm i (\xi_2 u_1 - \xi_1u_2) \cos\lambda} d\lambda,
  \end{align*}
  from which the second expression for $r^\pm_\eps$ follows.
\end{proof}

\subsection[The roots of $a_0-n$]{%
  \for{toc}{The roots of \texorpdfstring{$a_0-n$}{} and
  \texorpdfstring{$\lim_{\varepsilon \rightarrow 0+}(1/\sin \pi a_\varepsilon)$}{}}%
  \except{toc}{The roots of \texorpdfstring{$a_0-n$}{} and the distribution
    \texorpdfstring{$\lim_{\varepsilon \rightarrow 0^+}(1/\sin \pi a_\varepsilon)$}{}}%
  }
\label{subsec:eq:a0}
From expression~(\ref{eq:r-pm-eps}), one can see that the main issue in
computing the limit for $\eps \to 0^+$ of $r^\pm_\eps$ consists in the sets
of points $(\tau,\xi,u) \in \R \times \R^3 \times \R^3$ for which  
\begin{equation*}
  a_0(\tau,\xi,u) = \gamma (u) \tau - \xi_3 u_3 \in \Z.
\end{equation*}
We note a few preliminary facts about such points.

\begin{remark}
  The condition $a_0(\tau,\xi,u)=n\in\Z$ is equivalent to
  \begin{equation*}
    \omega - k_3 v_3 = n \Omega_{\species}(u),    \quad n \in \Z,
  \end{equation*}
  which defines the cyclotron resonances: particles of the species $\species$ 
  that satisfy this condition along their orbit, for some integer $n$, resonate
  to a plane wave with frequency and wave vector $(\omega,k)$. For resonant
  particles, the Doppler-shifted wave frequency $\omega - k_3 v_3$ matches a
  multiple (also referred to as a harmonic) of the gyration frequency
  $\Omega_\species$ of the particle's orbit around the magnetic field.
\end{remark}

For any given $n \in \Z$ and $(\tau,\xi) \in \R \times \R^3$, $\tau \not = 0$, let
\begin{equation}
  \label{eq:resonances}
  \mathcal{R}_n(\tau,\xi) =  \{u\in\R^3: a_0(\tau,\xi,u) = n \}.
\end{equation}
Physically $\mathcal{R}_n(\tau,\xi)$ is the set of normalized particle momenta
$u$ that resonate with the $n$-th harmonic of the cyclotron frequency when the
wave field is a plane wave with frequency and wave vector are given by
$(\tau,\xi)$. Since $a_0$ is constant in $(\xi_1,\xi_2)$, the sets
$\mathcal{R}_{n}(\tau,\xi)$ for a fixed harmonic number $n$ depend only on
$\tau$, and $\xi_3$. In addition it is enough to study them for $\tau>0$ and
$\xi_3 \geq 0$, in view of the symmetries of the function $a_0$. (The hyperplane
$\tau=0$ will be excluded in our main results.)
A necessary condition for $u \in \mathcal{R}_{n}(\tau,\xi)$ is  
\begin{equation*}
  (\tau^2 - \xi_3^2) u_3^2  + \tau^2 (1 + u_1^2 + u_2^2)
  - 2n\xi_3 u_3 - n^2 = 0.
\end{equation*}
which defines a family of surfaces of revolution obtained by the rotation of conics
around the $u_3$-axis. Specifically we find ellipsoids for $\tau^2 > \xi_3^2$, a
paraboloid for $\tau^2 = \xi_3^2$, and one branch of a hyperboloid for
$\tau^2 < \xi^2_3$. We shall speak of elliptic, parabolic, and hyperbolic 
resonances, with reference these three conditions, respectively. In the elliptic
case, $\mathcal{R}_{n}(\tau,\xi)$ is non-empty if and only if $n \tau > 0$, that
is when, $\tau$ and $n$ are both non-zero and have the same sign; ellipsoids 
degenerate to a point when $\tau^2 - \xi_3^2 = n^2 > 0$. In the hyperbolic
case one can check that $\mathcal{R}_{n}(\tau,\xi)$ is non-empty for all integers
$n$. The special case $n=0$ is a particular hyperbolic resonance, to be referred
to as Landau resonance, for which $\mathcal{R}_0(\tau,\xi)$ is non empty only if 
$\tau^2 < \xi_3^2$. 

We shall study the limit in $\Schwartz^\prime(\R^3)$ of functions of the form
\begin{equation}
  \label{eq:r-tilde}
  \tilde{r}_\eps(\tau,\xi,u) = \frac{\tilde{P}_\eps(\tau,\xi,u)}{
    \sin\big(\pi a_\eps(\tau,\xi,u)\big)},
\end{equation}
where $\tilde{P}_\eps$ is a family of functions parameterized by
$\eps \in [0, \eps_0]$ for a fixed $\eps_0 >0$. The choices of
$\tilde{P}_\eps$ relevant to our analysis are
$\tilde{P}_\eps = i e^{-i\pi a_\eps} \Phi_j P^\pm_{\eps}/2$ and
$\tilde{P}_\eps = i e^{-i\pi a_\eps} \Psi_j P^\pm_{\eps}/2$, with notation of
section~\ref{sec:notation-magnetized}. With this aim, it will be sufficient to
consider a family of functions $\tilde{P}_\eps$ that satisfy the following
conditions: 
\begin{equation}
  \label{eq:polynomial-growth}
  \left\{
  \begin{aligned}
    &\tilde{P}_\eps \in C^\infty(\R^7), && \forall \eps \in [0,\eps_0], \\
    &\big|\partial^\mu \tilde{P}_\eps(\tau,\xi,u)\big| \leq
    \hat{c}_{\mu} (1+\tau^2+\xi^2)^{\hat{m}_\mu} (1+u^2)^{\hat{n}_\mu},  &&
    \forall \eps \in [0, \eps_0],
    \; \forall \mu \in \N_0^7,\\
    & \tilde{P}_\eps (\tau,\xi,u) \to \tilde{P}_0(\tau,\xi,u)
    \text{ for } \eps \to 0^+, && \forall (\tau,\xi,u) \in \R^7,
  \end{aligned}
  \right.
\end{equation}
where the constants $\hat{c}_\mu, \hat{m}_\mu, \hat{n}_\mu \in \R$ are independent of
$\eps \in [0,\eps_0]$ and $(\tau,\xi,u) \in \R^7$.

We state the main result for the limit of~(\ref{eq:r-tilde}) as $\eps \to 0^+$.
We shall see that it is sufficient to consider the case $\kappa_\species = 1$, or
$\nu_\species = \omega_{c,\species}/\gamma$; then $a_\eps$ is independent of the
particle species. This is a valid choice of the damping coefficient
$\nu_\species$, since it satisfies conditions~(\ref{eq:nu-conditions}). 

\begin{proposition}
  \label{th:limit_tilde_r_eps}
  Let $\eps_0, \tau_0 >0$, and let $\tilde{r}_\eps$ be the family of functions
  defined for $\eps \in (0,\eps_0]$ in equation~(\ref{eq:r-tilde}) with
  $\tilde{P}_\eps$ satisfying condition~(\ref{eq:polynomial-growth}), and
  with $\kappa_\species = 1$. Then, for every $(\tau,\xi)$, $\tau \not= 0$, there is
  $\tilde{r}_0(\tau,\xi,\cdot) \in \Schwartz^\prime(\R^3)$,
  such that
  \begin{itemize}
    
  \item[(i)] in the limit $\eps \to 0^+$, 
    $\tilde{r}_\eps(\tau,\xi,\cdot) \to \tilde{r}_0(\tau,\xi,\cdot)$ in
    $\Schwartz^\prime(\R^3)$;
    
  \item[(ii)] for every $\phi \in \Schwartz(\R^3)$ the function
    $(\tau,\xi)\mapsto\big\langle\tilde{r}_0(\tau,\xi,\cdot),\phi\big\rangle$
    is continuous on $(\R\setminus\{0\})\times\R^3$, and $C^\infty$ near
    points $(\tau,\xi)$ such that $\tau^2 \not= \xi_3^2 + n^2$ for all integers
    $n \geq 0$;

  \item[(iii)] for any $\phi \in \Schwartz(\R^3)$ there are reals $K_0,M>0$,
    such that 
    \begin{equation*}
      \big|\big\langle \tilde{r}_\eps(\tau,\xi,\cdot), \phi\big\rangle\big|
      \leq K_0  (1 + \tau^2 +\xi^2)^M, \\
    \end{equation*}
    for all $(\tau,\xi)$ with $|\tau| \geq \tau_0$ and $\eps \in [0,\eps_0]$.
  \end{itemize}
\end{proposition}

\begin{remark}[Degenerate resonances]
  The varieties $\tau^2 = \xi^2 + n^2$ for $n \geq 1$ in the Fourier space
  correspond to plane waves for which the elliptic resonance
  $\mathcal{R}_n(\tau,\xi)$ degenerates to a point. The special case $n=0$, that
  is $\tau^2-\xi^2=0$, corresponds to the parabolic resonance which separates
  elliptic and hyperbolic resonances. The function
  $(\tau,\xi)\mapsto\big\langle\tilde{r}_0(\tau,\xi,\cdot),\phi\big\rangle$ is
  smooth away from such topological transitions, where we can show continuity only.
\end{remark}

For the proof of proposition~\ref{th:limit_tilde_r_eps}, we need
a few preparatory results.
Let us choose a cut-off function $\chi \in C_0^\infty(\R)$ such that
\begin{equation*}
  0 \leq \chi(z) \leq 1, \quad 
  \chi(z) = 1, \text{ for } |z| < \frac{1}{4}, \quad
  \chi(z) = 0, \text{ for } |z| \geq \frac{1}{3},
\end{equation*}
and for any $n \in \Z$, $\delta \in (0,1)$ let 
\begin{equation*}
  \chi_{\delta,n}(\tau,\xi,u) = \chi\big((a_0(\tau,\xi,u)-n)/\delta\big).
\end{equation*}
One can see that $\chi_{\delta,n} \in C^\infty(\R^7)$ but it is not
necessarily compactly supported; in fact, $\chi_{\delta,n}(\tau,\xi,\cdot)$ is
localized around $\mathcal{R}_{n}(\tau,\xi)$ which is unbounded when
$\tau^2 \leq \xi_3^2$.

\begin{lemma}
  \label{th:partition-of-unity}
  If $(\tau,\xi,u) \in \supp (\chi_{\delta,n})$ for a certain $n \in \Z$ then
  $\chi_{\delta,m}(\tau,\xi,u) = 0$ for all $m \not=n$.
\end{lemma}

\begin{proof}
  The set $\supp \chi_{\delta,n}$ is given by the condition $|a_0 - n| \leq \delta/3$, 
  so that
  \begin{equation*}
    |a_0(\tau,\xi,u) - m| \geq |m-n| - |a_0(\tau,\xi,u) -n |
    > 1 - \delta/3 > \delta/3,
  \end{equation*}
  and thus $\chi_{\delta,m}(\tau,\xi,u)=0$.
\end{proof}

It follows from lemma~\ref{th:partition-of-unity} that the sum
\begin{equation}
  \label{eq:cut-off}
  \chi_\delta = \sum_{n \in \Z} \chi_{\delta,n},
\end{equation}
is locally finite and thus it defines a family of smooth functions $\chi_\delta$
for $\delta \in (0,1)$.

\begin{lemma}
  \label{th:on-supp-chi-delta}
  For every $\delta \in (0,1)$, the cut-off function $\chi_\delta$ in
  (\ref{eq:cut-off}) is such that
  \begin{align*}
    \big|\cos\big(\pi a_0(\tau,\xi,u)\big)\big| &\geq \cos(\pi \delta/3) > 0, \quad
    &&\text{ for } (\tau,\xi,u) \in \supp(\chi_\delta), \\
    \big|\sin\big(\pi a_0(\tau,\xi,u)\big)\big| &\geq \sin(\pi \delta/4) > 0, \quad
    &&\text{ for } (\tau,\xi,u) \in \supp(1-\chi_\delta).
  \end{align*}
\end{lemma}

\begin{proof}
  For every point $(\tau,\xi,u)$ in the support of $\chi_\delta$, there is
  an integer $n$ such that
  \begin{equation*}
    |a_0 - n| \leq \delta/3, \qquad
    |\cos(\pi a_0)| = \big|\cos\big(\pi(a_0-n)\big)\big| \geq
    \cos(\pi\delta/3) > 0.
  \end{equation*}
  Analogously, on the support of $1-\chi_\delta$, 
  \begin{equation*}
    |a_0 - n| \geq \delta/4, \qquad
    |\sin(\pi a_0)| = \big|\sin\big(\pi(a_0-n)\big)\big| \geq
    \sin(\pi\delta/4) > 0,
  \end{equation*}
  for all integers $n$.
\end{proof}
  
The cut-off function $\chi_\delta$ allows us to isolate the singularities of
(\ref{eq:r-tilde})  when $\eps \to 0^+$.
If $\nu_\species$ in equation~(\ref{eq:norm-omega-k}) is chosen so that
$\kappa_\species = 1$, for $(\tau,\xi,u) \in \supp(\chi_\delta)$ we can write 
\begin{equation*}
  \sin(\pi a_\eps) = \cos(\pi a_0) \big[
    \tan(\pi a_0) \cosh(\pi\eps) + i\sinh(\pi\eps) \big], 
\end{equation*}
and
\begin{equation}
  \label{eq:resonance-trick}
  \frac{1}{\sin(\pi a_\eps)} = \frac{-i}{\cos(\pi a_0)}
  \int_0^{+\infty} e^{i\lambda [\tan(\pi a_0) \cosh(\pi\eps) + i\sinh(\pi\eps)]}
  d\lambda.
\end{equation}
Therefore, for any $\psi \in C_0^\infty(\R^3)$,
\begin{equation}
  \label{eq:partition-of-unity}
  \langle \tilde{r}_\eps(\tau,\xi,\cdot), \psi\rangle =
  - i \int_0^{+\infty} e^{-\lambda \sinh(\pi\eps)}
  \mathcal{I}_{\delta,\eps}^c(\psi)(\tau,\xi,\lambda) d\lambda
  + \mathcal{I}_{\delta,\eps}^s(\psi)(\tau,\xi),
\end{equation}
where we have defined the functions
\begin{subequations}
  \label{eq:Psi-delta}
  \begin{align}
    \label{eq:complex-phase}
    \vartheta_\eps (\tau,\xi,u) &\coloneqq \tan\big(\pi a_0(\tau,\xi,u)\big)
    \cosh(\pi\eps), \\
    \label{eq:Psi-delta-sing}
    \mathcal{I}_{\delta,\eps}^c(\psi) (\tau,\xi,\lambda) &\coloneqq
    \int_{\R^3} e^{i\lambda \vartheta_\eps(\tau,\xi,u)}
    \frac{\tilde{P}_\eps(\tau,\xi,u) \psi(u)}{\cos\big(\pi a_0(\tau,\xi,u)\big)}
    \chi_\delta(\tau,\xi,u) du, \\
    \label{eq:psi-delta-reg}
    \mathcal{I}_{\delta,\eps}^s(\psi) (\tau,\xi) &\coloneqq
    \int_{\R^3} \frac{\tilde{P}_\eps(\tau,\xi,u) \psi(u)}{
      \sin\big(\pi a_\eps(\tau,\xi,u)\big)}
    \big(1-\chi_{\delta}(\tau,\xi,u)\big)du. \\
    \intertext{We also define}
    \label{eq:Psi-sing}
    \mathcal{I}_{\delta,0}^c(\psi) (\tau,\xi,\lambda) &\coloneqq
    \int_{\R^3} e^{i\lambda \tan(\pi a_0(\tau,\xi,u))}
    \frac{\tilde{P}_0(\tau,\xi,u) \psi(u)}{\cos\big(\pi a_0(\tau,\xi,u)\big)}
    \chi_\delta(\tau,\xi,u) du, \\
    \label{eq:psi-reg}
    \mathcal{I}_{\delta,0}^s (\psi)(\tau,\xi) &\coloneqq
    \int_{\R^3} \frac{\tilde{P}_0(\tau,\xi,u) \psi(u)}{
      \sin\big(\pi a_0(\tau,\xi,u)\big)}
    \big(1-\chi_{\delta}(\tau,\xi,u)\big)du.
  \end{align}
\end{subequations}

The main step in the proof of
proposition~\ref{th:limit_tilde_r_eps} consists of an application 
of the stationary phase formula \cite[Chapter 7]{Hormander1} in order to prove
that $\mathcal{I}^c_{\delta,\eps}(\psi)(\tau,\xi,\cdot)$ is bounded in $L^1$
uniformly in $\eps$. The real-valued phase is given by $\vartheta_\eps$ and the
parameter is $\lambda \geq 1$. All technical results needed in the proof of
proposition~\ref{th:limit_tilde_r_eps} are collected in
section~\ref{sec:stationary-phase}, below.

\begin{proof}[Proof of proposition~\ref{th:limit_tilde_r_eps}]
  We start from identity~(\ref{eq:partition-of-unity}),
  \begin{equation*}
    \langle \tilde{r}_\eps(\tau,\xi,\cdot), \psi\rangle =
    -i\int_0^{+\infty} e^{-\lambda \sinh(\pi\eps)}
    \mathcal{I}^c_{\delta,\eps}(\psi)(\tau,\xi,\lambda)d\lambda +
    \mathcal{I}^s_{\delta,\eps}(\psi)(\tau,\xi),
  \end{equation*}
  for $\psi \in C_0^\infty (\R^3)$ and
  lemma~\ref{th:application-of-stationary-phase} (i) shows that this is extended
  to $\phi \in \Schwartz(\R^3)$.
  
  For every $\delta \in (0,1)$, lemma~\ref{th:application-of-stationary-phase}
  (iii) and (v) allows us to define the functional 
  \begin{equation}
    \label{eq:tilde-r0}
    \big\langle \tilde{r}_{\delta,0}(\tau,\xi,\cdot), \phi\big\rangle =
    - i \int_0^{+\infty} \mathcal{I}_{\delta,0}^c(\phi)(\tau,\xi,\lambda) d\lambda
    + \mathcal{I}_{\delta,0}^s(\phi)(\tau,\xi),
  \end{equation}
  over $\Schwartz(\R^3)$ for every $(\tau,\xi)$, $\tau\not=0$.  

  In view of lemma~\ref{th:application-of-stationary-phase} (ii), as $\eps \to 0^+$,
  we have $\mathcal{I}^c_{\delta,\eps} (\phi)(\tau,\xi) \to
  \mathcal{I}^c_{\delta,0}(\phi)(\tau,\xi)$ and
  $\mathcal{I}^s_{\delta,\eps} (\phi) (\tau,\xi) \to
  \mathcal{I}^s_{\delta,0}(\phi) (\tau,\xi)$ for all 
  $\phi \in \Schwartz(\R^3)$.
  By the dominated convergence theorem and
  lemma~\ref{th:application-of-stationary-phase} (iii) we deduce 
  \begin{equation*}
    \int_0^{+\infty} e^{-\lambda \sinh(\pi\eps)}
    \mathcal{I}^c_{\delta,\eps}(\phi)(\tau,\xi,\lambda)d\lambda \to
    \int_0^{+\infty} \mathcal{I}^c_{\delta,0}(\phi)(\tau,\xi,\lambda)d\lambda,
  \end{equation*}
  for all $\phi \in \Schwartz(\R^3)$, hence
  $\tilde{r}_\eps (\tau,\xi,\cdot)  \to \tilde{r}_{\delta,0}(\tau,\xi,\cdot)$ in
  $\Schwartz^\prime(\R^3)$. By uniqueness of the limit, we have that
  $\tilde{r}_{\delta,0}(\tau,\xi,\cdot)$ is the same tempered distribution for
  all $\delta \in (0,1)$ which we denote by $\tilde{r}_0(\tau,\xi,\cdot)$. This
  proves (i).

  As for the regularity of $\langle \tilde{r}_0(\tau,\xi,\cdot), \phi \rangle$
  with respect to $(\tau,\xi)$,
  lemma \ref{th:application-of-stationary-phase} (i) show in particular that
  $\mathcal{I}_{\delta,0}^s(\phi) \in C^\infty$, hence it is enough to address
  \begin{equation*}
    \langle \tilde{r}^c_{\delta,0}(\tau,\xi,\cdot), \phi \rangle \coloneqq
    \langle \tilde{r}_{\delta,0}(\tau,\xi,\cdot), \phi \rangle -
    \mathcal{I}_{\delta,0}^s(\phi)(\tau,\xi) =
    -i \int_0^{+\infty} \mathcal{I}^c_{\delta,0}(\phi)(\tau,\xi,\lambda)
    d\lambda.
  \end{equation*}
  In lemma \ref{th:application-of-stationary-phase} (i) we have established that
  $\mathcal{I}^c_{\delta,0}(\phi)$ is $C^\infty$, and the inequality in item (iii) of
  the same lemma \ref{th:application-of-stationary-phase} gives a function
  $\tilde{\mathcal{B}}_{\eps_0,\delta} \in L^1(\R_+)$ such that 
  \begin{equation*}
    |\mathcal{I}^c_{\delta,0}(\phi)(\tau,\xi,\lambda)| \leq
    \tilde{\mathcal{B}}_{\eps_0,\delta} (\lambda),
  \end{equation*}
  uniformly in $(\tau,\xi) \in K$ where $K$ is any compact where $|\tau| > 0$.
  Then the dominated convergence theorem can be applied to show that
  $\langle \tilde{r}^c_{\delta,0}(\tau,\xi,\cdot), \phi \rangle$
  is continuous at any point $(\tau,\xi)$, where $|\tau| > 0$.
  
  Concerning the derivatives of $\mathcal{I}^c_{\delta,0}(\phi)$ with respect to
  $(\tau,\xi)$, let as fix a point $(\bar{\tau}, \bar{\xi})$ such that
  $|\bar{\tau}| \not=0$ and $\bar{\tau}^2 \not= \bar{\xi}_3^2 + n^2$ for all
  $n \in \N_0$. For a sufficiently small radius $\rho>0$ the closed ball
  $K = \{(\tau,\xi) :\; (\tau - \bar{\tau})^2 + (\xi - \bar{\xi})^2 \leq \rho^2\}$ 
  satisfies the assumptions of lemma~\ref{th:application-of-stationary-phase}
  item (iv). Therefore, there is a value of $\delta$ depending only on $K$ and an
  upper bound $|\partial_{\tau,\xi}^\alpha
  \mathcal{I}^c_{\delta,\eps}(\phi)(\tau, \xi, \lambda)| \leq
  \mathcal{B}^{(\alpha)}_{\eps_0,\delta}(\lambda, \tau_*)$ with
  $\mathcal{B}^{(\alpha)}_{\eps_0,\delta}(\cdot, \tau_*) \in L^1$,
  where $\tau_* = \min |\tau|$ in $K$.
  %%
  %% Analogously, lemma~\ref{th:application-of-stationary-phase} (iv) yields an
  %% upper bound of the derivatives of $\mathcal{I}^c_{\delta,0}(\phi)$  by an
  %% $L^1$-function uniformly in any compact $K$ where $|\tau| > 0$ and
  %% $\tau^2 \not= \xi_3^2 + n^2$.
  Hence, $\langle \tilde{r}^c_{\delta,0}(\tau,\xi,\cdot), \phi \rangle$ is of class
  $C^\infty$ near any point where $|\tau|>0$ and $\tau^2 \not = \xi_3^2 + n^2$,
  as claimed in (ii).
  
  Again lemma \ref{th:application-of-stationary-phase}, items (iii) and (v),
  imply
  \begin{multline*}
    |\langle \tilde{r}_\eps(\tau,\xi,\cdot),\phi\rangle| \leq
    \|\mathcal{B}_{\eps_0,\delta}(\cdot,\tau_0)\|_{L^1(\R_+)}
    (1 + \tau^2 + \xi^2)^{m_4 + \frac{5}{2}} \|\phi\|_{2\tilde{m}_4 + 10} \\
    + K^s_0 (1 + \tau^2 + \xi^2)^{\ell_{0,0}} \|\phi\|_{2 \tilde{\ell}_{0,0} + 4},
  \end{multline*}
  uniformly in $(\tau,\xi)$ where $|\tau| \geq \tau_0 > 0$, where $m_j$,
  $\tilde{m}_j$, $\ell_{0,0}$and $\tilde{\ell}_{0,0}$ have been defined in
  lemma~\ref{th:estimates-of-integrands}
  and~\ref{th:application-of-stationary-phase}, respectively. This estimate is
  uniform in $\eps \in [0,\eps_0]$.
  Since $m_4 \geq \ell_{0,0}$, we obtain claim (iii) with constant $K_0$
  depending in particular on $\tau_0$, $\eps_0$ and $\delta$ and with $M=m_4+ 5/2$.   
\end{proof}

\begin{remark}
  The estimate in proposition~\ref{th:limit_tilde_r_eps} item (iii) could be
  replaced by 
   \begin{equation*}
     \big|\big\langle \tilde{r}_{\eps}(\tau,\xi,\cdot), \phi\big\rangle\big|
     \leq K_0(\tau)  (1 + \tau^2 +\xi^2)^M, \qquad \tau \not=0,
   \end{equation*}
   but $K_0(\tau)$ is not bounded near $\tau=0$. Hence our argument does
   not allow any conclusion for $\tau=0$ and we have excluded all frequencies in
   $|\tau| < \tau_0$ with $\tau_0$ arbitrarily small and fixed.  
\end{remark}

\subsection{Solution of the linear Vlasov equation for the magnetized case} 
\label{sec:solution-Vlasov-magnetized}

We address equation~(\ref{eq:magnetized-uniform-g}) for a given
species $\species$, and thus drop the subscript $\species$ for simplicity.
Particularly, we shall denote by $Q_j$ the differential operator
$Q_{\species,j}$ defined in equation~(\ref{eq:Qj-operator}), which depends on the
equilibrium distribution function and the sign of the electric charge of the
consider particle species.

We shall first address the existence of solutions of the linear Vlasov
equation~(\ref{eq:magnetized-uniform-g-damped}) for the time-derivative of the
distribution function including a damping term and address its dissipation-less
limit. The result will then be used to compute the time-derivative of the
induced current~(\ref{eq:djnu_dt}).

\begin{theorem}
  \label{th:magnetized-uniform-g}
  Let $\eps > 0$ and let $\nu$ be any function satisfying 
  conditions~(\ref{eq:nu-conditions}).
  Then equation~(\ref{eq:magnetized-uniform-g-damped}) has a solution  
  $g_\eps \in \Schwartz(\R^7)$ which is unique as an element of
  $\Schwartz^\prime(\R^7)$ and:  
  
  \begin{itemize}

  \item[(i)] The Fourier transform of the unique tempered solution $g_\eps$ is 
    \begin{equation*}
      \hat{g}_\eps (\omega,k,u) =
      \frac{q n_0}{(m c)^4} \sum_{j=1}^3 \hat{E}_j(\omega,k)
      Q_j (\tau,\xi,u,\partial_\xi) r^\pm_\eps(\tau,\xi,u),
    \end{equation*}
    where the sign $+$ (resp. $-$) is chosen for $q >0$ (resp. $q <0$).

  \item[(ii)] For $(u_1,u_2) \not=(0,0)$,
    \begin{equation*}
      \hat{g}_\eps (\omega,k,u) = i\frac{q n_0}{(m c)^4}
      \sum_{j=1}^3 \hat{E}_j(\omega,k)
      \sum_{n\in\Z} \frac{Q_j (\tau,\xi,u,\partial_\xi)
        A^\pm_n(b)}{a_\eps - n} e^{+in\phi},
    \end{equation*}
    where $\phi$ is defined in~(\ref{eq:cylindrical}) and $b = (b_1,b_2)$
    in~(\ref{eq:a-and-bs}). 
    
  \end{itemize}
\end{theorem}

\begin{remark}
  The expression of the solution in item (i) in terms of the integral
  $r^\pm_\eps$ defined in (\ref{eq:r-pm-eps}) was first proposed by 
  Qin et al. \cite{Qin2007}, cf. also the subsequent discussion in the
  literature \cite{Lerche2008, Qin2008}.   
\end{remark}

\begin{remark}
  \label{rem:resonances}
  The solution $\hat{g}_\eps$ given in item (ii) is in agreement
  with the standard expression obtained in the physics literature
  \cite{Bornatici1983, Brambilla1998, Stix1992}.
  For $\eps \to 0^+$ it exposes a
  countable number of poles for $a_0 \in \Z$ where $\sin(\pi a_0) = 0$. 
  These poles correspond to cyclotron resonances briefly discussed in
  section~\ref{subsec:eq:a0}. 
\end{remark}

\begin{proof}[Proof of theorem \ref{th:magnetized-uniform-g}]
  We look for tempered solutions of
  equation~(\ref{eq:magnetized-uniform-g-damped}) and 
  thus we can equivalently consider the Fourier transform $\hat{g}_\eps$ of
  the distribution function $g_\eps$. If we define
  \begin{equation}
    \label{eq:h-eps}
    \hat{h}_\eps = e^{\pm i(\xi_2 u_1 - \xi_1 u_2)} \hat{g}_\eps,
  \end{equation}
  where $\tau = \omega/\omega_{c,\species}$, $\xi = ck/\omega_{c,\species}$, and
  with sign $+$ (resp. $-$) for $q>0$ (resp. $q<0$), the Fourier transform of
  equation~(\ref{eq:magnetized-uniform-g-damped}) is equivalent to
  equation~(\ref{eq:generic-magnetized-kinetic}) with source  
  \begin{equation*}
    \hat{s}^\pm = \frac{|q| n_0}{(m c)^4} e^{\pm i (\xi_2 u_1 - \xi_1 u_2)}
    \sum_{j=1}^3 \hat{E}_j Q_j(\tau,\xi,u,\partial_\xi)
    e^{\pm i (\xi_2 u_1 - \xi_1 u_2)},
  \end{equation*}
  and with the sign chosen according to the electric charge $q$. (As in the
  proof of proposition \ref{th:lemma-4.0}, in the case $q<0$ we can multiply by
  $-1$ the equation for $\hat{h}_\eps$ in order recast it into the form of
  (\ref{eq:generic-magnetized-kinetic}); hence $a = - a_\eps$.) Under the
  hypotheses, $s^\pm$ belongs to $\Schwartz(\R^7)$.
  Proposition~\ref{th:psi-psi-Schwartz} gives a solution
  $\hat{h}_\eps \in \Schwartz(\R^7)$,
  and thus $\hat{g}_\eps \in \Schwartz(\R^7)$. This is the unique solution in
  $C^\infty$, as proven in proposition~\ref{th:psi-psi-Schwartz}, as well as in
  $\Schwartz^\prime(\R^7)$, as proven in proposition~\ref{th:uniqueness-magnetized}. 
  
  (i) Substitution of the claimed expression $\hat{g}_\eps$
  into~(\ref{eq:h-eps}) gives
  \begin{equation*}
    \hat{h}_\eps(\omega,k,u) = \frac{q n_0}{(m c)^4}
    \sum_{j=1}^3 \hat{E}_j(\omega,k) e^{\pm i(\xi_2 u_1 - \xi_1 u_2)}
    Q_j (\tau,\xi,u,\partial_\xi) r^\pm_\eps(\tau,\xi,u),
  \end{equation*}
  where $\tau = \omega/\omega_{c,\species}$, $\xi = ck/\omega_{c,\species}$,
  and we observe that the operator $e^{\pm i(\xi_2u_1-\xi_1 u_2)} Q_j$
  commutes with $u_2\partial_{u_1} - u_1 \partial_{u_2}$. Then upon using
  lemma~\ref{th:lemma-4.0}, we have that $\hat{h}_\eps$
  solves~(\ref{eq:generic-magnetized-kinetic}) and thus $\hat{g}_\eps$ is the
  solution of equation~(\ref{eq:magnetized-uniform-g}).  
  
  (ii) Upon using coordinates (\ref{eq:cylindrical}), for either choice of the
  sign of the particle charge the equation for $\hat{h}_\eps$ reduces to
  (\ref{eq:ode}) with $\tilde{a} = a_\eps$ and with source
  \begin{equation*}
    V(\phi) = \frac{q n_0}{(mc)^4}
    e^{\pm i(\xi_2u_1-\xi_1 u_2)} \sum_{j=1}^3 \hat{E}_j Q_j e^{\pm i(\xi_2u_1 - \xi_1u_2)}.
  \end{equation*}
  We evaluate the Fourier coefficients of the source, that is,
  \begin{align*}
    \hat{V}_n &= \frac{q n_0}{(mc)^4} \frac{1}{2\pi} \int_0^{2\pi} e^{-in\phi}
    e^{\pm i(\xi_2u_1-\xi_1 u_2)} \sum_{j=1}^3 \hat{E}_j Q_j 
    e^{\pm i(\xi_2u_1 - \xi_1u_2)} d\phi \\
    &= \frac{q n_0}{(mc)^4}
    e^{\pm i(\xi_2u_1-\xi_1 u_2)} \sum_{j=1}^3 \hat{E}_j Q_j
    \frac{1}{2\pi} \int_0^{2\pi} e^{-in\phi}
    e^{\pm i(\xi_2u_1 - \xi_1u_2)} d\phi,
  \end{align*}
  and, in the second identity, we have used the fact that the coefficients of the
  operator $e^{\pm i(\xi_2u_1-\xi_1 u_2)} Q_j$ are independent of $\phi$,
  i.e., the exponential factor is canceled by the corresponding factor in
  the definition of the coefficients in equations~(\ref{eq:psi-coeff}).
  It is now sufficient to compute the Fourier coefficients of
  \begin{equation*}
    e^{\pm i(\xi_2u_1 - \xi_1u_2)} = e^{\pm i(b_2 \cos\phi \pm b_1 \sin\phi)},
  \end{equation*}
  and those are equal to $A_n^\pm(b)$ as we have shown in
  equation~(\ref{eq:Anpm-der}). Then
  \begin{equation*}
    \hat{V}_n = \frac{q n_0}{(mc)^4} e^{\pm i(\xi_2u_1-\xi_1 u_2)}
    \sum_{j=1}^3 \hat{E}_j Q_j A^\pm_n.
  \end{equation*}
  The Fourier expansion in lemma
  \ref{th:ode-existence-main} then yields the claimed identity.
\end{proof}

We shall now show that the solution $g_\eps$ for $\eps \to 0^+$ approaches the
causal solution of (\ref{eq:magnetized-uniform-g}) as computed by integration
along the characteristics curves. This result in addition clarifies the relation
between the analytical expressions for the solution in Fourier and physical
spaces. In the position-momentum variables, the characteristic curve
$t' \mapsto \big(X(t'; t,x,p), P(t';t,x,p) \big)$ with terminal condition
$(x,p)$ at $t'=t$ is \cite{Bornatici1983}    
\begin{equation}
  \label{eq:magnetized-characteristics}
  \left\{
  \begin{aligned}
    X_1(t'; t,x,p) &= x_1
    - \frac{p_1}{m\gamma \Omega} \sin\big(\Omega \cdot (t-t') \big)
    - \frac{p_2}{m\gamma \Omega} \big[\cos\big(\Omega \cdot (t-t') \big) -1\big],
    \\
    X_2(t'; t,x,p) &= x_2
    + \frac{p_1}{m\gamma \Omega} \big[\cos\big(\Omega \cdot (t-t') \big) -1\big]
    - \frac{p_2}{m\gamma \Omega} \sin\big(\Omega \cdot (t-t') \big),
    \\
    X_3(t'; t,x,p) &= x_3 - v_3 (t-t'),
    \\
    P_1(t'; t,x,p) &= p_1\cos\big(\Omega \cdot (t-t') \big)
    - p_2 \sin\big(\Omega \cdot (t-t') \big),
    \\
    P_2(t'; t,x,p) &= p_1 \sin\big(\Omega \cdot (t-t') \big)
    + p_2 \cos\big(\Omega \cdot (t-t') \big),
    \\
    P_3(t'; t,x,p) &= p_3,
  \end{aligned}\right.
\end{equation}
where $\gamma$ and thus $\Omega$ are constants of motion. Under
assumption~(\ref{eq:nu-conditions}) $\nu$ is constant along the characteristics
as well. The characteristic flow has the semi-group property: 
for every $t_1, t_2, t_3$ with $t_1 \leq t_2 \leq t_3$ and $(x,p)$,
\begin{equation}
  \label{eq:magn-semi-group}
  \begin{aligned}
    X(t_1; t_2, X(t_2; t_3,x,p), P(t_2; t_3,x,p) ) &= X(t_1; t_3,x,p), \\
    P(t_1; t_2, X(t_2; t_3,x,p), P(t_2; t_3,x,p) ) &= P(t_1; t_3,x,p),
  \end{aligned}
\end{equation}
which can be verified directly from~(\ref{eq:magnetized-characteristics}). For
$E$ and $F_0$ at least in $C^1$, let
$\mathfrak{s} = -q (\partial_t E - v \times \nabla \times  E) \cdot \nabla_p F_0$ for
brevity, and $F_0 = n_0 G$ as in (\ref{eq:gyrotropic}). We shall always imply
the relation $u = p / (mc)$ between normalized and physical momentum variables.

\begin{proposition}
  \label{th:magnetized-characteristics}
  For every $\eps > 0$, $E \in [\Schwartz(\R^4)]^3$,
  $F_0 = n_0 G \in \Schwartz(\R^3)$, and $g_{0,\eps} \in \Schwartz(\R^6)$,
  the Cauchy problem
  for equation~(\ref{eq:magnetized-uniform-g}) with initial condition
  $g_{0,\eps}$ at time $t=t_0$ has a unique classical solution
  $g_\eps \in C^\infty(\R^7)$ and it holds that:
  \begin{itemize}
  \item[(i)] There is a unique Cauchy datum for which the solution 
    $g_\eps \in \Schwartz^\prime$ and that is given by
    \begin{equation*}
      g_{*,\eps}(x,u) \coloneqq \int_{-\infty}^{t_0} e^{-\eps\nu(t_0-t')}
      \mathfrak{s}\big(t', X(t';t_0,x,p), P(t';t_0,x,p) \big)dt',
    \end{equation*}
    which belongs to $\Schwartz(\R^6)$.
  \item[(ii)] For $g_{0,\eps} = g_{*,\eps}$ the solution is
    \begin{equation*}
      g_\eps (t,x,u) = \int_{-\infty}^t e^{-\eps\nu(t-t')}
      \mathfrak{s}\big(t', X(t';t,x,p), P(t';t,x,p) \big)dt',
    \end{equation*}
    and $\hat{g}_\eps$ is the same function defined in
    theorem~\ref{th:magnetized-uniform-g} (i).
  \item[(iii)] For $\eps \to 0^+$, the function $g_\eps$ has a pointwise limit
    \begin{equation*}
      g (t,x,u) = \int_{-\infty}^t
      \mathfrak{s}\big(t', X(t';t,x,p), P(t';t,x,p) \big)dt',
    \end{equation*}
    for every $(t,x,u) \in \R^7$. The limit $g$ belongs to $C^\infty_b(\R^7)$,
    it is a classical solution of equation~(\ref{eq:magnetized-uniform-g}),
    and it is independent of the choice of $\nu$.
  \end{itemize}
\end{proposition}

\begin{remark}
  The first part of this statement is essentially Wollman's result on the linear
  Vlasov equation \cite[theorem 3.1]{Wollman1987}, but for the case of a uniform
  plasma equilibrium and with somewhat relaxed hypotheses on the support of the
  initial datum.
\end{remark}

\begin{proof}
  By hypothesis the damping function $\nu$  is constant along the characteristics. 
  The application of the standard method of characteristics \cite{John1980} gives
  the unique classical solution of the Cauchy problem, 
  \begin{multline*}
    g_\eps (t,x,u) = g_{0,\eps} \big(X(t_0;t,x,p), P(t_0;t,x,p) \big)
    e^{-\eps \nu (t-t_0)} \\
    + \int_{t_0}^t e^{-\eps \nu(t-t')}
    \mathfrak{s}\big(t', X(t';t,x,p), P(t';t,x,p) \big)dt',
  \end{multline*}
  where $X$ and $P$ are given in equation~(\ref{eq:magnetized-characteristics}),
  and $g_\eps \in C^\infty$ follows by Leibniz rule for differentiation
  of integrals. This completes the first part of the proposition.

  Since $e^{-\eps\nu(t-t')}$ is integrable for $t' \in (-\infty, t]$, the integral
  \begin{equation*}
    H_\eps(t, x, p) = \int_{-\infty}^t e^{-\eps\nu(t-t')}
    \mathfrak{s}\big(t', X(t';t,x,p), P(t';t,x,p) \big)dt'
  \end{equation*}
  is finite. If the electric field is given in Fourier transform, using
  \begin{equation*}
    E(t,x) = \frac{1}{(2\pi)^4} \int_{\R^4} e^{-i\omega t + i k\cdot x}
    \hat{E}(\omega,k) d\omega dk,
  \end{equation*}
  and applying Fubini's theorem, we get 
  \begin{equation*}
    H_\eps(t, x, p) = \frac{1}{(2\pi)^4}
    \int_{\R^4} e^{-i\omega t + i k \cdot x}
    \tilde{H}_\nu(t,\omega,k,p) d\omega dk,
  \end{equation*}
  with, using $F_0 = n_0 G$,
  \begin{equation*}
    \tilde{H}_\eps(t,\omega,k,p) = \frac{q n_0}{(m c)^4}
    \sum_{j=1}^3 \hat{E}_j(\omega,k) Q_j (\tau,\xi,u,\partial_\xi) 
    R^\pm_\eps (t, \omega, k, p),
  \end{equation*}
  where $Q_j$ are the operators defined in (\ref{eq:Qj-operator}) for the
  considered particle species, and
  \begin{multline*}
    R^\pm_\eps(t,\omega,k,p) = \frac{\omega_c}{\gamma} \int_{-\infty}^t
    e^{-\eps\nu (t-t') + i \omega (t-t') -i k_\parallel v_\parallel (t-t')} \\
    e^{\pm i (\xi_2 u_1 - \xi_1 u_2) \cos(\frac{\omega_c}{\gamma}(t-t'))
      -i (\xi_1 u_1 + \xi_2 u_2)\sin(\frac{\omega_c}{\gamma} (t-t'))} dt',
  \end{multline*}
  with sign chosen according to the particle charge $q$.
  The change of variable $\lambda = \frac{\omega_c}{\gamma} (t-t')$ shows that
  $R^\pm_\eps$ is actually independent of time $t$ and
  $R^\pm_\eps(t,\omega,k,p) = r^\pm_\eps(\tau, \xi, u)$ where $r_\eps^\pm$ has been
  defined in~(\ref{eq:r-pm-eps}). Then, $\tilde{H}_\eps$ is also independent of
  time and $\tilde{H}_\eps(t,\omega,k,p) = \hat{g}_\eps(\omega,k,p)$ with
  $\hat{g}_\eps$ the unique tempered solution established in theorem
  \ref{th:magnetized-uniform-g} (i). We have $\hat{g}_\eps \in \Schwartz$ and
  thus $H_\eps$ is the inverse Fourier transform of a Schwartz function,
  so that $H_\eps \in \Schwartz(\R^7)$.
  By definition $g_{*,\eps}(x,p)=H_\eps(t_0,x,p) = g_\eps(t_0,x,p)$, hence
  $g_{*,\eps} \in \Schwartz$ and it is the Cauchy datum of the unique 
  tempered solution $g_\eps = H_\eps$. This proves (i) and item (ii) follows
  from the expression for $H_\eps$. 

  As for the pointwise limit of the solution, item (iii),
  we observe that the characteristic flow and the source term $\mathfrak{s}$
  satisfy the hypothesis of proposition~\ref{th:causal}; in particular, the
  flow~(\ref{eq:magnetized-characteristics}) satisfies assumptions (i)-(iv) of
  appendix~\ref{sec:causal-solutions}, while $\mathfrak{s}$ is defined as the
  sum of the products of a function in $\Schwartz(\R^4)$ in frequency and
  wave-vector and a function in $\Schwartz(\R^3)$ in momentum, hence
  $\mathfrak{s}\in\Schwartz(\R^7)$. The function $g$ is then the causal solution
  of~(\ref{eq:vlasov-magnetized-uniform}) in the sense of
  appendix~\ref{sec:causal-solutions} and proposition~\ref{th:causal} gives
  $g \in C_b^\infty(\R^7)$. For every $m \in \N_0$
  one has the estimate
  \begin{multline*}
    (1 + s^2)^m \big|\mathfrak{s}\big(s, X(s;t,x,p), P(s;t,x,p)\big)\big| \\
    \leq (1 + s^2 + X(s;t,x,p)^2 + P(s;t,x,p)^2)^m
    \big|\mathfrak{s}\big(s, X(s;t,x,p), P(s;t,x,p)\big)\big|,
  \end{multline*}
  which gives
  \begin{equation*}
    \big|\mathfrak{s}\big(s, X(s;t,x,p), P(s;t,x,p)\big)\big| \leq \
    \frac{\|\mathfrak{s}\|_{2m}}{(1+s^2)^m},
  \end{equation*}
  where $\|\cdot\|_m$ are the semi-norms defined in appendix~\ref{sec:basic-def}.
  Therefore, for every $(t,x,p)$ and for $m$ large enough, dominated convergence
  allows us to pass to the limit $\eps \to 0^+$ in the integrand.
\end{proof}

\subsection{Current density and conductivity operator} 
\label{sec:Limiting-absorption-magnetized}

First we show that the functions $r_\eps^\pm$ in
lemma~\ref{th:lemma-4.0} have a limit in $\Schwartz^\prime$. For simplicity let  
\begin{equation}
  \label{eq:zeta-pm}
  \zeta^\pm(\xi_1,\xi_2,u_1,u_2,\lambda) = (\xi_1 u_1 + \xi_2 u_2) \sin\lambda 
  \mp (\xi_2 u_1 - \xi_1 u_2) \cos\lambda,
\end{equation}
and we recall the definition of $a_0$ in equation~(\ref{eq:a-and-bs}) and
normalized Fourier variables~(\ref{eq:norm-omega-k}). 

\begin{proposition}
  \label{th:r_0}
  For every integer $m\geq 2$,
  \begin{itemize}

  \item[(i)] the linear map 
    \begin{align*}
      \psi &\mapsto  
      \int \int_0^1 e^{ia_0(\tau, \xi, u) \lambda
        - i \zeta^\pm(\xi_1, \xi_2, u_1,u_2, \lambda)}
      d\lambda \psi(\tau,\xi,u) d\tau d\xi du \\
      & + i^m \int \int_1^{+\infty} \lambda^{-m}
      e^{ia_0(\tau, \xi, u) \lambda - i \zeta^\pm(\xi_1, \xi_2, u_1,u_2, \lambda)} d\lambda
      \frac{\partial_\tau^m \psi(\tau,\xi,u)}{\gamma(u)^m} 
      d\tau d\xi du,
    \end{align*}
    is continuous on $\Schwartz(\R^7)$ and thus defines a
    tempered distribution $r^{\pm,m}_0$;
    
  \item[(ii)] for $\eps\to 0^+$, $r_\eps^\pm$ has a limit $r^\pm_0$ in
    $\Schwartz^\prime$, and $r^\pm_0 = r^{\pm,m}_0$ for all $m \geq 2$. The limit
    is independent of the choice of the function $\nu_\species$.

  \end{itemize}
\end{proposition}

\begin{proof}
  The two terms defining the linear map on $\psi$ are bounded by the norms
  $\|\psi\|_8$ and $\|\psi\|_{m+8}$, respectively, and this shows continuity in
  the topology of $\Schwartz(\R^7)$. Then the map is a tempered distribution
  which is denoted by $r_0^{\pm,m}$. Since
  \begin{equation*}
    \frac{(-i)^m}{\gamma^m \lambda^m}
    \partial_\tau^m e^{ia_\eps \lambda - i\zeta^\pm} =
    e^{ia_\eps \lambda - i\zeta^\pm},
  \end{equation*}
  after integration by parts in $\tau$,
  \begin{multline*}
    \langle r_\eps^\pm, \psi\rangle = \int \int_0^1
    e^{ia_\eps \lambda - i\zeta^\pm} d\lambda \psi d\tau d\xi du \\
    + i^m \int \Big[\int_1^{+\infty} \frac{e^{ia_\eps \lambda - i\zeta^\pm}}{\lambda^m}
    d\lambda \Big] \frac{\partial_\tau^m \psi}{\gamma^m} d\tau d\xi du,
  \end{multline*}
  and if $m\geq 2$ the integrand is uniformly bounded by an integrable
  function. We can then pass to the limit in the integral and obtain
  \begin{equation*}
    \langle r_\eps^\pm, \psi\rangle \xrightarrow[]{\eps \to 0^+}
    \langle r^{\pm,m}_0, \psi\rangle,
  \end{equation*}
  for every $\psi \in \Schwartz(\R^7)$ and every $m\geq 2$.
  Uniqueness of the limit implies that all distributions $r^{\pm,m}_0$ for $m\geq 2$
  are equal to the limit $r^\pm_0 = \lim_{\eps \to 0^+} r_\eps^\pm$.
\end{proof}

This result is already sufficient to compute the limit of the time derivative
of the induced current density as a tempered distribution under fairly general
assumptions. In the following we define, for any $\species$ labeling a particle
species, 
\begin{equation*}
  r_{\species,\eps} (\tau,\xi,u) \coloneqq r^{\pm}_\varepsilon(\tau, \xi,\nu),
  \quad r_{\species} \coloneqq r^{\pm}_0,
\end{equation*}
with sign $+$ (resp. $-$) for $q_\species >0$ (resp. $q_\species < 0$), with 
$r_\eps^\pm$ defined in equation~(\ref{eq:r-pm-eps}), and with $r^\pm_0$ being
the distributional limit in lemma~\ref{th:r_0}.

By definition, cf. equation~(\ref{eq:djnu_dt}), the current density associated
to $\{g_{\species,\eps}\}$ is
\begin{equation*}
  \partial_t j_\eps = \mathbb{K}(\{g_{\species,\eps}\}).
\end{equation*}
We recall that $\omega_{p,\species} = \sqrt{4\pi q_\species^2 n_{\species,0} / m_\species}$
is the plasma frequency of the $\species$-th species. We also need the operator
$Q_{\species,j}$ defined in equation~(\ref{eq:Qj-operator}).
Relations~(\ref{eq:norm-omega-k}) are also implied.

\begin{proposition}
  \label{th:magnetized-j-eps}
  The function $\partial_t j_\eps$ belongs to $[\Schwartz(\R^4)]^3$ and
  \begin{subequations}
    \label{eq:chi-eps-tensor}  
    \begin{equation}
      \label{eq:FT-Ohms-law-magnetized}
      \widehat{\partial_t j_\eps}(\omega,k) = \hat{\varsigma}_{\eps} (\omega,k)
      \hat{E}(\omega,k),
    \end{equation}
    where the tensor $\hat{\varsigma}_\eps$ is given component-wise by
    \begin{equation}
      \label{eq:varsigma-eps-integral}
      \hat{\varsigma}_{\eps,ij}(\omega,k) =
      \sum_\species \frac{\omega_{p,\species}^2}{4\pi}
      \int_{\R^3} \frac{1}{\gamma(u)} u_i
      Q_{\species,j} (\tau,\xi,u,\partial_\xi) r_{\species,\eps} (\tau,\xi,u)du.
    \end{equation}
  \end{subequations}
\end{proposition}

\begin{proof}
  In theorem \ref{th:magnetized-uniform-g} we have shown that
  $g_{\species,\eps} \in \Schwartz(\R^7)$, hence for every $l,m,n \in \N_0$ and
  $\alpha \in \N_0^4$,
  \begin{equation*}
    (1+t^2+x^2)^m (1+u^2)^n
    \big|\partial_t^l \partial_x^\alpha g_{\species,\eps}(t,x,u)\big|
      \leq \|g_{\species,\eps}\|_{|\alpha|+l+2m+2n}.
  \end{equation*}
  Since for $n > 3/2$, $(1+u^2)^{-n}$ is integrable, from equation~(\ref{eq:djnu_dt})
  we deduce that $\partial_t j_\eps = \mathbb{K}(\{g_{\species,\eps}\})$ is
  $C^\infty$ and the same inequality implies that all derivatives are rapidly
  decaying at infinity, that is, $\partial_t j_\eps \in [\Schwartz(\R^4)]^3$.

  Then the Fourier transform of $\partial_t j_\eps$ exists in the classical sense
  and, by Fubini's theorem, we have 
  \begin{equation*}
    \widehat{\partial_t j_\eps}(\omega,k) = \sum_\species q_\species
    (m_\species c)^3 \int v_\species(u) \hat{g}_{\species,\eps} (\omega,k,u) du.
  \end{equation*}
  Upon using the expression for $\hat{g}_{\species,\eps}$ in 
  theorem~\ref{th:magnetized-uniform-g} (i), we arrive at
  equation~(\ref{eq:chi-eps-tensor}). 
\end{proof}

\begin{remark}
  Convergence of the integral in~(\ref{eq:varsigma-eps-integral}) is ensured by
  the rapid decay in $u$ at infinity of the coefficients~(\ref{eq:test-functions})
  in the operators $Q_{\species,j}$.
\end{remark}

We recall that $Q_{\species,j}$ are first-order partial differential operators
and we can define the formal adjoint $Q^\prime_{\species,j}$ by
\begin{equation*}
  \int_{\R^7} \chi_1(\tau,\xi,u) Q_{\species,j} \chi_2(\tau,\xi,u) d\tau d\xi du
  = \int_{\R^7} \chi_2(\tau,\xi,u) Q^\prime_{\species,j}
  \chi_1(\tau,\xi,u) d\tau d\xi du, 
\end{equation*}
for every $\chi_1, \chi_2 \in \Schwartz(\R^7)$. The adjoint
$Q^\prime_{\species,j}$ is again a first-order partial differential operator and
it is continuous from $\Schwartz \to \Schwartz$. Therefore we can define the
distribution $\partial_t j \in [\Schwartz^\prime(\R^4)]^3$ by
\begin{equation*}
  \langle \widehat{\partial_tj}, \hat{\varphi} \rangle \coloneqq
  \sum_{j=1}^3 \sum_\species 
  \frac{\omega_{c,\species}^4 \omega_{p,\species}^2}{4\pi c^3}
  \langle r_\species , Q_{\species,j}' \big(
  \frac{\hat{\varphi} \cdot u }{\gamma} \hat{E}_j\big)\rangle,
\end{equation*}
for every vector test-function $\hat{\varphi} \in [\Schwartz(\R^4)]^3$.

\begin{proposition}
  \label{th:jnu-convergence}
  As $\eps \to 0^+$, $\partial_t j_\eps \to \partial_t j$
  in $[\Schwartz^\prime(\R^4)]^3$ and the limit satisfies 
  $\partial_t j = \mathbb{K}(\{g_\species\}) \in [C_b^\infty(\R^4)]^3$.
\end{proposition}

\begin{proof}
  Step 1: $\partial_t j_\eps \to \partial_t j$. For $\eps > 0$, one has
  \begin{equation*}
    \langle \widehat{\partial_tj_\eps}, \hat{\varphi} \rangle =
    \sum_\species \frac{\omega_{c,\species}^4 \omega_{p,\species}^2}{4\pi c^3}
    \sum_{j=1}^3 \langle r_{\species,\eps}, Q_{\species,j}' \big(
    \frac{\hat{\varphi} \cdot u}{\gamma} \hat{E}_j \big)\rangle,
  \end{equation*}
  where the factor $\omega_{c,\species}^4 c^{-3}$ comes for the chance of variable
  $d\omega dk = (\omega_{c,\species}^4/c^3) d\tau d\xi$.
  The result then follows from the convergence of $r_{\species,\eps}$ in
  $\Schwartz^\prime$ proven in proposition~\ref{th:r_0}.
  
  Step 2: $g_\species$ belongs to the domain of the operator $\mathbb{K}$.
  Wee observe that the flow (\ref{eq:magnetized-characteristics}) is
  polynomially bounded and thus satisfies the hypothesis of
  lemma~\ref{th:finite-moments}. Then, since $|v(u)| = c|u|/\gamma(u) \leq c$,
  for every $n,l \in \N_0$ and $\alpha \in \N^d_0$,
  \begin{equation}
    \label{eq:moment-estimate}
    (1+u^2)^n |v(u) \partial_t^l \partial_x^\alpha g_{\species}(t,x,u)|
    \leq c (1+u^2)^n |\partial_t^l \partial_x^\alpha g_{\species}(t,x,u)| \leq C,
  \end{equation}
  and upon choosing $n>3/2$ we deduce that
  $u \mapsto v(u) \partial_t^l \partial_x^\alpha g_\species(t,x,u)$ is bounded
  by an $L^1$-function. For $l = 0$ and $\alpha = 0$, this shows that
  $g_\species$ has finite first velocity moment. Iterated application of the
  dominated convergence theorem also shows that
  $\mathbb{K}(\{g_\species\}) \in C_b^\infty(\R^4)$.

  Step 3: $\partial_t j = \mathbb{K}(\{g_\species\})$.
  Let $\widetilde{\partial_t j} = \mathbb{K}(\{g_\species\})$. For every
  $\varphi \in [\Schwartz(\R^4)]^3$,
  \begin{equation*}
    \langle \partial_t j - \widetilde{\partial_t j}, \varphi\rangle =
    - \langle \partial_t j_\eps - \partial_t j, \varphi \rangle 
    + \langle \mathbb{K}(\{g_{\species,\eps} - g_\species\}), \varphi \rangle.
  \end{equation*}
  The first terms on the right-hand side converges to zero for $\eps \to 0^+$.
  The second term is estimated by
  \begin{equation*}
    \langle \mathbb{K}(\{g_{\species,\eps} - g_\species\}), \varphi \rangle =
    \sum_\species q_\species (m_\species c)^3 \int_{\R^4} \int_{\R^3}
    (g_{\species,\eps}-g_\species)(t,x,u) \varphi(t,x) \cdot v(u) du dt dx.
  \end{equation*}
  We know that $g_{\species,\eps} \in \Schwartz$ and upon using
  estimate~(\ref{eq:moment-estimate}) with $l=0$ and $\alpha=0$, we can
  bound the integrand by an $L^1$-function independent on $\eps$ so thus pass to
  the limit in the integrand. At last
  from proposition~\ref{th:magnetized-characteristics}, we have
  $g_{\species,\eps} - g_\species \to 0$.
\end{proof}

Proposition~\ref{th:jnu-convergence} is sufficient to prove the continuity of
$\varsigma : [\Schwartz(\R^4)]^3 \to [\Schwartz'(\R^4)]^3$,
cf. equation~(\ref{eq:sigma-2-2}). However, we wish to understand under which
conditions $\varsigma$ reduces to a Fourier multiplier and, when this is the
case, we wish to address the regularity of its symbol.

Since the limit $\partial_t j$ is independent of the choice of the damping
function, we choose $\nu_\species$ such that $\kappa_\species = 1$;
one can check that this satisfies condition~(\ref{eq:nu-conditions}).

We start from equation~(\ref{eq:chi-eps-tensor}) which can be rewritten as
\begin{multline}
  \label{eq:varsigma-eps}
  \hat{\varsigma}_{\eps,ij}(\omega,k) = \sum_\species
  \frac{\omega_{p,\species}^2}{4\pi} \Big[
    \int_{\R^3} \frac{u_i}{\gamma} \Phi_jr_{\species,\eps}(\tau,\xi,u)
    \mathscr{F}_\species(u) du \\
    +\int_{\R^3} \frac{u_i}{\gamma} \Psi_jr_{\species,\eps}(\tau,\xi,u)
    \mathscr{G}_\species(u) du\Big].
\end{multline}
Since $a_\eps$ depends on $\xi$ through $\xi_3$ only, we have
\begin{equation}
  \label{eq:Phi-r-eps}
  \Phi_j r_{\species,\eps} = 
  \frac{i e^{-i\pi a_\eps} \Phi_j P^\pm_{\eps}}{2\sin(\pi a_\eps)}, \qquad
  \Psi_j r_{\species,\eps} = 
  \frac{ie^{-i\pi a_\eps} \Psi_j P^\pm_{\eps}}{2\sin(\pi a_\eps)},
\end{equation}
with sign $+$ (resp., $-$) for $q_\species >0$ (resp., $q_\species<0$).

We shall now apply proposition~\ref{th:limit_tilde_r_eps}, which holds for
generic functions of the form (\ref{eq:r-tilde}) with $\tilde{P}_\eps$ subject
to condition~(\ref{eq:polynomial-growth}), to the case of interest, which is
given in equations~(\ref{eq:varsigma-eps}) and~(\ref{eq:Phi-r-eps}). We choose a
normalization time-scale common to all particle species, that is,
$\hat{\omega} \coloneqq \max_\species \omega_{c,\species}$. 

\begin{proposition}
  \label{th:main-4-2}
  Under the same hypotheses as in theorem~\ref{th:magnetized-uniform-g}, for
  $\omega \not=0$, the tensor $\hat{\varsigma}_{\eps}(\omega,k)$
  defined in~(\ref{eq:chi-eps-tensor}), has a pointwise limit
  $\hat{\varsigma}_0(\omega,k)$, which can be computed,
  cf. equation~(\ref{eq:varsigma-ph}), and such that:
  \begin{itemize}
  \item[(i)] The limit $\hat{\varsigma}_0$ is continuous on
    $(\R \setminus \{0\}) \times \R^3$
    and it is $C^\infty$ where $\omega \not=0$ and
    $\omega^2\not=(c k_3)^2 + n^2 \omega_{c,\species}^2$
    for all $n \in \Z$ and species index $\species$.
  \item[(ii)] For any $\omega_0 >0$, there is a constant $K_0(\omega_0)$ for which
    \begin{equation*}
      |\hat{\varsigma}_0(\omega,k)| \leq K_0(\omega_0)
      (1 + (\tfrac{\omega}{\hat{\omega}})^2 + (\tfrac{ck}{\hat{\omega}})^2)^M,
    \end{equation*}
    uniformly for $(\omega,k) \in \R^4$, $|\omega|\geq \omega_0$, with exponent
    $M$ independent of $\omega_0$.  
  \item[(iii)] For all $E \in [\Schwartz(\R^4)]^3$ such that
    $\omega \not = 0$ in $\supp\hat{E}$
    we have $\partial_t j = \mathcal{F}^{-1} \big(\hat{\varsigma}_0 \hat{E}\big)$.
  \end{itemize}  
\end{proposition}

\begin{proof}
  We start from equation~(\ref{eq:varsigma-eps}) which reads
  \begin{equation*}
    \hat{\varsigma}_{\eps,ij}(\omega,k) = \sum_\species
    \frac{\omega_{p,\species}^2}{4\pi} \Big[
      \Big\langle \Phi_jr_{\species,\eps}(\tau,\xi,\cdot),
      \frac{1}{\gamma} u_i \mathscr{F}_\species \Big\rangle
      + \Big\langle  \Psi_jr_{\species,\eps}(\tau,\xi,\cdot),
      \frac{1}{\gamma} u_i \mathscr{G}_\species \Big\rangle\Big],
  \end{equation*}
  where $\tau = \omega/\omega_{c,\species}$ and $\xi = ck/\omega_{c,\species}$,
  cf. equation~(\ref{eq:norm-omega-k}).
  
  As observed in equations~(\ref{eq:Phi-r-eps}), both $\Phi_j r_{\species,\eps}$
  and $\Psi_j r_{\species,\eps}$ are special cases of~(\ref{eq:r-tilde})
  corresponding to $\tilde{P}_\eps = (i/2) e^{-i\pi a_\eps} \Phi_j P^\pm_{\eps}$ and
  $\tilde{P}_\eps = (i/2)e^{-i\pi a_\eps} \Psi_j P^\pm_{\eps}$, respectively.
  Lemma~\ref{th:condition-59} shows that in both cases $\tilde{P}_\eps$ satisfies
  condition~(\ref{eq:polynomial-growth}). Proposition \ref{th:limit_tilde_r_eps}
  gives, for any $(\tau,\xi)$ with $\tau \not=0$, distributions 
  $r^\Phi_{\species,j}(\tau,\xi,\cdot), r^\Psi_{\species,j}(\tau,\xi,\cdot) \in
  \Schwartz^\prime(\R^3)$ such that 
  \begin{equation*} 
    \Phi_j r_{\species,\eps}(\tau,\xi,\cdot) \to
    r^\Phi_{\species,j}(\tau,\xi,\cdot),
    \qquad
    \Psi_j r_{\species,\eps}(\tau,\xi,\cdot) \to
    r^\Psi_{\species,j}(\tau,\xi,\cdot),
  \end{equation*} 
  in the topology of $\Schwartz^\prime(\R^3)$. In this case, let 
  \begin{equation}
    \label{eq:varsigma-ph}
    \hat{\varsigma}_{0,ij}(\omega,k) = 
    \sum_\species \frac{\omega_{p,\species}^2}{4\pi}\Big[
      \Big\langle r_{\species,j}^\Phi(\tau,\xi,\cdot),
      \frac{1}{\gamma} u_i \mathscr{F}_\species \Big\rangle 
      + \Big\langle r^\Psi_{\species,j}(\tau,\xi,\cdot),
      \frac{1}{\gamma} u_i \mathscr{G}_\species \Big\rangle
      \Big],
  \end{equation}
  and we have
  $\hat{\varsigma}_{\eps,ij}(\omega,k) \to \hat{\varsigma}_{0,ij}(\omega,k)$
  pointwise in Fourier space where $\omega \not=0$.

  (i) Proposition~\ref{th:limit_tilde_r_eps} implies that 
  $\hat{\varsigma}_{0,ij}$ is continuous in $(\R \setminus \{0\}) \times \R^3$
  and $C^\infty$ where $\tau \not= 0$ and $\tau^2 \not= \xi_3^2 + n^2$ for all
  integers $n$, and for all particle species.

  (ii) We observe that for any $\species$,
  $\hat{\omega} / \omega_{c,\species} \geq 1$, hence,
  \begin{equation*}
    1 + \tau^2 + \xi^2 \leq \frac{\hat{\omega}^2}{\omega_{c,\species}^2}
    \Big(1 + (\frac{\omega}{\hat{\omega}})^2 + (\frac{ck}{\hat{\omega}})^2
    \Big).
  \end{equation*}
  Then proposition~\ref{th:limit_tilde_r_eps} item (iii) implies
  that there are constants $K_0, M > 0$ independent of $\eps$ such that
  \begin{equation*}
    \big|\hat{\varsigma}_\eps(\omega,k)\big| \leq K_0(\omega_0)
    \big(1 + (\tfrac{\omega}{\hat{\omega}})^2
    + (\tfrac{ck}{\hat{\omega}})^2\big)^M,
  \end{equation*}
  uniformly in $\eps \in [0,\eps_0]$ and $|\omega| \geq \omega_0$.

  (iii) We use the estimate proven in (ii), which implies
  \begin{equation*}
    \big|\hat{\varsigma}_\eps(\omega,k) - \hat{\varsigma}_0(\omega,k)\big|
    \leq 2K_0(\omega_0) \big(1 + (\tfrac{\omega}{\hat{\omega}})^2
    + (\tfrac{ck}{\hat{\omega}})^2\big)^M,
  \end{equation*}
  uniformly in $\eps \in [0,\eps_0]$ and for $|\omega| \geq \omega_0$. Since by
  hypothesis $\omega \not= 0$ and thus $\tau \not= 0$ in $\supp\hat{E}$, we can
  choose $\omega_0$ small enough that, for any $\hat{\varphi} \in[\Schwartz(\R^4)]^3$,
  \begin{multline*}
    \big|\big[
      \big(\hat{\varsigma}_\eps(\omega,\tau) - \hat{\varsigma}_0(\omega,k)\big)
      \hat{E}(\omega,k)\big] \cdot \hat{\varphi}(\omega,k)\big| \\
    \leq 2K_0(\omega_0) \big(1 + (\tfrac{\omega}{\hat{\omega}})^2
    + (\tfrac{ck}{\hat{\omega}})^2\big)^M
    \big|\hat{E}(\omega,k)\big| \big|\hat{\varphi}(\omega,k)\big|,
  \end{multline*}
  uniformly in $\eps \in [0,\eps_0]$ and $(\omega,k) \in \R^4$,
  $|\omega| \geq \omega_0$ .
  Since $E,\varphi \in [\Schwartz(\R^4)]^3$, the right-hand side of the above
  inequality belongs to $L^1(\R^4)$ and we have already established that
  $\hat{\varsigma}_\eps(\omega,\tau) - \hat{\varsigma}_0(\omega,k) \to 0$
  pointwise. Therefore the hypothesis of the dominated convergence theorem are
  satisfied and we can conclude
  \begin{equation*}
    \int_{\R^4} \big[\big(\hat{\varsigma}_\eps(\omega,\tau) -
    \hat{\varsigma}_0(\omega,k)\big) \hat{E}(\omega,k)\big] \cdot
    \varphi(\omega,k) d\omega dk \to 0,
  \end{equation*}
  for all $\hat{\varphi} \in[\Schwartz(\R^4)]^3$
  which means that $\hat{\varsigma}_\eps \hat{E} \to \hat{\varsigma}_0 \hat{E}$
  in $\Schwartz^\prime(\R^4)$. At last proposition~\ref{th:jnu-convergence}
  establishes that the $\Schwartz^\prime$-limit of
  $\hat{\varsigma}_\eps \hat{E}$ is exactly $\widehat{\partial_t j}$. 
\end{proof}

\subsection[Proof of the results of section~\ref{sec:main-2}]{%
  \for{toc}{Proof of the results of section~\ref{sec:main-2}}%
  \except{toc}{Proof of the results on the relativistic, three-dimensional case
    with uniform magnetic field (section~\ref{sec:main-2})}%
}

We can now collect the partial results obtained in this section and give a proof
of the main theorems stated in section~\ref{sec:main-2}.

\begin{proof}[Proof of theorem~\ref{th:main-2-1}.]
  (i) For each $\species$ and $\eps$, existence of a solution
  $g_{\species,\eps} \in \Schwartz(\R^7)$ of the Vlasov equation with damping
  and its uniqueness is established in theorem~\ref{th:magnetized-uniform-g}.
  We have shown that $\widehat{\partial_t j_\eps} \in [\Schwartz(\R^4)]^3$ in
  proposition~\ref{th:jnu-convergence}.

  (ii) Proposition~\ref{th:magnetized-characteristics} proves that
  $g_{\species,\eps}$ converges pointwise to $g_{\species} \in C_b^{\infty}$ and
  the limit is independent of the damping function $\nu$. Explicit expressions
  for $g_{\species,\eps}$ and $g_\species$ are given. 
  Using the inequality, cf. the proof
  of proposition~\ref{th:magnetized-characteristics},
  \begin{equation*}
    (1+s^2)^m \big|\mathfrak{s}\big(s, X(s;t,x,p), P(s;t,x,p)\big)\big| \leq 
    \|\mathfrak{s}\|_{2m},
  \end{equation*}
  yields the bound
  \begin{equation*}
    \big|g_{\species,\eps} (t,x,u)\big| \leq C,
  \end{equation*}
  uniformly in $\eps$ and thus for any $\varphi \in \Schwartz(\R^7)$,
  \begin{equation*}
    \big|\big(g_{\species,\eps}(t,x,u) - g_\species(t,x,u)\big)\varphi(t,x,u) \big| \leq
    2C \big|\varphi(t,x,u)\big|,
  \end{equation*}
  and $|\varphi|$ is in $L^1$. Hence the dominated convergence theorem yields
  \begin{equation*}
    \langle g_{\species,\eps} - g_\species, \varphi \rangle =
    \int \big(g_{\species,\eps}(t,x,u) - g_\species(t,x,u)\big)\varphi(t,x,u) dtdxdu
    \to 0,
  \end{equation*}
  for all $\varphi \in \Schwartz(\R^7)$. This is equivalent to
  $g_{\species,\eps} \to g_\species$ in $\Schwartz^\prime$. Existence of the
  $\Schwartz^\prime$-limit of $\partial_t j_\eps$ is proven in
  proposition~\ref{th:jnu-convergence}, in which we also show that the limit
  $\partial_t j$ is determined by $g_\species$ only, hence it is independent of $\nu$.

  (iii) The fact that the limit $g_\species$ is a classical solution of the
  linearized Vlasov equation in $C_b^\infty$ is proven
  in proposition~\ref{th:magnetized-characteristics}. At last, the identity
  $\partial_t j = \mathbb{K}(\{g_{\species}\})$ is shown in
  proposition~\ref{th:jnu-convergence}.
\end{proof}

\begin{proof}[Proof of theorem~\ref{th:main-2-2}.]
  After Fourier transform, the action of $\varsigma(E)$ on a test function is
  given in proposition~\ref{th:jnu-convergence}. Since $r_\species$ is a
  tempered distribution and $Q^\prime_{\species,j}$ is continuous from
  $\Schwartz \to \Schwartz$, we have
  \begin{equation*}
    \big|\langle \widehat{\partial_t j}, \hat{\varphi} \rangle\big| \leq
    C \|u\cdot\hat{\varphi} \hat{E}/\gamma \|_\ell,
  \end{equation*}
  and Leibniz rule for the derivative yields
  \begin{equation*}
    \big|\langle \widehat{\partial_t j}, \hat{\varphi} \rangle\big| \leq
    \tilde{C} \|\hat{\varphi}\|_{\ell} \|\hat{E}\|_\ell.
  \end{equation*}
  If $\hat{E} \in [\Schwartz(\R^4)]^3$, then $(1-\chi(\omega))\hat{E}(\omega,k)$
  satisfies the hypothesis of proposition~\ref{th:main-4-2} which gives
  $\varsigma(E) = \mathcal{F}^{-1}\big((1-\chi)\hat{\varsigma}_0 \hat{E}\big)$.
\end{proof}

\section{Stationary and non-stationary phase results for some integrals}
\label{sec:stationary-phase}

This section is devoted to the proof of four lemmas, in which we study the phase
$\vartheta_\eps = \tan (\pi a_0) \cosh(\pi\eps)$ with the aim of evaluating
integrals in $u$ via stationary phase results. It replaces the study of
$\frac{1}{\sin \pi a_\epsilon}$ by 
the study of $-\frac{i}{\cos \pi a_0}\int_0^{+\infty}e^{i\lambda \theta_{\epsilon}}d\lambda$
and exchanging integration in $\lambda$ and $u$ in integrals defining the
current, cf. equations~(\ref{eq:partition-of-unity}) and~(\ref{eq:Psi-delta}) in
section~\ref{sec:uniform}. 
In addition, in this section, as the distribution
$\frac{1}{\sin \pi a_\epsilon}$ acts on the functions $\tilde{P}_\eps$, we give
results on the functions $\tilde{P}_\eps$ relevant to our applications (lemma
\ref{th:condition-59} on functions satisfying (\ref{eq:polynomial-growth})
below). In this section, we let $\kappa_\species = 1$, cf. comments before
proposition~\ref{th:limit_tilde_r_eps}.

We note basic facts about the stationary phase points, that are the same as the
critical points of $a_0$, since $\tan(\pi a_0)$ behaves as $\pi (a_0-n)$. Because
of the symmetries of the function $a_0$, it is sufficient to consider the case
$\tau>0$ and $\xi_3\geq 0$.  

\begin{lemma}
  \label{th:stationary-phase}
  For $\eps>0$, let $\vartheta_\eps$ be defined in
  equation~(\ref{eq:complex-phase}) and let $\tau > 0$, $\xi \in \R^3$
  with $\xi_3 \geq 0$.   
  \begin{itemize}

  \item[(i)] For $\tau^2 - \xi_3^2 \leq 0$,
    $\vartheta_\eps (\tau,\xi,\cdot)$ has no critical points and 
    \begin{equation*}
      |\nabla_u \vartheta_\eps(\tau,\xi,u)| > (\pi\tau/2) (1+u^2)^{-1},
      \quad
      \text{for all }
      u \in \R^3.
    \end{equation*}
    
  \item[(ii)] For $\tau^2 - \xi_3^2 > 0$,
    $\vartheta_\eps(\tau,\xi,\cdot)$ has an isolated non-degenerate
    critical point at  
    \begin{equation*}
      u = u_c(\tau,\xi) = \big(0,0,\xi_3 / (\tau^2 - \xi^2_3)^{1/2}\big),
    \end{equation*}
    and at the critical point $u = u_c$ one has
    \begin{align*}
      \big|\det\vartheta_\eps'' (\tau,\xi,u_c)\big|
      &\geq (\pi\tau)^3 (1+u_c^2)^{-5/2}, \\
      |\langle \vartheta_\eps'' (\tau,\xi,u_c)^{-1}
      \nabla_u, \nabla_u\rangle \Upsilon| &\leq
      \frac{3}{\pi\tau} (1+u_c^2)^{3/2} 
      \max_{|\beta|=2}|\partial_u^\beta \Upsilon|,
    \end{align*}
    for all $\Upsilon(\tau,\xi,u)$ of class $C^2$, where
    $\vartheta_\eps''$ denotes the Hessian matrix of
    $\vartheta_\eps$ with respect to $u$, and $\langle\cdot, \cdot\rangle$
    is the Euclidean product in $\R^3$.
  \end{itemize}
\end{lemma}

\begin{proof}
  A direct computation gives
  \begin{equation*}
    \nabla_u \vartheta_\eps = \pi (1+\tan^2(\pi a_0)\big) \cosh(\pi \eps)
    \nabla_u a_0, 
  \end{equation*}
  and thus $\nabla_u \vartheta_\eps =0$ if and only if $\nabla_u a_0=0$. On the
  other hand the critical points of $a_0$ for a given $(\tau,\xi)$ are solution to 
  \begin{equation*}
    u_1 = u_2 = 0, \quad \frac{\tau u_3}{\sqrt{1+u_3^2}} = \xi_3.
  \end{equation*}
  For $\tau\not=0$, this equation is satisfied only if $\xi_3^2 / \tau^2 \leq 1$.
  In that case the solution is given by $u_c$ as claimed. The Hessian matrix
  \begin{multline*}
    \vartheta_\eps'' = 2 \pi^2 \tan(\pi a_0)
    \big(1+\tan^2(\pi a_0)\big) \cosh(\pi\eps) \nabla a_0 \otimes \nabla a_0  \\
    + \pi (1+\tan^2(\pi a_0)\big) \cosh(\pi \eps) a_0'',
  \end{multline*}
  at the critical point reduces to
  \begin{equation*}
    \vartheta_\eps'' (\tau,\xi,u_c) =
    \pi \big(1+\tan^2\big(\pi a_0(\tau,\xi,u_c)\big)\big) \cosh(\pi \eps)
    a_0'' (\tau,\xi,u_c),
  \end{equation*}
  with
  \begin{equation*}
    a_0''(\tau,\xi,u_c) = a_0(\tau,\xi,u_c)
    \begin{pmatrix}
      1   &    0     &    0    \\
      0   &    1     &    0    \\
      0   &    0     &  1-\xi_3^2/\tau^2
    \end{pmatrix}
  \end{equation*}
  being the Hessian matrix of $a_0(\tau,\xi,\cdot)$ evaluated at $u = u_c$ and
  we have accounted for the identity $a_0(\tau,\xi,u_c) = (\tau^2 - \xi_3^2)^{1/2}$.

  Proof of the inequality in (i). We write
  \begin{equation*}
    |\nabla_u a_0(\tau,\xi,u)|^2 = \frac{\tau^2}{1+u^2}
    \Big[ u_1^2 + u_2^2 + \Big(u_3 - \frac{\xi_3}{\tau} \sqrt{1+u^2} \Big)^2\Big],
  \end{equation*}
  and observe that for $\tau > 0$, $\xi_3 \geq 0$, $\tau^2 - \xi_3^2 \leq 0$ and all
  $u \in \R^3$, 
  \begin{equation*}
    \Big|u_3 - \frac{\xi_3}{\tau} \sqrt{1+u^2} \Big| \geq
    \sqrt{1+u_3^2} - |u_3|,
  \end{equation*}
  therefore $(1+u^2) |\nabla_u a_0(\tau,\xi,u)|^2 \geq \tau^2 \psi(u_3)$, with
  $\psi(u_3) = \big(\sqrt{1+u_3^2} - |u_3|\big)^2$. The function
  $u_3 \mapsto (1+u_3^2) \psi(u_3)$ is even and for $u_3 \geq 0$, it
  decreases monotonically starting from the value $\psi(0) = 1$, and
  approaching $1/4$ as $u_3 \to +\infty$. Hence $(1+u_3^2)\psi(u_3) > 1/4$ for all
  $u_3 \in \R$. This yields the claimed inequality. 
  
  Estimates in (ii). We compute
  \begin{equation*}
    \det a_0'' (\tau,\xi,u_c) = (\tau^2 - \xi_3^2)^{5/2} /\tau^2,
  \end{equation*}
  and, for every function $\Upsilon \in C^2$,
  \begin{align*}
    \big|\langle \big(\vartheta_\eps''(\tau,\xi,u_c)\big)^{-1}
    \nabla_u, \nabla_u\rangle \Upsilon \big|
    &= \frac{|\partial_{u_1}^2 \Upsilon + \partial_{u_2}^2 \Upsilon +
      \frac{\tau^2}{\tau^2 - \xi_3^2} \partial^2_{u_3} \Upsilon|}
          {\pi (1+\tan^2(\pi a_0))\cosh(\pi \eps)) (\tau^2 - \xi_3^2)^{1/2}}  \\
          &\leq \frac{3\tau^2}{\pi (\tau^2 - \xi_3^2)^{3/2}} \max_{|\alpha|=2}
    |\partial_u^\alpha \Upsilon|,
  \end{align*}
  since $\tau^2/(\tau^2-\xi_3^2) \geq 1$. At last it is enough noting that
  $\sqrt{\tau^2-\xi_3^2} = \tau/\sqrt{1 + u_c^2}$.
\end{proof}

\begin{remark}[Non-relativistic limits]
  In the non-relativistic and weakly relativistic limits, one has
  \begin{align*}
    a_0(\tau,\xi,u) &\approx a_{\mathrm{nr},0}(\tau,\xi,u)
    \coloneqq \tau - \xi_3u_3, \\
    \intertext{and} 
    a_0(\tau,\xi,u) &\approx a_{\mathrm{wr},0}(\tau,\xi,u)
    \coloneqq \tau (1 + u^2/2) - \xi_3 u_3,
  \end{align*}
  respectively. In the first case (the non-relativistic limit), there is no
  stationary phase point if $\xi_3 \not= 0$, but the phase reduces to a constant
  (in $u$) when $\xi_3=0$; particularly, when $\xi_3 = 0$, all values of
  $u \in \R^3$ are either in resonance if $\tau \in \Z$ or not in resonance if
  $\tau \not\in \Z$, that is, the set $\mathcal{R}(\tau,\xi)|_{\xi_3=0}$ is not
  a closed surface. On the other hand, in the weakly non-relativistic limit there
  is a non-degenerate stationary phase point $u_c = (0,0,\xi_3/\tau)$ for
  $\tau \not=0$. Then, $u_c \in \supp \chi_\delta(\tau,\xi,\cdot)$, where
  $\chi_\delta$ is the cut-off function introduced in
  equation~(\ref{eq:cut-off}) of section~\ref{sec:uniform}, only
  if there is an integer $n$ such that
  $\tau - \xi_3^2/(2\tau) \in [n-\delta/3,n+\delta/3]$.
  The Hessian matrix of the phase at the critical point in the weakly
  non-relativistic case amounts to
  \begin{equation*}
    \vartheta_{\mathrm{wr},\eps}''(\tau,\xi,u_c) =
    \pi \tau \big(1 + \tan^2(\pi a_{\mathrm{wr},0})\big) \cosh(\pi \eps) I
  \end{equation*}
  where $I$ is the identity matrix. In a sense, the relativistic Lorentz factor
  removes the degeneracy of the case $\xi_3 = 0$. The stationary phase argument
  developed here applies to both the relativistic and weakly relativistic cases.
\end{remark}

\begin{lemma}
  \label{th:phase-estimates}
  For any $\alpha \in \N_0^4$, $\alpha \not=0$, and any integer $n \geq 0$,
  there are polynomials $\pi_\alpha$, $\pi_{\alpha,j}$, $j = 1, \ldots,|\alpha|$,
  and $\pi_{\alpha}^{(n)}$ of one real variable such that 
  \begin{align}
    \label{eq:phase-estimate-1}
    \partial_{\tau,\xi}^\alpha \vartheta_\eps &= \pi_\alpha\big(\tan(\pi a_0)\big)
    \cosh(\pi\eps) (\nabla_{\tau,\xi}a_0)^\alpha, \\
    \label{eq:phase-estimate-2}
    e^{-i\lambda \vartheta_\eps} \partial_{\tau,\xi}^\alpha e^{i\lambda \vartheta_\eps}
    &= \sum_{j=1}^{|\alpha|} \big(i\lambda \cosh(\pi\eps)\big)^j
    \pi_{\alpha,j}\big(\tan(\pi a_0)\big) (\nabla_{\tau,\xi}  a_0)^\alpha,
    \quad \lambda \in \R,\\
    \label{eq:phase-estimate-3}
    \big|\partial_{\tau,\xi}^\alpha \big[\vartheta_\eps^n
      e^{i\lambda \vartheta_\eps} \big]\big| &\leq \pi_{\alpha}^{(n)}(\lambda)
    (1 + u^2)^{\frac{|\alpha|}{2}}, \quad
    u \in \supp\chi_\delta(\tau,\xi,\cdot), \quad \lambda \geq 0,
  \end{align}
  where $\chi_\delta$ is the cut-off function defined in equation~(\ref{eq:cut-off}).
\end{lemma}

\begin{proof}
  The key observation is that $a_0$ is a linear function of $(\tau,\xi)$, hence
  $\partial_{\tau,\xi}^\beta a_0 = 0$ for $|\beta| \geq 2$.
  Identities~(\ref{eq:phase-estimate-1}) and (\ref{eq:phase-estimate-2}) are
  true for $|\alpha| = 1$ and can be extended to all $\alpha$ with
  $|\alpha| \geq 1$ by induction.  
  As for inequality~(\ref{eq:phase-estimate-3}), Leibniz formula applied to the
  function $G_n(\lambda,z) = z^n e^{i\lambda z}$ with $\lambda\in\R$ and $z \in \C$,
  gives for any positive integer $\ell$, 
  \begin{equation*}
    \partial_z^\ell G_n(\lambda,z) = e^{i\lambda z} \sum_{m=0}^{\ell}
    \binom{\ell}{m} (i \lambda)^{\ell-m} \partial_z^{m} (z^n).
  \end{equation*}
  Because of lemma~\ref{th:on-supp-chi-delta},
  for $u \in \supp\chi_{\delta}(\tau,\xi,\cdot)$ and $\eps \leq \eps_0$,
  \begin{equation*}
    |\vartheta_\eps| \leq |\tan(\pi a_0) \cosh(\pi \eps)| + |\sinh(\pi\eps)|
    \leq 
    \cosh(\pi\eps_0) / \cos(\pi\delta/3) + \sinh(\pi\eps_0),
  \end{equation*}
  therefore, for $\lambda \geq 0$,
  \begin{equation*}
    |\partial_z^\ell G_n(\lambda,\vartheta_\eps)| \leq
    \tilde{\pi}^{(n)}_\ell (\lambda),
  \end{equation*}
  where $\tilde{\pi}^{(n)}_\ell(\lambda)$ is a polynomial of degree $\ell$ in
  $\lambda$ and with positive real coefficients. Since
  $\big|\nabla_{\tau,\xi} a_0(\tau,\xi,u)\big| \leq (1+u^2)^{1/2}$,
  lemma~\ref{th:on-supp-chi-delta} and identity (\ref{eq:phase-estimate-1}) imply
  \begin{equation*}
    |\partial_{\tau,\xi}^\alpha \vartheta_\eps(\tau,\xi,u)| \leq c_\alpha
    (1+u^2)^{|\alpha|/2},\;
    \text{ for $\eps \in [0,\eps_0]$, $|\alpha|\geq 1$
      and $u \in \supp\chi_\delta(\tau,\xi,\cdot)$.}
  \end{equation*}
  The Fa\`a di Bruno's formula for
  $\partial_{\tau,\xi}^\alpha[\vartheta_\eps^n e^{i\lambda \vartheta_\eps}] =
  \partial_{\tau,\xi}^\alpha[G_n(\lambda, \vartheta_\eps)]$ amounts to the sum of
  terms of the form
  \begin{equation*}
    \partial_z^\ell G_n(\lambda,\vartheta_\eps) \partial_{\tau,\xi}^{\alpha_1}
    \vartheta_\eps \cdots \partial_{\tau,\xi}^{\alpha_\ell} \vartheta_\eps,
  \end{equation*}
  where the multi-indices $\alpha_1,\ldots, \alpha_\ell$ in the $\ell$ factors are
  such that $\partial_{\tau,\xi}^\alpha = \partial_{\tau,\xi}^{\alpha_1} \cdots
  \partial_{\tau,\xi}^{\alpha_\ell}$, and in articular,
  $|\alpha_1| + \cdots + |\alpha_\ell| = |\alpha|$. It follows that
  \begin{equation*}
    \big|\partial_{\tau,\xi}^\alpha[\vartheta_\eps^n e^{i\lambda \vartheta_\eps}]\big|
    \leq \pi_\alpha^{(n)}(\lambda)
    (1 + u^2)^{|\alpha|/2},
  \end{equation*}
  where $\pi_\alpha^{(n)}$ is a polynomials of degree $|\alpha|$ in $\lambda$,
  with positive real coefficients.
\end{proof}

Next, we establish estimates for the integrands in equations~(\ref{eq:Psi-delta}). 
(Let us recall that $\|\phi\|_j$ denotes the Schwartz semi-norms of
$\phi \in \Schwartz$ as defined in appendix~\ref{sec:basic-def}.) 

\begin{lemma}
  \label{th:estimates-of-integrands}
  Let $\eps_0 >0$ and $\delta \in (0,1)$ be fixed and let $a_\eps$ and $\chi_\delta$ 
  be defined in equation~(\ref{eq:a-and-bs}) and (\ref{eq:cut-off}). 
  For every $\phi \in \Schwartz(\R^3)$, $\tilde{P}_\eps$ satisfying
  condition~(\ref{eq:polynomial-growth}), $\alpha \in \N_0^4$, $\rho \in \N_0^3$,
  and integer $m\geq0$, there are constants
  $C^c_{\alpha,\rho}$ and $C^s_{\alpha,\rho}$, dependent on $\eps_0$,and $\delta$ but
  independent of $m$, such that 
  \begin{align*}
    (1+u^2)^m \Big|\partial_{\tau,\xi}^\alpha\partial_u^\rho\Big[
      \frac{\chi_\delta \tilde{P}_\eps \phi}{\cos(\pi a_0)}
       \Big]\Big| &\leq C^c_{\alpha,\rho}
    (1+ \tau^2 + \xi^2)^{\ell_{\alpha,\rho} + \frac{|\rho|}{2}}
    \|\phi\|_{2(m + \tilde{\ell}_{\alpha,\rho}) + |\alpha| + |\rho|} \\
    (1+u^2)^m \Big|\partial_{\tau,\xi}^\alpha\partial_u^\rho \Big[
      \frac{(1-\chi_{\delta}) \tilde{P}_\eps \phi}{\sin(\pi a_\eps)}
      \Big]\Big| &\leq C^s_{\alpha,\rho}
    (1+ \tau^2 + \xi^2)^{\ell_{\alpha,\rho} + \frac{|\rho|}{2}}
    \|\phi\|_{2(m +\tilde{\ell}_{\alpha,\rho}) + |\alpha| + |\rho|},
  \end{align*}
  uniformly in $(\tau,\xi,u)$ and $\eps \in [0,\eps_0]$, with 
  $\ell_{\alpha,\rho} = \min \{\ell \in \N_0 :\; \ell \geq \hat{m}_\mu,\,
  \forall \mu:\; |\mu| \leq |\alpha| + |\rho|\}$ and
  $\tilde{\ell}_{\alpha,\rho} = \min \{\ell \in \N_0 :\; \ell \geq \hat{n}_\mu,\,
  \forall \mu:\; |\mu| \leq |\alpha| + |\rho|\}$.
\end{lemma} 

\begin{proof} 
  For $\alpha=0$ and $\rho=0$, both inequalities follows immediately from
  (\ref{eq:polynomial-growth}) and lemma~\ref{th:on-supp-chi-delta}. 
  For the first inequality with $|\alpha|+|\rho|\geq 1$, the Leibniz rule yields
  \begin{equation*}
    \Big|\partial_{\tau,\xi}^\alpha \partial_u^\rho \Big[
      \frac{\tilde{P}_\eps\phi}{\cos(\pi a_0)}
      \chi_\delta\Big]\Big| \leq \hat{C}_{\alpha,\rho}
    \max_{\beta,\rho'} \Big|\partial_{\tau,\xi}^{\beta}\partial_u^{\rho'}
    (\frac{\chi_\delta}{\cos(\pi a_0)})\Big|
    \max_{\beta,\rho'} \big|\partial_{\tau,\xi}^\beta \partial_u^{\rho'} \tilde{P}^\eps\big|
    \max_{\rho'} \big|\partial_u^{\rho'}\phi\big|,
  \end{equation*}
  where the $\max$ are on multi-indices $\beta \in \N_0^4$ and $\rho' \in \N_0^3$
  such that $\beta_i \leq \alpha_i$ for all $i =0,1,2,3$ and $\rho'_i \leq \rho_i$
  for all $i=1,2,3$, and with constant depending only of $\alpha,\rho$.

  The function $a_0$ is linear in $(\tau,\xi)$,
  hence $\partial_{\tau,\xi}^\beta a_0 = 0$ for all multi-indices $\beta$ with
  $|\beta|\geq 2$, while $|\partial_{\tau, \xi} a_0| \leq (1+u^2)^{1/2}$, and 
  \begin{equation*}
    |\partial_u^\rho a_0| \leq (1 + \tau^2 + \xi^2)^{1/2},\;
    |\partial_{\tau, \xi} \partial_u^\rho a_0| \leq \tilde{c}_\rho,
    \text{ for all } \rho \in \N_0^3:\, |\rho|\geq 1,
  \end{equation*}
  where $\tilde{c}_\rho = \max\{1, \sup_u |\partial_u^\rho \sqrt{1+u^2}|\}$.
  
  Because of lemma~\ref{th:partition-of-unity}, near any point
  $(\tau,\xi,u) \in \supp\chi_\delta$, there is a unique integer $n$ such that
  $\chi_\delta = \chi\big((a_0 - n)/\delta\big)$. One can estimate the derivatives
  of $\chi\big((a_0 - n)/\delta\big) / \cos(\pi a_0)$ by means of a multivariate
  version of the standard Fa\`a di Bruno's formula \cite{Gzyl1986}. However, we
  observe that each factor $\partial_{\tau,\xi}^\beta \partial_u^{\rho'} a_0$
  grows at most like either $(1+\tau^2+\xi^2)^{1/2}$ or $(1+u^2)^{1/2}$
  regardless of the order of the derivative. Therefore it is sufficient to
  estimate the term in the Fa\`a di Bruno's formula with the highest number of
  factors, that is,
  \begin{equation*}
    \partial_{\tau,\xi}^{\beta}\partial_u^{\rho'}
    \Big[\frac{\chi_\delta}{\cos(\pi a_0)} \Big]
    = (\nabla_{\tau,\xi,u} a_0)^\mu \frac{d^N}{dz^N}
    \Big[\frac{\chi\big((z-n)/\delta\big)}{\cos(\pi z)}\Big]_{z=a_0} + \cdots
  \end{equation*}
  where $\mu = (\beta,\rho') \in \N_0^7$, and $N = |\mu| = |\beta| + |\rho'|$.
  All the derivative with respect to $z$ are bounded because of the support of
  $\chi$ and its derivatives as in lemma~\ref{th:on-supp-chi-delta}.
  Hence, for all $\beta \leq \alpha$ and $\rho' \leq \rho$,
  \begin{equation*}
    \Big| \partial_{\tau,\xi}^{\beta} \partial_u^{\rho'}
    \Big[\frac{\chi_\delta}{\cos(\pi a_0)}\Big] \Big| \leq
    c_{\alpha,\rho}(\delta) (1 + u^2)^{|\alpha|/2} (1 + \tau^2 + \xi^2)^{|\rho|/2}.
  \end{equation*}
  Assumption (\ref{eq:polynomial-growth}) then gives
  \begin{multline*}
    \Big|\partial_{\tau,\xi}^\alpha \partial_u^{\rho} \Big[
      \frac{\tilde{P}_\eps\phi}{\cos(\pi a_0)}
      \chi_\delta\Big]\Big| \leq 
    C^c_{\alpha,\rho} (1+\tau^2 + \xi^2)^{\ell_{\alpha,\rho} + \frac{|\rho|}{2}}
    (1 + u^2)^{\tilde{\ell}_{\alpha,\rho}+\frac{|\alpha|}{2}}
    \max_{|\rho'| \leq |\rho|} \big|\partial_u^{\rho'} \phi\big|,
  \end{multline*}
  where $\ell_{\alpha,\rho}$ and $\tilde{\ell}_{\alpha,\rho}$ are smallest
  integers larger than all the exponents $\hat{m}_\mu$ and $\hat{n}_\mu$
  for $|\mu| \leq |\alpha| + \rho|$, respectively. Upon multiplying by
  $(1+u^2)^m$, one obtains the claimed inequality. The second inequality follows
  analogously.  
\end{proof}

The functions defined in~(\ref{eq:Psi-delta}) have the following properties.

\begin{lemma}
  \label{th:application-of-stationary-phase}
  Let $\eps_0 > 0$ be fixed, $a_\eps$ and $\chi_\delta$
  be defined in equation~(\ref{eq:a-and-bs}) and (\ref{eq:cut-off}), 
  $\ell_{\alpha,\rho}$ and $\tilde{\ell}_{\alpha,\rho}$ be the integers given in
  lemma~\ref{th:estimates-of-integrands}, 
  and $m_j = \min \{\ell \in\N_0:\ell \geq \ell_{0,\rho}, |\rho|\leq j\}$,
  $\tilde{m}_j = \min \{\ell \in\N_0: \ell \geq \tilde{\ell}_{0,\rho},\;
  |\rho|\leq j\}$. If $\tilde{P}_\eps$, $\eps\in[0,\eps_0]$, satisfies
  condition~(\ref{eq:polynomial-growth}), 
  
  \begin{itemize}
  \item[(i)] for any $\delta \in (0,1)$, $\mathcal{I}_{\delta,\eps}^c(\psi)$,
    $\mathcal{I}_{\delta,\eps}^s(\psi)$, $\mathcal{I}_{\delta,0}^c(\psi)$, and
    $\mathcal{I}_{\delta,0}^s(\psi)$ defined in~(\ref{eq:Psi-delta}) for
    $\psi \in C_0^\infty$ are defined and
    of class $C^\infty$ for any $\psi \in \Schwartz(\R^3)$;
    
  \item[(ii)] for any $\delta \in (0,1)$ and $\phi \in \Schwartz(\R^3)$,
    $\mathcal{I}_{\delta,\eps}^c(\phi) \to \mathcal{I}_{\delta,0}^c(\phi)$ and
    $\mathcal{I}_{\delta,\eps}^s(\phi) \to \mathcal{I}_{\delta,0}^s(\phi)$,
    pointwise as $\eps \to 0^+$;
    
  \item[(iii)] for any $\delta \in (0,1)$, there exists a function
    $\mathcal{B}_{\eps_0,\delta} : \R_+ \times (\R \setminus\{0\}) \to \R_+$
    such that $\mathcal{B}_{\eps_0,\delta}(\cdot,\tau) \in L^1(\R_+)$,
    $\mathcal{B}_{\eps_0,\delta}(\lambda,\tau_1) \leq
    \mathcal{B}_{\eps_0,\delta}(\lambda,\tau_2)$  for $|\tau_1| \geq |\tau_2|$, and
    \begin{equation*} 
      |\mathcal{I}_{\delta,\eps}^c(\phi)(\tau,\xi,\lambda)| \leq
      \mathcal{B}_{\eps_0,\delta}(\lambda,\tau) 
      (1+\tau^2+\xi^2)^{m_4 + \frac{5}{2}} \|\phi\|_{2\tilde{m}_4 + 10},
    \end{equation*}
    for all $\phi \in \Schwartz(\R^3)$,
    $(\tau,\xi,\lambda) \in (\R\setminus\{0\})\times\R^3\times\R_+$, and
    $\eps \in [0,\eps_0]$;
    
  \item[(iv)] for $\alpha \in \N_0^4$, $|\alpha| \geq 1$,
    $\phi \in \Schwartz(\R^3)$, and for any connected, compact set $K \subset \R^4$
    with non-empty interior, such that $\tau \not = 0$ and
    $|\tau^2 - \xi_3^2 - n^2| \geq  \delta_K >0$ for all $n \in \N_0$ and
    $(\tau,\xi) \in K$, 
    there is $\delta_0 \in (0,1)$ depending on $K$ such that
    \begin{equation*}
      |\partial_{\tau,\xi}^\alpha \mathcal{I}_{\delta,\eps}^c(\phi)
      (\tau,\xi,\lambda)| \leq \mathcal{B}_{\eps_0,\delta}^{(\alpha)}(\lambda,\tau),
      \quad (\tau,\xi) \in K, \quad \delta \in (0, \delta_0),
    \end{equation*}
    where 
    $\mathcal{B}_{\eps_0,\delta}^{(\alpha)} : \R_+ \times (\R \setminus\{0\}) \to \R_+$
    satisfies
    $\mathcal{B}_{\eps_0,\delta}^{(\alpha)}(\cdot,\tau) \in L^1(\R_+)$, and
    decreases in $|\tau|$, that is,
    $\mathcal{B}_{\eps_0,\delta}^{(\alpha)} (\lambda,\tau_1) \leq
    \mathcal{B}_{\eps_0,\delta}^{(\alpha)}(\lambda,\tau_2)$  for $|\tau_1| \geq |\tau_2|$; 
 
  \item[(v)] for any multi-index $\alpha \in \N^4_0$, 
    there is $K^s_\alpha>0$ dependent on $\eps_0$ and $\delta$ such that 
    \begin{equation*}
      |\partial_{\tau,\xi}^\alpha \mathcal{I}_{\delta,\eps}^s(\phi)(\tau,\xi)| \leq
      K^s_\alpha (1+\tau^2+\xi^2)^{\ell_{\alpha,0}}
      \|\phi\|_{2\tilde{\ell}_{\alpha,0} + |\alpha| + 4},
    \end{equation*}
    uniformly for $\eps \in [0, \eps_0]$, for all $\phi \in \Schwartz(\R^3)$.
  \end{itemize}
\end{lemma}

\begin{proof}
  (i) We shall show that, for any $\phi \in \Schwartz(\R^3)$, the integrands
  in equations~(\ref{eq:Psi-delta-sing})-~(\ref{eq:psi-reg})
  and their derivatives with respect to $(\tau,\xi,\lambda)$ are uniformly
  bounded by an integrable function of $u \in \R^3$ for $(\tau,\xi)$ and
  $\lambda$ in a bounded set and for $\eps \in [0,\eps_0]$. Then it follows that
  all integrals in equations~(\ref{eq:Psi-delta}) are finite also when
  $\psi \in C_0^\infty(\R^3)$ is replaced by $\phi \in \Schwartz(\R^3)$, and the
  dominated convergence theorem allows us to differentiate in the integral.
  
  For the case of $\mathcal{I}_{\delta,\eps}^s$ and $\mathcal{I}_{\delta,0}^s$,
  the needed uniform upper bound follows directly from the second inequality of
  lemma~\ref{th:estimates-of-integrands} with $m > 3/2$ and $\rho = 0$.
  
  As for $\mathcal{I}^c_{\delta,\eps}$, if $\phi \in \Schwartz(\R^3)$,
  \begin{align*}
    \partial_{\tau,\xi}^\alpha \partial_\lambda^n
    \Big[e^{i\lambda \vartheta_\eps}
      \frac{\chi_\delta\tilde{P}_\eps \phi}{\cos(\pi a_0)}\Big]
    &= \partial_{\tau,\xi}^\alpha 
    \Big[\big(i\vartheta_\eps\big)^n e^{i\lambda \vartheta_\eps}
      \frac{\chi_\delta\tilde{P}_\eps \phi}{\cos(\pi a_0)}\Big] \\
    &= i^n \sum_{\beta \leq \alpha} \binom{\alpha}{\beta}
    \partial_{\tau,\xi}^{\alpha-\beta} \big[\vartheta_\eps^n
      e^{i\lambda\vartheta_\eps} \big] \partial_{\tau,\xi}^\beta\Big[
      \frac{\chi_\delta \tilde{P}_\eps \phi}{\cos(\pi a_0)} \Big].
  \end{align*}
  For any $R > 0$,
  inequality~(\ref{eq:phase-estimate-3}) in lemma~\ref{th:phase-estimates} gives
  \begin{equation*}
    \big|\partial_{\tau,\xi}^{\alpha-\beta} \big[\vartheta_\eps^n
      e^{i\lambda \vartheta_\eps} \big]\big| \leq M_{R,\alpha,n}
    (1 + u^2)^{\frac{|\alpha-\beta|}{2}},
  \end{equation*}
  for $u \in \supp \chi_\delta\big((\tau,\xi,\cdot)\big)$ and $|\lambda| \leq R$.
  Then the first inequality proven in lemma~\ref{th:estimates-of-integrands} with
  $\rho=0$ gives
  \begin{equation}
    \label{eq:bound}
    \Big| \partial_{\tau,\xi}^\alpha \partial_\lambda^n
    \Big[e^{i\lambda \vartheta_\eps}
      \frac{\chi_\delta\tilde{P}_\eps \phi}{\cos(\pi a_0)}\Big]
    \Big| \leq
    \frac{\widetilde{M}_{R,\alpha,n}}{(1+u^2)^m}
    \|\phi\|_{2(m + \tilde{\ell}_\alpha) + |\alpha|}
  \end{equation}
  uniformly for $\eps \in [0,\eps_0]$, $(\tau,\xi,\lambda)$ in the ball
  $\tau^2 + \xi^2 + \lambda^2 \leq R^2$ for all $R$, and with
  $\tilde{\ell}_{\alpha} = \max_{\beta\leq\alpha} \tilde{\ell}_{\beta,0}$.
  For $m>3/2$ the right-hand side is integrable over $\R^3$.

  (ii) The integrands in
  equations~(\ref{eq:Psi-delta-sing})-~(\ref{eq:psi-delta-reg}) are continuous
  for $\eps \to 0^+$ and the inequalities proven in
  lemma~\ref{th:estimates-of-integrands} with $m > 3/2$, $\alpha = 0$,
  and $\rho = 0$ imply upper bounds by $L^1$ functions uniformly in
  $\eps \in [0,\eps_0]$. Then, the hypothesis of the dominated convergence
  theorem are satisfied and one can pass to the limit in the integral.  

  (iii) Let us first address the case $\mathcal{I}^c_{\delta,\eps}(\psi)$ for
  $\psi \in C_0^\infty(\R^3)$. Since 
  $\mathcal{I}^c_{\delta,\eps}(\psi)(\tau, \xi,\cdot)$ is continuous, it is
  measurable and integrable on compact intervals. We write
  \begin{equation*}
    \int_0^{+\infty} \mathcal{I}_{\delta,\eps}^c(\psi)(\tau,\xi,\lambda) d\lambda =
    \int_0^1 \mathcal{I}_{\delta,\eps}^c(\psi)(\tau,\xi,\lambda) d\lambda +
    \int_1^{+\infty} \mathcal{I}_{\delta,\eps}^c(\psi)(\tau,\xi,\lambda) d\lambda.
  \end{equation*}
  The first integral on the right-hand side is bounded by
  \begin{equation*}
    \int_{\R^3} \Big| \frac{\tilde{P}_\eps \psi}{\cos(\pi a_0)}
    \chi_\delta \Big| du \leq
    C_{0,0}^c (1+\tau^2+\xi^2)^{\ell_{0,0}}
    \|\psi\|_{2(m+\tilde{\ell}_{0,0})} \int_{\R^3} \frac{du}{(1+u^2)^m},
  \end{equation*}
  in view of the first inequality of lemma~\ref{th:estimates-of-integrands}
  with $\alpha = 0$, $\rho=0$, and $m > 3/2$.

  For the second integral, we estimate the decay in $\lambda$ of
  $\mathcal{I}^c_{\delta,\eps}(\psi)(\tau,\xi,\cdot)$ by means of the stationary
  phase formula.  
  
  For $\tau \not= 0$ and $\tau^2 - \xi_3^2 \leq 0$, lemma~\ref{th:stationary-phase}
  shows that there are no stationary phase points, and we have a lower bound for
  the gradient of the phase. The standard stationary phase lemma
  \cite[Theorem 7.7.1]{Hormander1} gives
  \begin{equation*}
    \Big|\int_{\R^3} e^{i\lambda \vartheta_\eps}
    \frac{\chi_\delta \tilde{P}_\eps \psi}{\cos(\pi a_0)} du \Big| \leq
    \frac{c_{\mathrm{sp},\ell}}{\lambda^\ell} \sum_{|\rho|\leq \ell}
    \sup \Big(\frac{1}{|\nabla_u \vartheta_\eps|^{2\ell - |\rho|}} \Big|
    \partial_u^\rho \Big[
      \frac{\chi_\delta \tilde{P}_\eps \psi}{\cos(\pi a_0)} \big]\Big|\Big),
  \end{equation*}
  for any integer $\ell \geq 0$. The right-hand side is integrable in
  $\lambda \in [1,+\infty)$ if $\ell \geq 2$. We choose $\ell=2$ and 
  lemma~\ref{th:stationary-phase} (i) gives, for $|\rho|\leq \ell = 2$,
  \begin{equation*}
    \frac{1}{|\nabla_u \vartheta_\eps|^{2\ell - |\rho|}} \leq
    \frac{(1+u^2)^{4-|\rho|}}{(\pi \tau/2)^{4-|\rho|}} \leq
      \frac{16}{\pi^2 \tau^{4-|\rho|}} (1+u^2)^4,
  \end{equation*}
  and lemma~\ref{th:estimates-of-integrands} with $m=4$ and $\alpha = 0$
  implies that, for $|\rho| \leq \ell = 2$,
  \begin{equation*}
    \frac{1}{|\nabla_u \vartheta_\eps|^{2\ell-|\rho|}} \Big|
    \partial_u^\rho \Big[
      \frac{\chi_\delta \tilde{P}_\eps \psi}{\cos(\pi a_0)} \big]\Big| \leq
    \frac{16 C^c_{0,\rho} }{\pi^2 \tau^{4-|\rho|}}
    (1+\tau^2+\xi^2)^{\ell_{0,\rho}+|\rho|/2} \|\psi\|_{2\tilde{\ell}_{0,\rho} + |\rho| + 8},
  \end{equation*}
  which in turns yields
  \begin{equation*}
    \Big|\int_{\R^3} e^{i\lambda \vartheta_\eps}
    \frac{\chi_\delta \tilde{P}_\eps \psi}{\cos(\pi a_0)} du \Big| \leq
    \mathcal{B}'_{\eps_0,\delta}(\lambda,\tau)
    (1+\tau^2+\xi^2)^{m_2 + 1} \|\psi\|_{2\tilde{m}_2 +10}.
  \end{equation*}
  where
  \begin{equation*}
    \mathcal{B}'_{\eps_0,\delta}(\lambda,\tau) =
    \frac{16 c_{\mathrm{sp,2}}}{\pi^2 \lambda^2} \sum_{|\rho|\leq 2}
    \frac{C_{0,\rho}^c}{\tau^{4-|\rho|}}, \quad \lambda \geq 1.
  \end{equation*}
  Since by definition $m_2 \leq m_4$ this proves the claim for
  $\tau^2 - \xi_3^2 \leq 0$.
  
  For $\tau^2 - \xi_3^2 >0$, there is an isolated non-degenerate stationary phase
  point $u = u_c(\tau,\xi)$ as shown in lemma~\ref{th:stationary-phase}.  
  In this case, the stationary phase lemma \cite[Theorem 7.7.5]{Hormander1} gives
  \begin{multline*}
    \Big| \int_{\R^3} e^{i\lambda \vartheta_\eps}
    \frac{\chi_\delta \tilde{P}_\eps \psi}{\cos(\pi a_0)} du \Big| \leq
    \big|\det(\lambda \vartheta_\eps''/2\pi)\big|^{-\frac{1}{2}}
    \sum_{j=0}^{\ell-1} \lambda^{-j} |L_{j,\tau,\xi}(\psi)| \\
    + \frac{c'_{\mathrm{sp},\ell}}{\lambda^\ell} \sum_{|\rho|\leq 2\ell}
    \sup \Big| \partial_u^\rho \Big[
      \frac{\chi_\delta \tilde{P}_\eps \psi}{\cos(\pi a_0)} \big]\Big|,
  \end{multline*}
  where the Hessian matrix $\vartheta_\eps''$ and $L_{j,\tau,\xi}(\psi)$
  are differential operators acting on $\psi$ and evaluated at the stationary
  phase point $u=u_c(\tau,\xi)$. We seek an upper bound in $L^1$, hence it is
  sufficient to choose $\ell = 2$, and we have 
  \begin{align*}
    L_{0,\tau,\xi}(\psi) &= \frac{\chi_\delta\tilde{P}_\eps \psi}{\cos(\pi a_0)},
    \\
    L_{1,\tau,\xi}(\psi) &= -\frac{i}{2}
    \langle (\vartheta_\eps'')^{-1} \nabla_u, \nabla_u\rangle
    \Big[\frac{\chi_\delta \tilde{P}_\eps \psi}{\cos(\pi a_0)}\Big],
  \end{align*}
  evaluated at $(\tau,\xi,u) = (\tau,\xi,u_c)$. Using the results of
  lemma~\ref{th:stationary-phase} (ii) together with the first inequality of
  lemma~\ref{th:estimates-of-integrands} yields
  \begin{align*}
    \frac{|L_{0,\tau,\xi}(\psi)|}{|\det(\vartheta_\eps'')|^{1/2}} &\leq
    \frac{1}{(\pi \tau)^{3/2}} (1+u_c^2)^{5/4}
    \Big|\frac{\chi_\delta\tilde{P}_\eps \psi}{\cos(\pi a_0)}\Big| \\
    &\leq \frac{1}{(\pi \tau)^{3/2}}
    C_0 (1+\tau^2+\xi^2)^{m_0} \|\psi\|_{2\tilde{m}_0+4}, \\
    \frac{|L_{1,\tau,\xi}(\psi)|}{|\det(\vartheta_\eps'')|^{1/2}} &\leq
    \frac{3/2}{(\pi\tau)^{5/2}} 
    \max_{|\beta|=2} \Big|(1+u_c^2)^{11/4} \partial_u^\beta \Big[
      \frac{\chi_\delta\tilde{P}_\eps \psi}{\cos(\pi a_0)}\big]\Big| \\
    &\leq \frac{3/2}{(\pi\tau)^{5/2}}
    C_2 (1+\tau^2+\xi^2)^{m_2 + 1} \|\psi\|_{2\tilde{m}_2 + 8}, \\
    \intertext{and}
    \Big| \partial_u^\rho \Big[
      \frac{\chi_\delta \tilde{P}_\eps \psi}{\cos(\pi a_0)} \big]\Big| &\leq
    C_4 (1+\tau^2+\xi^2)^{m_4 + 2} \|\psi\|_{2\tilde{m}_4 + 4},
  \end{align*}
  where $C_j = \max_{|\rho| \leq j} C^c_{0,\rho}$, $j \in \N_0$. Therefore,
  \begin{align*}
    \Big| \int_{\R^3} e^{i\lambda \vartheta_\eps}
    \frac{\chi_\delta \tilde{P}^\eps \psi}{\cos(\pi a_0)} du \Big|
    &\leq \mathcal{B}_{\eps_0,\delta}''(\lambda,\tau)
    (1+\tau^2+\xi^2)^{m_4 + \frac{5}{2}} \|\psi\|_{2\tilde{m}_4 + 8},
  \end{align*}
  with
  \begin{equation*}
    \mathcal{B}_{\eps_0,\delta}''(\lambda,\tau) =
    \Big(\frac{2\pi}{\lambda}\Big)^{3/2} \Big[
    \frac{C_0}{(\pi\tau)^{3/2}} +
    \frac{3}{2} \frac{C_2}{(\pi\tau)^{5/2}}\frac{1}{\lambda} \Big] +
    \frac{C_4'}{\lambda^2}, \quad \lambda \geq 1,
  \end{equation*}
  with the constant $C_4'$ depending on $\delta$ and $\eps_0$.
  At last we can combine the estimates obtained in the three cases by defining
  $\mathcal{B}_{\eps_0,\delta} =
  \mathcal{B}_{\eps_0,\delta}'+\mathcal{B}_{\eps_0,\delta}''$ for  
  $\lambda \geq 1$ and extending it to a constant for $\lambda \in [0,1]$.
  This gives the claimed estimate for $\psi \in C_0^\infty(\R^3)$. Then we
  observe that for any $(\tau,\xi,\lambda)$ with $\tau \not=0$, the map
  $\psi \mapsto \mathcal{I}_{\delta,\eps}^c(\psi)(\tau,\xi,\lambda)$ defines a
  linear functional on $C_0^\infty(\R^3)$, bounded by a Schwartz
  semi-norm. Since $C_0^{\infty}(\R^3)$ is dense in $\Schwartz(\R^3)$ the
  inequality remains true for $\psi \in \Schwartz(\R^3)$: given a sequence
  $\psi_i \in C_0^{\infty}$, $i \in \N$, converging to
  $\phi \in \Schwartz(\R^3)$ in the topology of $\Schwartz$, 
  inequality~(\ref{eq:bound}) with $\alpha=0$ and $n=0$ shows that
  $|\mathcal{I}_{\delta,\eps}^c(\psi_i)| \to |\mathcal{I}_{\delta,\eps}^c(\phi)|$, 
  while by definition of convergence we have $\|\psi_i\|_j \to \|\phi\|_j$;
  then we can pass to the limit $i \to +\infty$ on both sides of the inequality.
  
  (iv) First, let $\psi \in C_0^\infty(\R^3)$. From (i), we have that
  \begin{equation*}
    \partial_{\tau,\xi}^\alpha \mathcal{I}_{\delta,\eps}^c(\psi)(\tau,\xi,\lambda)
    = i^n \sum_{\beta \leq \alpha} \binom{\alpha}{\beta} \int_{\R^3}
    \partial_{\tau,\xi}^{\alpha-\beta} \big[e^{i\lambda\vartheta_\eps} \big]
    \partial_{\tau,\xi}^\beta\Big[
      \frac{\chi_\delta \tilde{P}_\eps \psi}{\cos(\pi a_0)} \Big] du.
  \end{equation*}
  We exploit the second identity in lemma~\ref{th:phase-estimates} for the
  factor $\partial_{\tau,\xi}^{\alpha-\beta} e^{i\lambda \vartheta_\eps}$ with
  the result that
  \begin{equation*}
    \partial_{\tau,\xi}^\alpha \mathcal{I}_{\delta,\eps}^c(\psi)(\tau,\xi,\lambda) =
    \sum_{j=0}^{|\alpha|} \lambda^j \cosh(\pi\eps)^j
    \int_{\R^3} e^{i\lambda \vartheta_\eps} \psi_{\delta,\eps,j} (\tau,\xi,u) du,
  \end{equation*}
  where $\psi_{\delta,\eps,j}$ are combinations of polynomials in $\tan(\pi a_0)$
  and $(\tau,\xi)$-derivatives of the function
  $\chi_\delta \tilde{P}_\eps \psi / \cos(\pi a_0)$. The assumption states that
  $(\tau,\xi)$ varies in a connected, compact set $K$ with non-empty interior,
  therefore the continuous function $(\tau,\xi) \mapsto \tau^2 - \xi_3^2$ maps
  $K$ into a closed interval $[c_1,c_2] \subset \R$. The assumption with $n=0$
  also implies that $|\tau^2-\xi_3^2|\geq \delta_K$, hence zero is not in the
  interval, that is, either $c_1 < c_2 < 0$, 
  or $0 < c_1 < c_2$. In the first case, $\tau^2 - \xi_3^3 < 0$ and
  lemma~\ref{th:stationary-phase} established that the phase $\vartheta_\eps$
  has no critical points. In the second case, let $n_0 \in \N_0$ be the unique
  non-negative integer for which
  \begin{equation*}
    n_0^2 + \delta_K \leq \tau^2 - \xi_3^2 \leq
    (n_0+1)^2 - \delta_K, \quad (\tau,\xi) \in K.
  \end{equation*}
  We have shown in the proof of lemma~\ref{th:stationary-phase} that
  $a_0(\tau,\xi,u_c) = \sqrt{\tau^2-\xi_3^2}$ is the value of $a_0$ at the
  stationary phase point $u_c(\tau,\xi)$. If $n_c \in \{n_0,n_0+1\}$ is the
  closest integer to $a_0(\tau,\xi,u_c)$, we find
  $\big|a_0(\tau,\xi,u_c) - n_c\big| \geq c(\delta_K)$, and for any $n \in \Z$,
  \begin{equation*}
    \big|a_0(\tau,\xi,u_c) - n\big| \geq
    \big|a_0(\tau,\xi,u_c) - n_c\big| \geq c(\delta_K).
  \end{equation*}
  We can now choose $\delta$ sufficiently small that $\delta/3 < c(\delta_K)$,
  which depends only on the set $K$, and we obtain $\chi_\delta(\tau,\xi,u_c) = 0$:
  there are no critical points of the phase in the support of $\chi_\delta$ for
  $(\tau,\xi) \in K$. 
  %%  
  %% Since by assumption
  %% $\tau^2 - \xi_3^2 \not= n^2$ for all integers $n \geq 0$, while at a
  %% critical point $a_0 = (\tau^2 - \xi_3^2)^{1/2}$, we can reduce $\delta$ so
  %% that the stationary phase point has a finite distance from
  %% $\supp\chi_\delta$.
  %%
  We can then apply the same argument as in (iii) in order to show that each of
  the integrals on the right-hand side decreases like $1/\lambda^\ell$ for all
  positive integers $\ell$ and since $(\tau,\xi)$ varies in a compact set $K$ we
  obtain the claimed inequality with functions
  $\mathcal{B}^{(\alpha)}_{\eps_0,\delta}$ depending on
  $\sup \{(1 + \tau^2 + \xi^2) :\; (\tau,\xi) \in K\}$ and the Schwartz
  semi-norms of $\psi$. Then the inequality can be extended to
  $\phi \in \Schwartz(\R^3)$ by choosing a sequence $\psi_n \in C_0^{\infty} (\R^3)$
  converging to $\phi$ in $\Schwartz(\R^3)$.
  
  (v). The derivatives of $\mathcal{I}_{\delta,\eps}^s(\phi)$ for
  $\phi \in \Schwartz(\R^3)$ are estimated by the second 
  inequality in lemma~\ref{th:estimates-of-integrands} with $\rho=0$, and $m=2$.  
\end{proof}

We conclude this section with the proof that the functions defined in
equation~(\ref{eq:Phi-r-eps}) are of the form (\ref{eq:r-tilde}).
The function in the following lemma are defined in
section~\ref{sec:notation-magnetized}. 

\begin{lemma}
  \label{th:condition-59}
  The functions $e^{-i\pi a_\eps} \Phi_j P^\pm_{\eps}$ and
  $e^{-i\pi a_\eps} \Psi_j P^\pm_{\eps}$ satisfy condition~(\ref{eq:polynomial-growth}).
\end{lemma}

\begin{proof}
  The considered functions are all defined as integrals in $\lambda$ over a
  compact interval and with smooth integrands, hence they are of class $C^\infty$.
  As for the polynomial-growth estimate, the derivatives of any
  \begin{equation*}
    F_\eps \in \{e^{-i\pi a_\eps} \Phi_j P^\pm_{\eps}:\, j=1,2,3 \} \cup
    \{e^{-i\pi a_\eps} \Psi_j P^\pm_{\eps}:\, j=1,2,3\}
  \end{equation*}
  are given by 
  \begin{equation*}
    \partial^\mu F_\eps(\tau,\xi,u) =
    \int_{0}^{2\pi} \Pi_\mu(\tau,\xi,u,\lambda)
    e^{i(\lambda - \pi) a_\eps - i\tilde{\zeta}} d\lambda,
  \end{equation*}
  where $\tilde{\zeta}(\tau,\xi,u,\lambda) = (\xi_1 u_1 + \xi_2 u_2)\sin\lambda
  \mp (\xi_2 u_1 - \xi_1 u_2)(\cos\lambda -1)$ and $\Pi_\mu$ is a linear
  combination of monomials of the form
  \begin{equation*}
    \gamma^m \tau^{\ell_0} \xi^\alpha u^\beta \lambda^{\ell_1} (\cos\lambda)^{\ell_2}
    (\sin\lambda)^{\ell_3}, 
  \end{equation*}
  for $m \in \Z$, $\alpha,\beta \in \N_0^3$, $\ell_j \in \N_0$. This can be
  checked directly for $\mu=0$ and $|\mu| = 1$, then extended to all
  $\mu \in \N_0^7$ by induction. We also have that $\Pi_\mu$ is independent of
  $\eps$. For $\lambda \in [0,2\pi]$ we readily have
  \begin{equation*}
    |\gamma^m \tau^{\ell_0} \xi^\alpha u^\beta \lambda^{\ell_1} (\cos\lambda)^{\ell_2}
    (\sin\lambda)^{\ell_3}| \leq (2\pi)^{\ell_1} (1+\tau^2+\xi^2)^{(\ell_0 + |\alpha|)/2}
    (1+u^2)^{(m+|\beta|)/2}, 
  \end{equation*}
  and (since we have chosen $\kappa = 1$)
  \begin{equation*}
    |e^{i(\lambda - \pi) a_\eps - i\tilde{\zeta}}| \leq e^{\pi \eps}, \quad
    \text{ for } \lambda \in [0,2\pi],
  \end{equation*}
  hence
  \begin{equation*}
    |\partial^\mu F_\eps(\tau,\xi,u)| \leq C_{F,\mu} e^{\pi\eps}
    (1+\tau^2+\xi^2)^{\hat{m}_{F,\mu}} (1+u^2)^{\hat{n}_{F,\mu}},
  \end{equation*}
  with constant $C_{F,\mu}$ and exponents $\hat{m}_{F,\mu}, \hat{n}_{F,\mu}$
  depending on the specific function $F$. This is the claimed inequality if
  $\eps \in [0,\eps_0]$. By means of the same identity one can also check
  continuity of all derivatives for $\eps \to 0$, since
  \begin{align*}
    \big|\partial^\mu F_\eps(\tau,\xi,u) - \partial^\mu F_0(\tau,\xi,u)\big|
    &\leq \int_0^{2\pi} \big|\Pi_\mu(\tau,\xi,u,\lambda)\big| \cdot
    \big|e^{-(\lambda-\pi)\eps} - 1\big| d\lambda \\
    &\leq \tilde{C}_{F,\mu}(\tau,\xi,u) |e^{\pi\eps} - 1|,
  \end{align*}
  and thus, as $\eps \to 0^+$, 
  $\partial^\mu F_\eps(\tau,\xi,u) \to \partial^\mu F_0(\tau,\xi,u)$ pointwise.
\end{proof}

\begin{remark}
  This induction argument has the advantage of avoiding lengthy calculations, but
  does not allows us to compute the exponents $\hat{m}_{F,\mu}$ and $\hat{n}_{F,\mu}$
  explicitly.
\end{remark}

% Appendices
% =========================================================================
%
\appendix

\section{Notation and basic definitions}
\label{sec:basic-def}

Positive integers are denoted by $\N = \{1,2,\ldots\}$ and natural numbers by
$\N_0 = \N \cup \{0\}$ including zero. Integers including zero, reals and
complex numbers  are denoted by $\Z$, $\R$, and $\C$, respectively. Strictly
positive real numbers are denoted by $\R_+ = \{x \in \R :\; x>0\}$.
We use the standard multi-index notation, a multi-index of dimension $n \in \N$
being an element $\alpha = (\alpha_1, \ldots, \alpha_n) \in \N_0^n$.
The length of a multi-index is $|\alpha| = \sum_i \alpha_i$, and
for $x = (x_1, \ldots,x_n)\in\R^n$, we write
$x^\alpha = x_1^{\alpha_1} \cdots x_n^{\alpha_n}$ and
$\partial_x^\alpha = \partial_{x_1}^{\alpha_1} \cdots \partial_{x_n}^{\alpha_n}$.

As usual $C^k(\R^n)$ denotes the space of $k$-times continuously differentiable
functions of $n$ variables, $C^\infty(\R^n) = \bigcap_k C^k(\R^n)$.

\begin{definition}
  We denote by  $C^\infty_b(\R^n)$ the space of functions
  $u \in C^\infty(\R^n)$ that are bounded with all derivatives bounded, that is,
  for every $\alpha \in \N_0^n$ there are real constants
  $C_\alpha > 0$ such that $|\partial_x^\alpha u(x)| \leq C_\alpha$ uniformly.
\end{definition}

We work with tempered distributions on $\R^n$. As usually we denote by
$\Schwartz(\R^n)$ the Schwartz space of rapidly decreasing functions on
$\R^n$. Those are functions $\varphi \in C^\infty(\R^n)$ for which 
\begin{equation*}
  \sup_{x\in\R^n} \big| x^\alpha \partial_x^\beta \varphi(x)\big| < \infty.
\end{equation*}
Semi-norms in $\Schwartz$ are defined by
\begin{equation*}
  \|\varphi\|_j = \max_{|\alpha| + |\beta| \leq j} \sup_{x\in\R^n}
  \big| x^\alpha \partial_x^\beta \varphi(x)\big|, \quad
  j \in \N_0.
\end{equation*}
(With the condition $|\alpha| + |\beta| \leq j$, instead of equality, these are
actually norms.) The countable set of semi-norms gives the Schwartz space the
topology of a Fr\'echet space and its topological dual $\Schwartz^\prime (\R^n)$
is the space of tempered distributions, namely, the space of continuous linear
functionals $u : \Schwartz(\R^n) \to \C$; the action of $u \in \Schwartz^\prime$
on a test-function  $\varphi \in \Schwartz$ is equivalently denoted by
$\langle u, \varphi \rangle = u(\varphi)$. In this case continuity of a linear
functional $u$ means that there exists an integer $j \geq 0$ and a constant $C_j > 0$
such that $|\langle u, \varphi \rangle| = |u(\varphi)| \leq C_j \|\varphi\|_j$. 

We also note that for any $\varphi \in \Schwartz$, for every $m \in \R$, and
$\beta \in \N_0^d$, there is a constant $C_{m,\beta} > 0$ such that 
\begin{equation}
  \label{eq:polynomial-bound}
  (1+x^2)^m \big|\partial_x^\beta \varphi(x)\big| \leq C_{m,\beta},
\end{equation}
uniformly in $x$. For $m \leq 0$, this inequality follows from
$(1+x^2)^m \leq 1$, while for $m > 0$, one can observe that
$(1+x^2)^m \leq (1 + x^2)^l$ for any integer $l \geq m$, and use the
multi-nomial formula.

The Fourier transform of a function $\varphi \in \Schwartz(\R^n)$ is defined by
\begin{equation*}
  \hat{\varphi}(\xi) = \int_{\R^n} \varphi(x) e^{-i\xi \cdot x} dx.
\end{equation*}
The Fourier transform $\mathcal{F} : \varphi \mapsto \hat{\varphi}$ is
continuous from $\Schwartz$ into itself and extends to $\Schwartz^\prime$ by
duality. Specifically, this means
\begin{equation*}
  \hat{u}(\varphi) = u(\hat{\varphi}).
\end{equation*}
That $\hat{u}$ is a continuous linear functional follows from the continuity of
$u$ and of the Fourier transform $\varphi \to \hat{\varphi}$. For a function 
$\varphi(t,x)$ in $\Schwartz(\R^{1+d})$ of physical time and space
($d=3$), we write 
\begin{equation*}
  \hat{\varphi} (\omega, k) = \int_{\R^{1+d}} \varphi(t,x) 
  e^{i\omega t - i k \cdot x} dt dx,
\end{equation*}
where, $\omega \in \R$ is the angular frequency and $k \in \R^d$ is the wave vector. 

\begin{definition}[Fourier multipliers]
  \label{def:FourierMultipliers}
  Let $a \in C(\R^n)$ and $m, C_0 \in \R$, be such that
  $|a(\xi)| \leq C_0 (1+|\xi|)^m$. The Fourier multiplier with symbol $a$ is the
  continuous linear operator on $\Schwartz(\R^n)$ defined by
  \begin{equation*}
    A\varphi(x) = \mathcal{F}^{-1} \big(a \mathcal{F}(\varphi)\big)(x)
    = \frac{1}{(2\pi)^n} \int_{\R^n} e^{ix \cdot \xi} a(\xi) \hat{\varphi}(\xi)
    d\xi,
  \end{equation*}
  where the inverse Fourier transform is an absolutely convergent integral, and
  the integrand has partial derivatives in $x$ of order $k$ bounded by the
  $L^1$-function $\xi \mapsto|\xi|^k |a(\xi)| |\hat{\varphi}(\xi)|$; hence,
  $A: \Schwartz(\R^n) \to C^\infty_b(\R^n) \subset \Schwartz^\prime(\R^n)$.
\end{definition}

If in addition, $a \in C^\infty(\R^n)$ with all derivatives bounded by
\begin{equation*}
  |\partial_\xi^\alpha a(\xi)| \leq C_\alpha (1 + |\xi|)^{m_\alpha},
\end{equation*}
for all $\xi \in \R^n$ and $\alpha \in \N^n_0$, with constants
$C_\alpha, m_\alpha \in \R$ depending only on $\alpha$,
then for any $\varphi \in \Schwartz(\R^n)$, $a \hat{\varphi} \in \Schwartz(\R^n)$
and the inverse Fourier transform gives 
$\mathcal{F}^{-1}(a \hat{\varphi}) \in \Schwartz(\R^n)$. Such operations are
continuous on $\Schwartz$. Hence, the corresponding Fourier multiplier defined
by $a$ is a continuous linear operator form $\Schwartz(\R^n) \to \Schwartz(\R^n)$.

Fourier multipliers are relevant to the case of a uniform plasma
equilibrium. Specifically, the high-frequency component of the current density
induced in a uniform plasma by an electromagnetic disturbance is related to the
electric field of the disturbance by a Fourier multiplier.

\section{Proofs for the case study of section~\ref{sec:simple-case}}
\label{sec:proofs-toy-model}

\begin{proof}[Proof of proposition~\ref{th:simple-model-causal-solution}]
  With an initial condition $u_0$ in $C_b^\infty$, the solution of this problem is
  \begin{equation*}
    u(t,x) = u_0(x) + \int_0^t v(s,x)ds,
  \end{equation*}
  and we have $u \in C^\infty (\R^{1+d})$. In addition, $u \in C_b^\infty (\R^{1+d})$,
  thanks to $v \in \mathcal{S}$. This follows upon considering the derivatives 
  $\partial_t^\ell \partial_x^\alpha u(t,x)$. For $\ell > 1$ one has
  $\partial_t^\ell \partial_x^\alpha u(t,x) = \partial_t^{\ell-1} \partial_x^\alpha v(t,x)$
  with $v \in \mathcal{S}(\R^{1+d})$, while for $\ell = 0$ and for every $m>1/2$
  we have  
  \begin{equation*}
    |\partial_x^\alpha u(t,x)| \leq |\partial_x^\alpha u_0(x)| + 
    \sup_{s\in\R} \big| (1+s^2)^m \partial_x^\alpha v(s,x) \big|
    \int_{-\infty}^{+\infty} \frac{ds}{(1+s^2)^m},
  \end{equation*}
  with $\partial_x^\alpha u_0$ bounded by hypothesis. 
  
  The initial condition 
  \begin{equation*}
    u_0 (x) = \int_{-\infty}^0 v(s,x)ds,
  \end{equation*}
  belongs to $\mathcal{S}(\R^d)$ and thus to
  $C^\infty_b(\R^d)$. The corresponding unique solution in $C_b^{\infty}(\R^{1+d})$ is
  \begin{equation*}
    u(t,x) =  (\int_{-\infty}^0 + \int_0^t) v(s,x)ds 
    =\int_{-\infty}^t v(s,x)ds.
  \end{equation*}    
  For every integers $\ell \geq 0$ and $t \in \R$, 
  $\partial_t^\ell u (t, \cdot) \in \mathcal{S}(\R^d)$ and $u$ satisfies the
  condition $\lim_{t\to-\infty} u(t,x) = 0$. As for the uniqueness, if 
  $u_* \in C^\infty_b (\R^d)$ is another initial condition such that the limit
  for $t \to -\infty$ of the corresponding solution vanishes, then 
  \begin{equation*}
    0 = u_*(x) + \int_0^{-\infty} v(s,x)ds = u_*(x) - u_0(x),
  \end{equation*}
  which shows that $u_* = u_0$. Since, in particular, 
  $u \in  L^\infty(\R^{1+d})$, it defines a tempered distribution by integration,
  \begin{equation*}
    \langle u, \varphi \rangle = \int_{\R^{1+d}} u(t,x) \varphi(t,x)dtdx,
    \quad \forall \varphi \in \Schwartz(\R^{1+d}).
  \end{equation*} 
  The continuity of the map $\varphi \mapsto  \langle u, \varphi \rangle$
  follows for $m > (1+d)/2$ from 
  \begin{equation*}
    |\langle u, \varphi \rangle | \leq
    \| u\|_{L^\infty (\R^{1+d})} \sup_{y\in\R^{1+d}} \big|
    (1+y^2)^m \varphi(y)\big| \int_{\R^{1+d}} \frac{dy}{(1+y^2)^m},
  \end{equation*}
  and, if $m > 1/2$ is an integer,
  \begin{equation*}
    \sup_{y\in\R^{1+d}} \big| (1+y^2)^m \varphi(y)\big| \leq 
    C_m \|\varphi\|_{2m},
  \end{equation*}
  On the other hand,
  \begin{equation*}
    |u(t,x)| \leq \sup_{t \in \R} \big|(1+t^2)^\mu v(t,x)\big|
    \int_{-\infty}^{+\infty} \frac{ds}{(1+s^2)^\mu}, \quad
    \mu > 1/2,
  \end{equation*}
  and thus
  \begin{equation*}
    \| u \|_{L^\infty (\R^{1+d})} \leq
    (\int_{-\infty}^{+\infty} \frac{ds}{(1+s^2)^\mu}) 
    \sup_{(t,x)\in \R^{1+d}} \big| (1+t^2)^\mu v(t,x)\big|.
  \end{equation*}
  Since, if we choose $\mu > 1/2$ in $\N$,
  \begin{equation*}
    \sup_{(t,x)\in \R^{1+d}} \big| (1+t^2)^\mu v(t,x)\big| \leq
    \sup_{(t,x)\in \R^{1+d}} \big| (1+t^2 + x^2)^\mu v(t,x)\big| 
    \leq C_{\mu} \|v\|_{2\mu},
  \end{equation*}
  we have
  \begin{equation*}
    \|u\|_{L^\infty(\R^{1+d})} \leq \tilde{C}_\mu \|v\|_{2\mu}, \quad \text{and} \quad
    |\langle u, \varphi \rangle| \leq K_{m,\mu} \|v\|_{2\mu} \|\varphi\|_{2m},
  \end{equation*}
  for $m > (1+d)/2$ and $\mu > 1/2$, both integers.
\end{proof}

\begin{proof}[Proof of proposition~\ref{th:characterization}]
  If $u$ is a solution in $\Schwartz^\prime$ of the damped equation, its
  Fourier transform satisfies
  \begin{equation*}
    -i (\omega + i \nu)\hat{u}^\nu = \hat{v}.
  \end{equation*}
  For $\nu >0$, this has one and only one solution
  \begin{equation*}
    \hat{u}^\nu (\omega,k) = i \frac{\hat{v}(\omega, k) }{\omega + i \nu},
  \end{equation*}
  and we have $\hat{u}^\nu \in \Schwartz(\R^{1+d})$ since $(\omega + i \nu)^{-n}$
  is smooth and polynomially bounded for $\omega \in \R$ and for all integers 
  $n >0$. Hence, its inverse Fourier transform belongs to $\Schwartz(\R^{1+d})$.
  We recall that
  \begin{equation*}
    u^\nu(-t, -x) = (2\pi)^{-(1+d)} \hat{\hat{u}}^\nu(t,x),
  \end{equation*}
  so that, if  $\check{\varphi}(t,x) = \varphi(-t,-x)$,
  \begin{equation*}
    \langle u^\nu, \varphi \rangle = \langle (2\pi)^{-(1+d)}
    \hat{\hat{u}}^\nu, \check{\varphi} \rangle
    = \langle \hat{u}^\nu, (2\pi)^{-(1+d)} 
    \hat{ \check{\varphi}} \rangle. 
  \end{equation*}
  Let us introduce, for any $\varphi \in \mathcal{S}(\R^{1+d})$,   
  the function $\hat{\psi} \in \Schwartz(\R)$ given by
  \begin{equation*}
    \hat{\psi} (\omega) = (2\pi)^{-(1+d)}
    \int_{\R^d} \hat{v}(\omega, k) \hat{\varphi} (-\omega, -k) dk.
  \end{equation*}
  We deduce
  \begin{equation*}
    \langle u^\nu, \varphi \rangle = 
    \int_\R \frac{i\hat{\psi}(\omega)}{\omega + i \nu} d\omega.
  \end{equation*}
  However, the sequence $\{u^\nu\}_{\nu \in\R_+}$ is not bounded in
  $\mathcal{S}$. In order to take the limit, we use the identity
  \begin{equation*}
    \frac{i}{\omega + i \nu} = \int_0^{+\infty} e^{i(\omega + i \nu)t} dt,
  \end{equation*}
  and note that the function 
  $(t,\omega) \mapsto e^{i(\omega + i \nu)t} \hat{\psi}(\omega)$ belongs to
  $L^1 (\R_+ \times \R)$ so that, by Fubini's theorem,
  \begin{align*}
    \langle u^\nu, \varphi \rangle &= \int_0^{+\infty} e^{-\nu t} 
    (\int_{-\infty}^{+\infty} e^{i\omega t} \hat{\psi}(\omega) d\omega) dt \\
    &= 2\pi \int_0^{+\infty} e^{-\nu t} \psi(-t)dt.
  \end{align*}
  Also the Fourier inversion theorem gives
  \begin{align*}
    \psi(t) &= \frac{1}{2\pi} \int_{-\infty}^{+\infty} e^{-i\omega t} 
    \hat{\psi }(\omega) d\omega \\
    &= \frac{1}{2\pi} \int_{-\infty}^{+\infty} e^{-i\omega t} 
    \int_{\R^d} \hat{v}(\omega, k) (2\pi)^{-(1+d)} \hat{\varphi}
    (-\omega, -k) dk d\omega \\
    &= \frac{1}{(2\pi)^{d+2}} \int_{-\infty}^{+\infty} e^{-i\omega t}
    \int_{\R^d}
    \int_{\R^{1+d}} e^{-i(k \cdot x_1 - \omega t_1)} v(t_1,x_1) dt_1 dx_1 \\
    & \qquad\qquad\qquad\qquad\qquad \times
    \int_{\R^{1+d}} e^{+i(k \cdot x_2 - \omega t_2)} \varphi(t_2,x_2) dt_2 dx_2 
    dk d\omega \\
    &= \frac{1}{(2\pi)^2} \int_{-\infty}^{+\infty} e^{-i\omega t}
    \int_{\R^{1+d}} e^{+i \omega t_1} v(t_1,x_1) dt_1
    \int_\R e^{-i\omega t_2} \varphi(t_2,x_1) dt_2 dx_1 d\omega\\    
    &= \frac{1}{2\pi} \int_{\R^{1+d}} v(t',x_1) \varphi(t' - t, x_1) dt' dx_1,
  \end{align*}
  and this yields
  \begin{equation*}
    \langle u^\nu, \varphi \rangle = 
    \int_0^{+\infty} \int_{\R^{1+d}} e^{-\nu t''}
    v(t',x) \varphi(t' + t'', x) dt'dx dt''.
  \end{equation*}
  By the change of variables $t'' = t -s$, $t' = s$, one has
  \begin{equation*}
    \langle u^\nu, \varphi \rangle = 
    \int_{\R^{1+d}} \int_{-\infty}^t e^{-\nu (t-s)}
    v(s,x) \varphi(t, x) ds dt dx,
  \end{equation*}  
  which shows that the distribution $u^\nu$ is regular and equal to the $C^\infty$
  function 
  \begin{equation*}
    u^\nu (t,x) = \int_{-\infty}^t e^{-\nu (t-s)} v(s,x) ds,
  \end{equation*}
  as claimed. In addition, $e^{-\nu ( t-s)} \leq 1$ for $s \in (-\infty,t]$, and
  with $u$ defined in proposition~\ref{th:simple-model-causal-solution}, 
  \begin{equation*}
    \lim_{\nu \to 0^+} \langle u^\nu, \varphi \rangle = 
    \int_{\R^{1+d}} \int_{-\infty}^t
    v(s,x) \varphi(t, x) ds dt dx = \int u(t,x) \varphi(t,x) dtdx,
  \end{equation*}
  for every $\varphi \in \Schwartz(\R^{1+d})$, that is, 
  \begin{equation*}
    u^\nu \to u, \quad \text{in }
    \Schwartz^\prime(\R^{1+d}),
  \end{equation*}
  and the limit is the causal solution of proposition 
  \ref{th:simple-model-causal-solution}.
\end{proof}

\section{Causal solutions of linear kinetic equations}
\label{sec:causal-solutions}

Let $\Omega$ be a domain in $\R^d$ and $\varphi \in C^\infty( \R \times \Omega)$.
We denote $\varphi_t = \varphi(t,\cdot)$ and assume that
\begin{itemize}
\item[(i)] $\varphi_t : \Omega \to \Omega$ for every $t \in \R$,
\item[(ii)] $\varphi_0 = \mathrm{Id}$ is the identity map on $\Omega$,
\item[(iii)] $\varphi_{t+s} = \varphi_t \circ \varphi_s$ for every $t,s\in \R$.
\item[(iv)] For any multi-index $\alpha \in \N^d$ there are constants
  $C, m \in \R$ such that
  \begin{equation*}
    \big|\partial_x^\alpha \varphi_t(x) \big| \leq C
    \big(1+t^2 + |\varphi_t(x)|^2 \big)^m.
  \end{equation*}
\end{itemize}
\begin{remark}
  \label{rem:flows-cond-iv}
  In condition (iv) the case $\alpha = 0$ is excluded because, for $\alpha=0$,
  $|\varphi_t(x)| \leq \big(1+t^2 + |\varphi_t(x)|^2 \big)^{1/2}$ for any map 
  $\varphi$, and the inequality in (iv) is trivially verified.
\end{remark}

As a consequence of properties (i)-(iii), $\{\varphi_t :\; t \in \R\}$ is a
one-parameter Abelian group of diffeomorphisms of $\Omega$. Particularly,
$\varphi_t^{-1} = \varphi_{-t}$. We can associate to $\varphi_t$ the  autonomous
vector field $X \in C^\infty(\Omega, \R^d)$ defined by 
\begin{equation}
  \label{eq:generator}
  X(x) = \frac{d\varphi_t(x)}{dt} \bigg|_{t=0},
\end{equation}
and, with $x = \varphi_t(x_0)$  for every $x_0 \in \Omega$, we have
\begin{equation*}
  X\big(\varphi_t(x_0)\big) = \frac{d\varphi_s }{ds}
  \big(\varphi_t(x_0)\big)\bigg|_{s=0} =
  \frac{d}{ds} \varphi_{t+s}(x_0)\bigg|_{s=0} = \frac{d\varphi_t(x_0)}{dt},
\end{equation*}
which shows that the orbit $x(t) = \varphi_t(x_0)$ of the group solves the
Cauchy problem 
\begin{equation*}
  \frac{dx}{dt} = X(x), \qquad x(0) = x_0,
\end{equation*}
globally in time and for every initial point $x_0 \in \Omega$. While conditions
(i)-(iii) are standard properties of flows of autonomous vector fields,
condition (iv) is specific to the application considered here. An example of
flow satisfying condition (iv) is
\begin{equation}
  \label{eq:flow}
  \varphi_t(x) = A(t)x,
\end{equation}
where $A(t)$ is a $d \times d$ matrix satisfying suitable conditions.

\begin{proposition}
  \label{th:example-flow}
  If the matrix-valued function $A \in C^\infty(\R,\R^{d\times d})$ is such
  that, $A(0) = I$, $A(t+s) = A(t) A(s)$ and, in the norm induced by the
  standard Euclidean norm in $\R^d$,    
  \begin{equation*}
    \|A(t)\| \leq C (1+t^2)^m,
  \end{equation*}
  for given constants $C,m>0$, then the map (\ref{eq:flow}) satisfies conditions
  (i)-(iv) above with $\Omega = \R^d$.
\end{proposition}

\begin{proof}
  Condition (i) is true because of definition (\ref{eq:flow}), while (ii) and (iii)
  follow  directly from the assumptions. As for condition (iv), since the flow
  is linear in $x$ the spatial derivative vanish for $|\alpha| \geq 2$.
  For $|\alpha| = 1$, condition (iv) is implied by the polynomial growth of $A$. 
\end{proof}

\begin{remark}
  \label{rem:flows}
  The characteristics flow of the kinetic equation considered in
  section~\ref{sec:Landau} and the non-relativistic version of the one in
  section~\ref{sec:uniform} (i.e., equation~(\ref{eq:magnetized-characteristics})
  with $\gamma=1$) are both special cases of~(\ref{eq:flow}). The relativistic
  flow~(\ref{eq:magnetized-characteristics}) is not of the same form since
  $\Omega_c$ depends on momentum. However, it satisfies assumptions (i)-(iv). 
\end{remark}

In this section we consider the linear advection equation
\begin{equation}
  \label{eq:linear-adv-eq}
  \partial_t f + X \cdot \nabla f = g, \quad \text{in $\R \times \Omega$}, 
\end{equation}
with given source $g \in \Schwartz(\R^{1+d})$. By construction the orbits
of the group are the characteristics curves of (\ref{eq:linear-adv-eq}) and
exist globally in time.

We define the function
\begin{equation}
  \label{eq:causal}
  f(t,x) \coloneqq \int_{-\infty}^t g\big(s, \varphi_{t-s}^{-1}(x)\big)ds, \qquad
  (t,x) \in \R \times \Omega,
\end{equation}
for which we prove the following.

\begin{proposition}
  \label{th:causal}
  Let $\varphi$ be a map satisfying properties (i)-(iv) above.
  Then, for any $g \in \Schwartz(\R^{1+d})$ the function $f$ defined in
  (\ref{eq:causal}) belongs to $C^\infty_b(\R \times \Omega)$ and is a classical
  solution of the linear advection equation (\ref{eq:linear-adv-eq}).
\end{proposition}

\begin{proof}
  First we observe that, for every $t \in \R$ the function
  $s \mapsto (1+s^2) / (1+ (t-s)^2)$ is in $C^\infty(\R)$, strictly positive,
  tends to $1$ for $s \to \pm \infty$, and for $t \not = 0$ has two critical
  points at $s = s_\pm = (t \pm (t^2+4)^{1/2})/2$ that correspond to a local
  minimum and a local maximum depending of the sign of $t$. The local maximum,
  in particular, is also the global maximum and thus
  \begin{equation*}
    (1+s^2) \leq C_t (1 + (t-s)^2),
  \end{equation*}
  uniformly in $s$. The constant $C_t$ is the value of the function at the
  maximum which is $C_t = \big[4 + t^2 + |t| (4 + t^2)^{1/2}\big]/\big[4 + t^2 -
    |t| (4 + t^2)^{1/2}\big]$.
  The trivial case $t=0$ is included with $C_t=1$. The substitutions
  $s \to s/a$ and $t \to t/a$ yield $(a^2+s^2) \leq C_{t/a} (a^2 + (t-s)^2)$,
  for all $a \not=0$. Explicitly, $C_{t/a} = (\xi + |t|)/(\xi-|t|)$ where
  $\xi = \sqrt{4a^2 + t^2}$; this is a monotonically decreasing function of
  $\xi$ in the interval $[\sqrt{4+t^2},+\infty)$ corresponding to $a\geq 1$, and
  the maximum is exactly $C_t$. Therefore, for every $a \geq 1$,
  \begin{equation}
    \label{eq:simp-est}
    (a^2+s^2) \leq C_t (a^2 + (t-s)^2),
  \end{equation}
  uniformly in $s \in\R$ and the constant is independent of $a$.
    
  We now apply this inequality to the function (\ref{eq:causal}). After the
  change of variable $s' = t-s$, we have 
  \begin{equation*}
    f(t,x) = \int_0^{+\infty} g\big(t-s', \varphi_{s'}^{-1}(x) \big) ds'.
  \end{equation*}
  We shall show that all derivatives
  $\partial_t^n \partial_x^\alpha \big[g\big(t-s', \varphi_{s'}^{-1}(x) \big)\big]$ 
  are uniformly bounded by $L^1$-functions of $s$. If this is the case, repeated
  use of the dominated convergence theorem gives $f \in C^\infty(\R \times \Omega)$ 
  with all derivative bounded as claimed.

  In the case $\alpha=0$, for any real $k \geq 0$ and for any integer $n\geq 0$,
  we have 
  \begin{align*}
    \big|\partial_t^n g\big(t-s, \varphi_s^{-1}(x)\big)\big| &=
    \frac{(1+(t-s)^2 + |\varphi_s^{-1}(x)|^2)^k
      \big|\partial_t^n g\big(t-s, \varphi_s^{-1}(x)\big)\big|}{(1+(t-s)^2 +
      |\varphi_s^{-1}(x)|^2)^k} \\ &\leq
    \frac{(1+(t-s)^2 + |\varphi_s^{-1}(x)|^2)^k
      \big|\partial_t^n g\big(t-s, \varphi_s^{-1}(x)\big)\big|}{(1+(t-s)^2)^k}
  \end{align*}
  and since $g \in \Schwartz$ inequality~(\ref{eq:polynomial-bound}) gives
  \begin{equation*}
    \big|\partial_t^n g\big(t-s, \varphi_s^{-1}(x)\big)\big|
    \leq \frac{C}{(1+(t-s)^2)^k},
  \end{equation*}
  with constant $C$ independent on $s$ and $s \mapsto (1+(t-s)^2)^{-k}$ is in
  $L^1$ for $k > 1/2$. With $n=0$, this also shows that $f(t,x)$ is bounded.

  Including spatial derivative requires condition (iv). In the
  case $|\alpha| = 1$,
  \begin{align*}
    \partial_t^n \partial_x^\alpha \big[g\big(t-s, \varphi_{s}^{-1}(x)
      \big)\big] &= \partial_x^\alpha \varphi_{s}^{-1}(x) \cdot
    (\partial_t^n \nabla g)\big(t-s, \varphi_{s}^{-1}(x) \big) \\
    & \leq d \cdot |\partial_x^\alpha \varphi_{-s}(x) |
    \cdot
    \max_j \big|(\partial_t^n \partial_{x_j} g)\big(t-s, \varphi_{s}^{-1}(x)
    \big)\big|.
  \end{align*}
  Hypothesis (iv) and the estimate~(\ref{eq:simp-est}) give, for every $k>0$,
  \begin{align*}
    |\partial_x^\alpha \varphi_{-s}(x) |
    & \leq C (1 + s^2 + |\varphi_s^{-1}(x)|^2)^m \\
    & \leq \tilde{C}_t (1 + (t-s)^2 + |\varphi_s^{-1}(x)|^2)^m \\
    & \leq \tilde{C}_t \frac{(1 + (t-s)^2 + |\varphi_s^{-1}(x)|^2)^{m+k}}
    {(1 + (t-s)^2)^k}.
  \end{align*}
  Proceeding as before, we choose $k > 1/2$ and obtain that the first-order
  derivatives are uniformly bounded by an $L^1$-function.

  The case of general $\alpha$ is complicated by the form of the chain
  rule. However, the same argument can be applied to the explicit formula for
  the multi-variate chain rule (Fa\`a di Bruno formula) which is linear in the
  derivatives of $g$ and polynomial in the derivatives of $\varphi$.

  The fact that $f$ is a classical solution of~(\ref{eq:linear-adv-eq}) can be
  checked by substitution.
\end{proof}

If the flow is polynomially bounded, we can deduce that $f$, viewed as a
distribution on $\Omega$, has finite moments at all orders. This is a
consequence of the following result.

\begin{lemma}
  \label{th:finite-moments}
  Let $\varphi$ satisfy assumptions (i)-(iv) and $g \in \Schwartz(\R^{1+d})$. In
  addition let us assume the there are $n \in\N_0$ and $C > 0$ such that
  $|\varphi_t(x)| \leq C (1 + x^2)^n$, uniformly in $\R \times \Omega$. Then,
  $p \cdot \partial_t^l \partial_x^\alpha f(t,\cdot) \in L^\infty(\Omega)$ for
  every polynomial $p = p(x)$, $l \in N_0$ and $\alpha \in \N_0^d$.
\end{lemma}

\begin{proof}
  It is enough to show that $|x|^k \partial_t^l \partial_x^\alpha f(t,x)$ is in
  $L^\infty(\Omega)$. Since $x = \varphi_s\big(\varphi_s^{-1}(x)\big)$,
  \begin{equation*}
    |x|^k \leq C^k(1+ |\varphi_s^{-1}(x)|^2)^{kn} \leq
    C^k(1+ (t-s)^2 + |\varphi_s^{-1}(x)|^2)^{kn},
  \end{equation*}
  and
  \begin{align*}
    |x|^k |f(t,x)| &\leq C^k \int_0^{+\infty} (1+(t-s)^2+|\varphi_s^{-1}(x)|^2)^{kn}
    \partial_t^l \partial_x^\alpha \big[g\big(t-s, \varphi_{s}^{-1}(x) \big)\big] ds \\
    &\leq C_{k,l,m,\alpha} \int_0^{+\infty} \frac{ds}{(1+(t-s)^2)^m},
  \end{align*}
  where the derivatives of $g$ are estimates as in proposition~\ref{th:causal}.
\end{proof}

The function~(\ref{eq:causal}) is referred to as the causal solution of
equation~(\ref{eq:linear-adv-eq}).

\section{The Hilbert transform and its action on symbols} 
\label{sec:calc-limits-proof}

The Hilbert transform is defined by
\begin{equation*}
  \mathcal{H}(\phi)(x) = \frac{1}{\pi} \pv \int \frac{\phi(y)}{x - y} dy = 
  \frac{1}{\pi} (\pv \frac{1}{x}) * \phi (x),
\end{equation*}
for a function $\phi \in C^\infty_0$. It has the following properties, which we
state without proof.   

\begin{proposition}
  \label{th:Hilbert-transform}
  The Hilbert transform $\mathcal{H}$ defined above is extended to functions
  in $\Schwartz(\R)$ through the equality
  \begin{equation*}
    \mathcal{H}(\phi)(x) = \frac{1}{\pi} \int_\R \frac{1}{2u} \big[
      \phi(x-u) - \phi(x+u) \big]du.
  \end{equation*}
  Moreover, it extends to an isometry of Sobolev spaces $H^k(\R)$ through the
  equality 
  \begin{equation*}
    \widehat{\mathcal{H}(u)} = -i \sgn(\xi) \hat{u},
    \qquad \forall u \in H^k(\R),
  \end{equation*}
  for every non-negative integer $k$. Particularly if $u \in \Schwartz(\R)$, 
  then $\mathcal{H}(u) \in H^{\infty}(\R)$.
\end{proposition}

With $\nu > 0$, $G \in \Schwartz(\R)$, and $\phi \in \Schwartz(\R^2)$ let us
consider the integrals
\begin{equation*}
  A_\nu(\omega,k) \coloneqq \int_\R \frac{G(v)}{\omega - kv + i\nu} dv,\quad
  B_\nu(v,k) \coloneqq \int_\R \frac{\phi(\omega,k)}{\omega - kv + i\nu} d\omega,
\end{equation*}
that involve the same denominator as in~(\ref{sigma-nu}).

\begin{lemma}
  \label{th:Anu-Bnu}
  Let $G \in \Schwartz(\R)$, $\phi \in \Schwartz(\R^2)$, and $\nu >0$.
  \begin{itemize}
  \item[(i)] There are constants $C_{A,j}$, $j=0,1$, depending on $G$, such that
    for $k \not = 0$,
    \begin{align*}
      \big|A_\nu(\omega,k)\big| &\leq C_{A,0} + C_{A,1}/|k|, \\
      A_\nu(\omega,k) &\xrightarrow{\nu \to 0^+}
      \frac{1}{k}\Big[\pi \mathcal{H}(G)(\omega/k) - i\pi G(\omega/k)\Big].
    \end{align*}
  \item[(ii)] For any integer $m \geq 0$ there exists a constant
    $C_B$, depending on $\phi$ and $m$, such that
    \begin{align*}
      \big|B_\nu(v,k)\big| &\leq C_{B} (1+k^2)^{-m}, \\
      B_\nu(v,k) &\xrightarrow{\nu \to 0^+}
      -\pi \mathcal{H}\big(\phi(\cdot,k)\big)(kv) - i\pi \phi(kv,k).
    \end{align*}
  \end{itemize}
\end{lemma}
\begin{proof}
  Let $G \in \mathcal{S}(\R)$ and $\phi \in \mathcal{S}(\R^2)$.
  For $k\not=0$, define the two functions $G_0$ and $G_1$ by  
  \begin{align*}
    G_0\big( \tfrac{\omega}{k},u \big) &\coloneqq \tfrac{1}{2} \big[
      G \big(\tfrac{\omega}{k} - u \big) +
      G \big(\tfrac{\omega}{k} + u \big) \big], \\
    u G_1 \big (\tfrac{\omega}{k},u \big) &\coloneqq \tfrac{1}{2} \big[
      G \big(\tfrac{\omega}{k} - u \big) -
      G \big(\tfrac{\omega}{k} + u \big) \big].
  \end{align*}
  Similarly, for all $(\omega,k,v)\in \R^3$ we define
  \begin{align*}
    \phi_0(kv, \varpi, k) &\coloneqq
    \frac{1}{2} (\phi(kv + \varpi, k)+\phi(kv-\varpi,k)), \\
    \varpi \phi_1(kv, \varpi,k) &\coloneqq
    \frac{1}{2} \Big[\phi(kv + \varpi, k)-\phi(kv - \varpi,k)\Big].
  \end{align*}
  We observe that,
  \begin{equation}
    \label{eq:G0-phi0}
    \big|G_0(\tfrac{\omega}{k},u)\big| \leq \|G\|_0, \qquad
    (1 + k^2)^m \big|\phi_0(kv,\varpi,k)\big| \leq \|\phi\|_{2m}.
  \end{equation}
  The first inequality follows directly from the definition of $G_0$, and we have
  \begin{multline*}
    (1 + k^2)^m \big|\phi_0(kv,\varpi,k)\big| \leq \frac{1}{2} \big[
      (1 + (kv+\varpi)^2 + k^2)^m \big|\phi(kv+\varpi,k)\big| \\
      + (1 + (kv-\varpi)^2 + k^2)^m \big|\phi(kv-\varpi,k)\big|\big]
    \leq \|\phi\|_{2m}.
  \end{multline*}
  Moreover,
  \begin{equation*}
    \big|G_1(\tfrac{\omega}{k},u)\big| \leq \|G\|_1, \qquad
    (1+k^2)^m \big|\phi_1(kv,\varpi,k)\big| \leq \|\phi\|_{1+2m},
  \end{equation*}
  and this can be proven by Taylor's formula and
  \begin{align*}
    \big|G_1(\tfrac{\omega}{k},u)\big| &\leq \frac{1}{2} \int_{-1}^{+1}
    \big|G'(\tfrac{\omega}{k} + \lambda u)\big| d\lambda \leq \sup |G'(v)|, \\
    (1 + k^2)^m \big|\phi_1(kv,\varpi,k)\big| &\leq \frac{1}{2} \int_{-1}^{+1}
    (1 + k^2)^m \big|\partial_\omega \phi(kv + \lambda \varpi,k)\big|d\lambda \\
    &\leq \sup \big[(1 + (kv + \lambda \varpi)^2 + k^2)^m
      \big|\partial_\omega \phi(kv + \lambda \varpi,k)\big|\big].
  \end{align*}
  When $|u| > 1$ we also have, for any $M \in \N_0$,
  \begin{align*}
    \big|G_1(\tfrac{\omega}{k},u)\big| &\leq
    \frac{1}{2}\Big[\big|G(\tfrac{\omega}{k} - u)\big| +
      \big|G(\tfrac{\omega}{k} + u)\big|\Big] \\ &\leq
    \frac{1}{2} \big[
      \tfrac{1}{(1+(\tfrac{\omega}{k} - u)^2)^M} +
      \tfrac{1}{(1+(\tfrac{\omega}{k} + u)^2)^M} \big]
    \sup \big[(1+v^2)^M |G(v)|\big],
  \end{align*}
  and, analogously, when $|\varpi| > 1$,
  \begin{multline*}
    (1+k^2)^m \big|\phi_1(kv,\varpi,k)\big| \leq
    \frac{1}{2} \big[\tfrac{1}{(1+(kv+\varpi)^2)^M} +
      \tfrac{1}{(1+(kv-\varpi)^2)^M}\big] \\
    \times  \sup \big[(1+ \omega^2 + k^2)^{M+m} \big|\phi(\omega,k)\big|\big].
  \end{multline*}
  Therefore,
  \begin{equation}
    \label{eq:G1-phi1}
    \big|G_1(\tfrac{\omega}{k},u)\big| \leq G_1^*(\tfrac{\omega}{k},u), \qquad
    (1+k^2)^m \big|\phi_1(kv,\varpi,k)\big| \leq \phi_1^*(kv,\varpi),
  \end{equation}
  where
  \begin{equation*}
    G_1^*(\tfrac{\omega}{k},u) \coloneqq
    \begin{cases}
      \|G\|_1, & |u|\leq 1, \\
      \frac{1}{2}\big[
        \tfrac{1}{(1+(\tfrac{\omega}{k} - u)^2)^M} +
        \tfrac{1}{(1+(\tfrac{\omega}{k} + u)^2)^M}\big] \|G\|_{2M}, & |u| > 1.
    \end{cases}
  \end{equation*}
  and
  \begin{equation*}
    \phi_1^*(kv,\varpi) \coloneqq
    \begin{cases}
      \|\phi\|_{1+2m}, & |\varpi| \leq 1, \\
      \frac{1}{2}\big[
        \tfrac{1}{(1+(kv - \varpi)^2)^M} +
        \tfrac{1}{(1+(kv + \varpi)^2)^M}\big]\|\phi\|_{2(M+m)}, & |\varpi| > 1.
    \end{cases}
  \end{equation*}  
  For $M \geq 1$, one has
  $G_1^*(\tfrac{\omega}{k},\cdot), \phi_1^*(kv,\cdot) \in L^1(\R)$ with
  \begin{equation}
    \label{eq:G1-phi1-L1norms}
    \begin{aligned}
      \|G_1^*(\tfrac{\omega}{k},\cdot)\|_{L^1} &\leq 2 \|G\|_1 +
      \|G\|_{2M} \int_\R \frac{dt}{(1+t^2)^M} \eqqcolon C_{M,G}, \\
      \|\phi_1^*(kv,\cdot)\|_{L^1} &\leq 2 \|\phi\|_{1+2m} + \|\phi\|_{2(M+m)}
      \int_\R \frac{dt}{(1+t^2)^M} \eqqcolon C_{M,m,\phi},
    \end{aligned}
  \end{equation}
  where $C_{M,G}$ is a constant depending only on the Schwartz semi-norms of $G$ and
  the integer $M$ and $C_{M,m,\phi}$ depends only on $M$, $m$, and the Schwartz
  semi-norms of $\phi$. 
  
  Then, one has the identities
  \begin{align*}
    A_{\nu}(\omega,k) &=
    \int_{\R} \frac{G_0(\frac{\omega}{k}, u) + 
      uG_1(\frac{\omega}{k}, u)}{ku+i\nu}du, \\
    B_{\nu}(v,k) &= \int_{\R}\frac{\phi_0(kv, \varpi, k) + 
    \varpi \phi_1(kv, \varpi, k)}{\varpi + i\nu} d\varpi.
  \end{align*}
  As $G_0,G_1$ are even functions in $u$, and $\phi_0,\phi_1$ are even
  functions in $\varpi$, one deduces 
  \begin{align*}
    kA_{\nu}(\omega,k) &= \int_{\R} \frac{k^2u^2G_1(\frac{\omega}{k}, u) - 
      ik\nu G_0(\frac{\omega}{k}, u)}{k^2u^2+\nu^2}du \\
    &= \int_{\R} \frac{k^2u^2G_1(\frac{\omega}{k}, u)}{k^2u^2+\nu^2}du - 
    ik\int_{\R} \frac{G_0(\frac{\omega}{k}, \nu t)}{k^2t^2+1}dt, \\
    B_{\nu}(v,k) &= \int_{\R}\frac{\varpi^2 \phi_1(kv, \varpi, k) -
      i\nu \phi_0(kv, \varpi, k)}{\varpi ^2+ \nu^2}d\varpi \\
    &=\int_{\R}\frac{\varpi^2 \phi_1(kv, \varpi, k)}{\varpi ^2+ \nu^2}d\varpi 
    -i\int_{\R}\frac{\phi_0(kv, \nu t, k)}{t^2+1}dt.
  \end{align*}
  We observe that $\frac{k^2u^2}{k^2u^2+\nu^2}$ and 
  $\frac{\varpi^2 }{\varpi ^2+ \nu^2}$ are uniformly bounded by $1$. 
  Moreover,  $t \mapsto \frac{1}{t^2+1}$ is in $L^1(\R)$, and, for $k\not=0$,
  $t\mapsto \frac{1}{k^2t^2+1}$ belongs to $L^1(\R)$, with the values
  $\int_{\R}\frac{dt}{t^2+1}=\pi$ and
  $\int_{\R}\frac{dt}{k^2t^2+1}=\frac{\pi}{k}$.

  Then estimates (\ref{eq:G0-phi0}), (\ref{eq:G1-phi1}), and
  (\ref{eq:G1-phi1-L1norms}) give
  \begin{align*}
    \big|k A_\nu(\omega,k)\big| &\leq \|G_1^*(\tfrac{\omega}{k},\cdot)\|_{L^1} +
    \pi |k| \|G\|_0 \leq C_{M,G} + \pi |k| \|G\|_0, \\
    (1+k^2)^m \big|B_\nu(v,k)\big| &\leq \|\phi_1^*(kv,\cdot)\|_{L^1} + \pi
    \|\phi\|_{2m} \leq C_{M,m,\phi} + \pi \|\phi\|_{2m}.
  \end{align*}
  
  As for the limits of $A_\nu$ and $B_\nu$ as $\nu \to 0^+$,
  estimate (\ref{eq:G0-phi0}) and (\ref{eq:G1-phi1}) with $M \geq 1$ imply
  that the integrands are bounded by an $L^1$ function uniformly in $\nu$.  
  Then the dominated convergence theorem allows us to conclude that,
  as $\nu \to 0^+$,
  \begin{align*}
    kA_{\nu}(\omega, k) &\rightarrow \int_{\R} G_1(\frac{\omega}{k}, u)du
    -i\pi G_0(\frac{\omega}{k}, 0), \\
    B_{\nu}(v,k) &\rightarrow  \int_{\R}\phi_1(kv, \varpi, k)d\varpi  -
    i\pi \phi_0(kv,0, k).
  \end{align*}
  Upon accounting for the definitions of $G_0, G_1$ and $\phi_0$, $\phi_1$ that
  reads
  \begin{align*}
    kA_{\nu}(\omega,k) &\rightarrow \pi \mathcal{H}\big(G)(\omega/k)
    -i\pi G_0(\omega/k, 0), \\
    B_{\nu}(v,k) &\rightarrow  
    - \pi \mathcal{H}\big(\phi(\cdot,k) \big)(kv) - i\pi \phi(kv, k),
  \end{align*}
  as claimed.
\end{proof}

\section{A useful linear algebra result}
\label{sec:matrices}

In this appendix we establish a linear algebra result which constitutes a key
step in the proof of the results in section~\ref{sec:auxiliary-pde}.

\begin{lemma}
  \label{th:matrix-Ak}
  For any integer $\ell \geq 0$, the matrices
  \begin{equation*}
    A_\ell = \begin{pmatrix}
      0           &   i            &  0             & 0        & \dots  \\
      - i \ell    &   0            &  2i            & 0        & \dots  \\
      0           & - i(\ell-1)    &  0             & 3 i      & \dots  \\
      0           &   0            & -i(\ell-2)  & 0        & \dots  \\
      \vdots      & \vdots         & \vdots         & \vdots   & \ddots
    \end{pmatrix} \in \C^{(\ell+1)\times(\ell+1)}
  \end{equation*}
  are diagonalizable with integer eigenvalues
  $\{(2s -\ell) :\; s = 0,1,\ldots, \ell\}$.
  For $z \in \C\setminus\Z$, $A_\ell - z$ is invertible and
  \begin{equation*}
    |(A_\ell - z)^{-1}|_1 \leq (\ell+1)^2 2^\ell/\delta_z,
  \end{equation*}
  where the norm $|\cdot|_1$ is induced by the $L^1$-norm on $\C^{\ell+1}$ and
  $\delta_z = \min \{|z - m|: m\in\Z\}$.
\end{lemma}

\begin{proof}
  The map $i_1 : \C^{\ell+1} \to C^\infty(\T)$, $\T = \R/(2\pi \N)$,  defined by
  \begin{equation*}
    x \mapsto U(\phi) = \sum_{r=0}^\ell x_r \big(\cos\phi\big)^{\ell-r}
    \big(\sin\phi)^r,
  \end{equation*}
  is an linear embedding of $\C^{\ell+1}$ and becomes an isomorphism when
  restricted to its range. Injectivity in particular holds since,
  if $x \in \C^{\ell+1}$ is such that $U = i_1(x) = 0$, then for any $R\geq 0$
  \begin{equation*}
    R^\ell U(\phi) = \sum_{r=0}^\ell x_r u_1^{\ell-r} u_2^r = 0,
  \end{equation*}
  uniformly for $(u_1, u_2) = (R \cos\phi, R \sin\phi) \in \R^2$.
  Since monomials $u_1^{\ell-r}u_2^r$ are linearly independent, we deduce
  $x=0$ and thus $i_1$ is injective.

  A second embedding is $i_2 : \C^{\ell+1} \to C^\infty(\T)$ with
  \begin{equation*}
    v \mapsto V(\phi) = \sum_{s=0}^{\ell} v_s e^{i(\ell-2s)\phi},
  \end{equation*}
  which is injective since the exponential are linearly independent.

  We claim that $i_1, i_2$ have the same range i.e.,
  $V_\ell \coloneqq i_1(\C^{\ell+1}) = i_2(\C^{\ell+1})$. In fact,
  $i_1(\C^{\ell+1})$ is spanned by functions 
  $f_r(\phi) = \big(\cos\phi\big)^{\ell-r} \big(\sin\phi)^r$ while
  $i_2(\C^{\ell+1})$ is spanned by $g_s(\phi) = e^{i(\ell-2s)\phi}$.
  On the one hand, the binomial formula gives
  \begin{align*}
    f_r(\phi) &= \frac{(-i)^r}{2^\ell} \big(e^{i\phi} + e^{-i\phi}\big)^{\ell-r}
    \big(e^{i\phi} - e^{-i\phi}\big)^r \\
    &= \frac{(-i)^r}{2^\ell} \sum_{m=0}^{\ell-r} \sum_{n=0}^{r} (-1)^n
    \binom{\ell-r}{m} \binom{r}{n} e^{i(\ell - 2(m+n))\phi},
  \end{align*}
  and thus
  \begin{equation*}
    f_r(\phi) = \sum_{s=0}^{\ell} T_{sr} g_s(\phi), \qquad
    T_{sr} = \frac{(-i)^r}{2^\ell}
    \sum_{(m,n)\in\Sigma(r,s)} (-1)^n \binom{\ell-r}{m} \binom{r}{n},
  \end{equation*}
  where the last sum is over the set of indices
  \begin{equation*}
    \Sigma(r,s) = \{(m,n) :\; m=0,\ldots,\ell-r,\;
    n = 0,\ldots,r,\;m+n=s\}.
  \end{equation*}
  On the other hand,
  \begin{align*}
    g_s(\phi) &= e^{i(\ell - 2s)\phi} = \big(\cos\phi + i\sin\phi\big)^{\ell-s}
    \big(\cos\phi-i\sin\phi\big)^s \\
    &= \sum_{m=0}^{\ell-s} \sum_{n=0}^{s} i^{m-n}
    \binom{\ell-s}{m} \binom{s}{n} \big(\cos\phi)^{\ell - (m+n)}
    \big(\sin\phi\big)^{m+n},
  \end{align*}  
  or
  \begin{equation*}
    g_s(\phi) = \sum_{r=0}^{\ell} S_{rs} f_r(\phi), \qquad
    S_{rs} =
    \sum_{(m,n) \in \Sigma(s,r)} i^{m-n} \binom{\ell-r}{m} \binom{r}{n}.
  \end{equation*}
  In summary, we have obtained that
  \begin{equation*}
    f_r = \sum_{s=0}^{\ell} T_{sr} g_s, \qquad
    g_s = \sum_{r=0}^{\ell} S_{rs} f_r, 
  \end{equation*}
  hence for any $\ell$, $\{f_r\}_{r=0}^\ell$ and $\{g_s\}_{s=0}^\ell$ span the
  same linear space $V_\ell$.
  The matrices $(T_{sr})$ and $(S_{sr})$ correspond to the linear operators 
  \begin{equation*}
    T = i_2^{-1} \circ i_1 : \C^{\ell+1} \to \C^{\ell+1}, \quad
    S = i_1^{-1} \circ i_2 : \C^{\ell+1} \to \C^{\ell+1}.
  \end{equation*}
  At last, let us introduce the differential operator
  \begin{equation*}
    B = i \frac{d}{d\phi} : C^\infty(\T) \to C^\infty(\T),
  \end{equation*}
  for which a direct calculation shows that
  \begin{equation*}
    B \circ i_1(x) = \sum_{r=0}^\ell (A_\ell x)_r f_r(\phi).
  \end{equation*}
  This means that $B$ can be restricted to operators from $V_\ell \to V_\ell$ and
  \begin{equation}
    \label{eq:Ak-representation}
    A_\ell = i_1^{-1} \circ B \circ i_1.
  \end{equation}
  On the other hand
  \begin{equation}
    \label{eq:Bpm-spectral}
    B \circ i_2 (v) = \sum_{s=0}^\ell (2s-\ell) v_s
    g_s(\phi),
  \end{equation}
  that is, the operator $i_2^{-1} \circ B \circ i_2$ is diagonal with
  eigenvalues $(2s-\ell)$.
  
  For equation~(\ref{eq:Ak-representation}) it follows that
  \begin{equation*}
    A_\ell = i_1^{-1} \circ B \circ i_1 = S \circ (
    i_2^{-1} \circ B \circ i_2) \circ T,
  \end{equation*}
  which shows that $A_\ell$ is diagonalizable with eigenvalues 
  $a = (2s-\ell)$ for $s = 0, \ldots \ell$.
  
  Therefore if $z \in \C \setminus \Z$, then $A_\ell - z$ is invertible and
  \begin{equation*}
    (A_\ell - z)^{-1} = S \circ (i_2^{-1} \circ (B - z)^{-1} \circ i_2)
    \circ T,
  \end{equation*}
  hence
  \begin{equation}
    \label{eq:Ak-estimate}
    |(A_\ell - z)^{-1}|_1 \leq |S|_1 \cdot
    | i_2^{-1} \circ (B - z)^{-1} \circ i_2 |_1 \cdot |T|_1.
  \end{equation}
  We claim that
  \begin{equation}
    \label{eq:Bpm}
    | i_2^{-1} \circ (B - z)^{-1} \circ i_2 |_1 \leq \delta_z^{-1}, \qquad
    |T|_1 \leq (\ell+1), \qquad
    |S|_1 \leq (\ell+1) 2^\ell.
  \end{equation}
  If this is true, then inequality~(\ref{eq:Ak-estimate}) becomes
  $|(A_\ell - z)^{-1}|_1 \leq  2^\ell (\ell+1)^2 /\delta_z$, which is the
  claimed estimate. Therefore, in order to complete the proof it is sufficient
  to show that (\ref{eq:Bpm}) holds.

  From (\ref{eq:Bpm-spectral}), one deduces the expression for the inverse,
  \begin{equation*}
    (B - z)^{-1} \circ i_2(v) = \sum_{s=0}^\ell
    \frac{v_s}{2s-\ell - z} e^{i(\ell -2s) \phi},
  \end{equation*}
  and thus, in the $L^1$-norm,
  \begin{equation*}
    |i_2^{-1} \circ (B - z)^{-1} \circ i_2(v) |_1
    = \sum_{s=0}^\ell \frac{|v_s|}{|2s-\ell - z|}.
  \end{equation*}
  The definition of $\delta_z$ implies
  \begin{equation*}
    |2s-\ell - z| \geq \delta_z = \min_{m\in\Z} |z - m|,
  \end{equation*}
  so that
  \begin{equation}
    \label{eq:B-estimate}
    |i_2^{-1} \circ (B - z)^{-1} \circ i_2(v) |_1
    \leq \frac{1}{\delta_z}  \sum_{s=0}^\ell |v_s|
    = \frac{|v|_1}{\delta_z},
  \end{equation}
  or $|i_2^{-1} \circ (B - z)^{-1} \circ i_2 |_1 \leq \delta_a^{-1}$.
  As for the linear operator $T$,
  \begin{equation*}
    |T_{sr}| \leq \frac{1}{2^\ell} \sum_{m=0}^{\ell-r} \sum_{n=0}^r
    \binom{\ell-r}{m} \binom{r}{n},
  \end{equation*}
  and for the last sum we can use the identity,
  \begin{equation*}
    (a+b)^\ell = (a+b)^{\ell-r} (a+b)^r = \sum_{m=0}^{\ell-r} \sum_{n=0}^{r}
    \binom{\ell-r}{m} \binom{r}{n} a^{\ell -m-n} b^{m+n},
  \end{equation*}
  which for $a=b=1$ and together with the previous inequality implies
  $|T_{sr}| \leq 1$; then
  \begin{equation*}
    |Tx|_1 = \sum_{s=0}^\ell \Big| \sum_{r=0}^\ell T_{sr} x_s\Big| \leq
    \sum_{s=0}^\ell \sum_{r=0}^\ell |T_{sr}| |x_s| \leq (\ell+1) |x|_1,
  \end{equation*}
  or $|T|_1 \leq (\ell+1)$. Analogously for $S$,
  \begin{equation*}
    |S_{rs}| \leq \sum_{m=0}^{\ell-s} \sum_{n=0}^s
    \binom{\ell-s}{m} \binom{s}{n} = 2^\ell,
  \end{equation*}
  and $|S|_1 \leq (\ell + 1) 2^\ell$.
\end{proof}

\subsection*{Acknowledgments} This work is a result of a long collaboration
between the authors during visits under the Erasmus program and the
collaboration agreement between Universit\'e Paris 13, Sorbonne Paris Cit\'e,
and Technische Universit\"at M\"unchen.

%--- bibliography ---
\bibliographystyle{amsplain}
\bibliography{PlasmaResponse_Landau_and_cyclotron}

\end{document}